\documentclass{article}
\usepackage[english]{babel}
\usepackage{amsmath,amssymb,graphicx,latexsym,theorem}

\newcommand{\assign}{:=}
\newcommand{\comma}{{,}}
\newcommand{\mathi}{\mathrm{i}}
\newcommand{\of}{:}
\newcommand{\tmem}[1]{{\em #1\/}}
\newcommand{\tmmathbf}[1]{\ensuremath{\boldsymbol{#1}}}
\newcommand{\tmop}[1]{\ensuremath{\operatorname{#1}}}
\newcommand{\tmtextit}[1]{\text{{\itshape{#1}}}}
\newcommand{\tmxspace}{\hspace{1em}}
\newenvironment{proof}{\noindent\textbf{Proof\ }}{\hspace*{\fill}$\Box$\medskip}
\newcounter{nnacknowledgments}

{\theorembodyfont{\rmfamily}\newtheorem{acknowledgments*}[nnacknowledgments]{Acknowledgments}}
{\theorembodyfont{\rmfamily}\newtheorem{convention}{Convention}}
\newtheorem{corollary}{Corollary}
\newtheorem{definition}{Definition}
{\theorembodyfont{\rmfamily}\newtheorem{example}{Example}}
\newtheorem{lemma}{Lemma}
\newtheorem{notation}{Notation}
\newtheorem{proposition}{Proposition}
{\theorembodyfont{\rmfamily}\newtheorem{remark}{Remark}}
\newtheorem{theorem}{Theorem}
\begin{document}

\title{Quenched and Annealed CLTs for the one-periodic Aztec diamond in random
environment}

\author{Panagiotis Zografos}

\maketitle

\begin{abstract}
  We study the asymptotic behavior of random dimer coverings of the
  one-periodic Aztec diamond in random environment. We investigate quenched
  limit theorems for the height function and we extend annealed limit theorems
  that were recently studied in {\cite{B52}}. We consider more general choices
  of random edge weights (independence is not assumed) and we distinguish two
  cases where the random edge weights satisfy the Central Limit Theorem (CLT)
  under different scalings. For both cases, we prove convergence to the
  Gaussian Free Field for the quenched fluctuations. For the annealed version,
  it had been shown in {\cite{B52}}, that Gaussian Free Field fluctuations can
  be dominated by the much larger fluctuations of the random environment. To
  access quenched fluctuations we analyze the Schur process with random parameters
  in a way that allows to prove the annealed CLT for the height function for
  non i.i.d. weights. We consider specific examples where we determine the
  asymptotic fluctuations.
\end{abstract}

\section{Introduction}

\subsection{Overview}

The dimer model, and in particular random domino tilings of the Aztec diamond,
is a well studied model of 2D statistical mechanics which exhibits very rich
universal behavior (see {\cite{B89}} for an exposition). They were first
introduced in {\cite{B56}}. Since then, there have been numerous results for
uniformly random domino tilings, studying among others the limit shape
(existence of the arctic curve) {\cite{B83}}, the local behavior close to the
boundary {\cite{B84}}, the fluctuations of the domino height function
{\cite{B3}}, {\cite{B74}}. Non-uniform random tilings have also been
considered extensively in the literature, and have led to richer geometric
phenomena {\cite{B78}}, {\cite{B79}}, {\cite{B68}}, {\cite{B80}},
{\cite{B65}}, {\cite{B76}}, {\cite{B75}}, {\cite{B77}}, {\cite{B81}}.

More recently, there has been drawn attention on studying dimer models in
random environment, where the edge weights are random themselves {\cite{B52}}, {\cite{DVP}}
{\cite{B82}}. Related models had been considered in the physics literature
before {\cite{B85}}.

In this paper, we continue in this direction and further study the one-periodic
Aztec diamond in random environment. The randomness of our one-periodic
weights is significantly more general than that of {\cite{B52}}, and
independence is not assumed in general. We consider random edge weights that
satisfy a Central Limit Theorem in the sense of moments (see Definition
\ref{BBBB} for the exact conditions). Such conditions are satisfied for i.i.d.
random variables, families of Markov chains, or eigenvalues of the Gaussian
Unitary Ensemble (GUE), among others. We will call such weights $C \text{} L
\text{} T$ $a \text{} p \text{} p \text{} r \text{} o \text{} p \text{} r
\text{} i \text{} a \text{} t \text{} e$.

We study the quenched Central Limit Theorem for the height function. For the
annealed version and i.i.d. edge weights, it had been shown in {\cite{B52}}
that Gaussian Free Field fluctuations can be perturbed or even dominated by
the fluctuations of the random environment. We show that convergence to the
Gaussian Free Field is recovered for the quenched version (see Theorem
\ref{NMNm}). Gaussian Free Field type fluctuations were conjectured by
Kenyon-Okounkov {\cite{B90}}, {\cite{B91}}, {\cite{B92}}, for a variety of
dimer models, and they were also verified in several models {\cite{B95}},
{\cite{B96}}, {\cite{B93}}, {\cite{B90}}, {\cite{B91}}, {\cite{B94}}. For the
one-periodic Aztec diamond, even with random edge weights, for a given
realization of disorder the fluctuations of the height function are still
described by the Gaussian Free Field.

Furthermore, we extend the annealed Central Limit Theorems for the height
function that were recently studied in {\cite{B52}}. Our results apply to CLT
appropriate edge weights. We also analyze the annealed fluctuations, for
specific examples of edge weights, such as Markov chains or eigenvalues of GUE
(see Examples \ref{ex21}, \ref{ex22}, \ref{ex24}).

To establish our results, we analyze Schur processes with random parameters.
Our approach is different from that of {\cite{B52}}, in the sense that we do
rely on the nice properties of Schur functions (determinantal form, Cauchy
identity, etc.), but we do not employ the method of Schur generating functions
{\cite{B1}}, {\cite{B3}}, {\cite{B4}}. The main reason is that both the
quenched and the annealed Schur generating functions cannot be simplified for
non i.i.d. random weights. As a result, we develop a unified set of arguments
that hold for more general random weights, and simultaneously establish both
quenched and annealed Central Limit Theorems.

\subsection{The model}

The Aztec diamond graph is a classical bipartite graph embedded in the square
lattice. An example of specific size is illustrated in Figure \ref{Fig1}. The
size of the Aztec diamond is the number of lattice vertices along each side. A
dimer covering of this graph (also known as perfect matching) is a subset of
its edges, with elements called dimers, such that each vertex is covered by
exactly one dimer. There is also a one-to-one correspondence between domino
tilings of the Aztec diamond and dimer coverings of this graph. We assign at
each edge of our graph of size $M$, weights $\{ w_i, u_i, x_i, v_i \}_{i =
1}^M$, in a so called ``one-periodic way''. This means that the weights change
in one direction, but repeated on the other (see Figure \ref{Fig1}). Given
such weights, we consider a probability measure on the set of all dimer
coverings $D$ of the Aztec diamond graph, defined as
\[ \mathbb{P} (D) \assign \frac{1}{Z} \prod_{e \in D} W (e),  \]
where $W (e)$ is equal to one of the $w_i, u_i, x_i, v_i$. For our model, the
weights of the above probability measure are considered to be random, and we
assume that they satisfy certain properties as $M \rightarrow \infty$ (see
Definition \ref{BBBB}). Our goal is to study a random dimer covering on the
graph with these weights, as $M \rightarrow \infty$.

\begin{figure}[h]
	\centering
	\includegraphics[width=12.91cm]{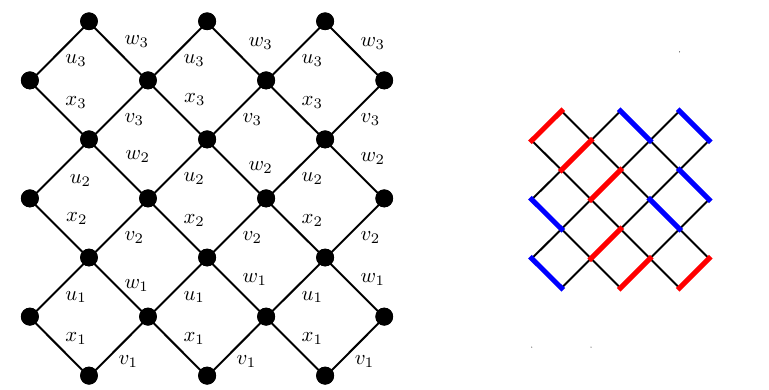}
	\caption{Left: The Aztec diamond graph of size 3 with
		one-periodic weights. Right: A dimer covering of the graph. The product of
		the edge weights for the specific covering is equal to
		$u_3 v_3 u_2 u_1 v_1^2 x_2^2 x_1 w_3^2 w_2$.}
	\label{Fig1}
\end{figure}

\subsection{Main results}

In this subsection, to avoid technicalities, we present informal versions of
our main results. For the precise statements, we refer to the corresponding
theorems in the next sections. Each dimer covering of the Aztec diamond is
uniquely determined by an integer-valued height function {\cite{B86}}, which
provides a standard way to parametrize it (see subsection \ref{mnarnths} for
more discussion). We describe random dimer coverings of our model
asymptotically through this height function.

We recall that the edge weights of our model (see Figure 1) are random and CLT
appropriate, in the sense that they satisfy certain asymptotic conditions (see
Definition \ref{BBBB}). Our main result is the quenched CLT for the height
function.

\begin{theorem}[Informal version of the quenched CLT; Theorem \ref{NMNm}]
  \label{Main1}Assume that the random edge weights of the one-periodic Aztec
  diamond are CLT appropriate. Then, almost surely with respect to the
  distribution of the random weights, the height function has a limit shape
  that corresponds to deterministic edge weights. Furthermore, almost surely
  with respect to the distribution of the random weights, the fluctuations of
  the height function are described by the Gaussian Free Field.
\end{theorem}

In contrast to the annealed version, the quenched fluctuations of the height
functions are not affected asymptotically by the randomness of the
environment, and they are still described by the Gaussian Free Field, as in
the deterministic case {\cite{B3}}, {\cite{B74}}. The proof of Theorem
\ref{NMNm} relies on analyzing the quenched moments of the centered and
rescaled height function. Such quenched moments will concentrate to moments of
the Gaussian Free Field. Although we do not address this question, we believe
that the randomness of the environment should still be visible in certain
ways, for example determining the fluctuations around these moments of the
Gaussian Free Field.

The limit shape of Theorem \ref{Main1} arises from the limit shapes of the
random empirical measures that encode our edge weights. The later limit shapes
exist due to our asymptotic conditions. From our techniques it emerges that
the limit shapes of the height function for the one-periodic Aztec diamond
with CLT appropriate edge weights, can be obtained from the same model with
deterministic edge weights. However, to our knowledge these limit shapes are
new.

Our second main result significantly extends the annealed CLT for the height
function, that was recently proved in {\cite{B52}}, for two distinguishing
types of random edge weights. For this case, we distinguish two types of CLT
appropriate edge weights (see Definition \ref{BBBB}), where each type
generalizes the class of random variables that had been considered before. As
already mentioned, CLT appropriate edge weights satisfy a CLT in the sense of
moments. The difference between the two types is that for the first type the
CLT is satisfied after scaling by $\sqrt{M}$, while for the second type the
CLT is satisfied after scaling by $M$, where $M$ is the size of the Aztec
diamond graph. For example, the first type corresponds to i.i.d. edge weights,
where the probability distribution is independent of $M$, while the second
type corresponds to edge weights that are given by the eigenvalues of GUE. In
both cases, we obtain explicit formulas for the annealed moments of the
centered and properly rescaled height function, which are expressed in terms
of moments of the edge weights, that can be analyzed asymptotically due to the
CLT appropriateness. In this way, the Gaussianity of the annealed fluctuations
arises naturally due to the Gaussian fluctuations of the random environment.

\begin{theorem}[Informal versions of annealed CLTs; Theorems \ref{nzgamskl},
\ref{alloena}]
  \label{Main2}Assume that the random edge weights of the one-periodic Aztec
  diamond are CLT appropriate. Then, the height function has a limit shape
  that corresponds to deterministic edge weights. Moreover, the following
  hold:
  \begin{enumerate}
    {\item For CLT appropriate edge weights of the first type, the
    annealed fluctuations of the height function occur on scale $\sqrt{M}$,
    and are described by a Gaussian field.}{\item For CLT appropriate edge
    weights of the second type, the annealed fluctuations of the height
    function occur on scale $M$, and are described by the sum of the Gaussian
    Free Field and an independent Gaussian Field.}
  \end{enumerate}
\end{theorem}

To determine the Gaussian field of Theorem \ref{Main2}, is crucial to know the
Gaussian field that appears from the CLT for the edge weights. When the CLT is
satisfied after scaling by $\sqrt{M}$, the limit is usually described by the
standard Brownian motion.

\begin{corollary}[Brownian motion type fluctuations; Examples \ref{ex21},
\ref{ex22}]
  Assume that the random edge weights of the one-periodic Aztec diamond are
  given either by i.i.d. random variables or by $V$-uniformly ergodic Markov
  chains, and their probability distributions do not depend on $M$. Then, the
  annealed fluctuations of the height function occur on scale $\sqrt{M}$, and
  are described by the sum of two correlated standard Brownian motions.
\end{corollary}

When the CLT is satisfied after scaling by $M$, different Gaussian fields can
appear. In {\cite{B52}} it was shown that for i.i.d. edge weights such that
the probability distribution of each random variable depends on $M$, and its
variance decays like $M^{- 1}$, the Gaussian field of Theorem \ref{Main2} is
again a standard Brownian motion. We consider the case where the edge weights
are described by eigenvalues of the GUE.

\begin{corollary}[GUE random environment; Example \ref{ex24}]
  Assume that the random edge weights of the one-periodic Aztec diamond are
  $\tmmathbf{w}_{M-i+1} =\tmmathbf{l}_i^2 + 1$, for $i = 1, \ldots, M$, where
  $\tmmathbf{l}_1 \leq \cdots \leq \tmmathbf{l}_M$ are the ordered eigenvalues
  of GUE of size $M$, and the rest weights are equal to $1$. Then, the
  annealed fluctuations of the height function occur on scale $M$, and are
  described by the sum of the Gaussian Free Field and an independent Gaussian
  field, whose singularity is weaker than that of the Gaussian Free Field.
\end{corollary}

Furthermore, we consider cases in which the annealed height fluctuations
exhibit non-Gaussian behavior. More precisely the Gaussian Free Field
fluctuations are perturbed by a random discrete component that describes
fluctuations caused by extremal particles on one-dimensional slices of the Aztec diamond (see
Theorem \ref{PoissonCLT} and Remark \ref{PoissonRemark}).

\begin{theorem}[Informal version of Theorem \ref{PoissonCLT}]
	Let $\Xi_M$ be a binomial random variable with parameters $M, p_M$, such that
	$\lim_{M \rightarrow \infty} M p_M =l \in (0, \infty)$. Consider
	the one-periodic Aztec diamond with random edge weights
	\[ (\tmmathbf{w}_M, \ldots, \tmmathbf{w}_1) = \left( \left( \frac{a}{1 - a}
	\right)^{\Xi_M}, \left( \frac{b}{1 - b} \right)^{M - \Xi_M} \right), \]
	where $a, b \in (0, 1)$ are fixed. Then, the annealed fluctuations of the
	height function occur on scale $M$ and are described by the Gaussian Free
	Field and an independent Poisson component.
\end{theorem}

The organization of this paper is as follows: In Section \ref{S2} we give some
preliminary material, related to interlacing signatures, Schur measures, and
their well-known connection to the one-periodic Aztec diamond. Before we prove
the quenched CLT (Theorem \ref{NMNm}), we first focus on the extension of the
annealed version, and develop a common framework of arguments that jointly
prove both versions. In Section \ref{S3}, we present some technical lemmas for
Schur measures with deterministic parameters, that will be useful for studying
the model with random edge weights. In Theorem \ref{LLN!} of this section, we
prove the Law of Large Numbers for the height function (along a given
one-dimensional slice) of the model with CLT appropriate edge weights. The
limit shape obtained here generalizes the analogous result of {\cite{B52}}. In
Section \ref{LLL}, we study the annealed global fluctuations of the height
function at one level, for the model with CLT appropriate edge weights of two
types. We show that the fluctuations converge to a Gaussian vector (see
Theorems \ref{normal1!}, \ref{normal2!}) with covariance structure determined
by the type of edge weights. In Section \ref{S5}, we focus on the multi-level
analogues of Theorems \ref{normal1!}, \ref{normal2!}. We consider
specific examples (see Examples \ref{ex21}, \ref{ex22}, \ref{ex24}), where we
analyze the annealed fluctuations. Moreover, we introduce and study examples of one-periodic random edge weights that lead to non-Gaussian annealed height fluctuations (see Theorem \ref{PoissonCLT}). Finally, Section \ref{S6} is devoted to the
quenched CLT for the height function at several levels (see Theorem
\ref{NMNm}).

\begin{acknowledgments*}
  The author is grateful to Alexey Bufetov for many valuable discussions and
  to Fabio Toninelli for bringing to his attention the quenched version. The
  author was partially supported by the European Research Council (ERC), Grant
  Agreement No. 101041499.
\end{acknowledgments*}

\section{Preliminaries}\label{S2}

In this section, we recall the well known relation between random dimer
coverings of the Aztec diamond with one-periodic weights and the Schur
process. This will allow to address the questions about random dimer
coverings.

\subsection{Dimer coverings of the Aztec diamond and
signatures}\label{mnarnths}

Given a positive integer $N$, an $N$-tuple of non-increasing integers $\lambda
= (\lambda_1 \geq \lambda_2 \geq \cdots \geq \lambda_N)$ is called a signature
of length $N$. We denote by $\mathbb{G}\mathbb{T}_N$ the set of all signatures
of length $N$. The weight of $\lambda$ is given by $| \lambda | \assign
\lambda_1 + \cdots + \lambda_N$. Moreover, given $\lambda \in
\mathbb{G}\mathbb{T}_N$, we will denote by $\lambda' \in
\mathbb{G}\mathbb{T}_{\lambda_1}$ its conjugate. We will use the notation
$\lambda \succ \mu$ for signatures $\lambda \in \mathbb{G}\mathbb{T}_N$ and
$\mu \in \mathbb{G}\mathbb{T}_{N - 1}$ that interlace, namely $\lambda_i \geq
\mu_i \geq \lambda_{i + 1}$, for all $i = 1, \ldots, N - 1$. We will also use
the notation $\theta \succ_v \lambda$ for signatures $\theta, \lambda \in
\mathbb{G}\mathbb{T}_N$ that interlace vertically, meaning that $\theta_i -
\lambda_i \in \{ 0, 1 \}$, for all $i = 1, \ldots, N$.

It is well known that dimer coverings of the Aztec diamond are in bijection
with collections of signatures $\lambda^{(t)}, \theta^{(t)} \in
\mathbb{G}\mathbb{T}_t$ that satisfy certain interlacing properties, and the
one-periodic probability measure on dimer coverings can be described
explicitly under this bijection.

\begin{proposition}
  \label{-+}Let $\{ x_i, w_i, u_i, v_i \}_{i = 1}^M$ be positive real numbers.
  \begin{enumerate}
    {\item There is a bijection between dimer coverings of the Aztec
    diamond of size M and collections of signatures $\{ \lambda^{(i)},
    \theta^{(i)} \}_{i = 1}^M$, where $\lambda^{(i)}, \theta^{(i)} \in
    \mathbb{G}\mathbb{T}_i$, and
    \begin{equation}
      \emptyset \prec \theta^{(1)} \succ_{\nu} \lambda^{(1)} \prec
      \theta^{(2)} \succ_{\nu} \lambda^{(2)} \prec \cdots \prec \theta^{(M -
      1) } \succ_{\nu} \lambda^{(M - 1)} \prec \theta^{(M)} \succ_{\nu}
      \lambda^{(M)} = \emptyset . \label{PRWTO}
    \end{equation}}
    
    \item Under the previous bijection, the product of edge weights of a dimer
    covering $\{ \lambda^{(i)}, \theta^{(i)} \}_{i = 1}^M$ of the Aztec
    diamond, with edge weights as in Figure \ref{Fig1}, is equal to
    \[ \prod_{i = 1}^M v_i^{M - i + 1} u_i^i  \left( \frac{w_i}{u_i}
       \right)^{| \theta^{(i)} | - | \lambda^{(i)} |}  \left( \frac{x_i}{v_i}
       \right)^{| \theta^{(i)} | - | \lambda^{(i - 1)} |}, \]
    and the one-periodic probability measure $\mathbb{P}_{x, w, u, v}$ on
    collections of signatures (\ref{PRWTO}), is given by
    \[ \mathbb{P}_{x, w, u, v} (\{ \lambda^{(i)}, \theta^{(i)} \}_{i = 1}^M) =
       \prod_{i = 1}^M \left( \frac{w_i}{u_i} \right)^{| \theta^{(i)} | - |
       \lambda^{(i)} |} \left( \frac{x_i}{v_i} \right)^{| \theta^{(i)} | - |
       \lambda^{(i - 1)} |} \text{{\hspace{9em}}} \]
    \begin{equation}
      \text{{\hspace{11em}}} \times \prod_{1 \leq i \leq j \leq M} \left( 1 +
      \frac{x_i w_j}{v_i u_j} \right)^{- 1} . \label{HNHNHN}
    \end{equation}
  \end{enumerate}
\end{proposition}

\begin{proof}
  This is a classical result {\cite{B71}}. For a detailed exposition of this
  construction, we refer to {\cite{B45}}.
\end{proof}

Since the one-periodic probability measure (\ref{HNHNHN}) depends on the
ratios $w_i / u_i$, $x_i / v_i$, we will assume that $u_i = v_i = 1$, for all
$i = 1, \ldots, M$.

\begin{notation}
  We will denote by $\mathbb{P}_{\beta, y}$ the probability measure
  (\ref{HNHNHN}) of Proposition \ref{-+}, where $w_i = \beta_i / (1 -
  \beta_i)$, $x_i = 1 - y_i$, and $u_i = v_i = 1$, for all $i = 1, \ldots, M$.
\end{notation}

As already mentioned, the variables $\beta_1, \ldots, \beta_M, x_1, \ldots,
x_M$ will be random. When we want to emphasize on it, we will denote them by
$\tmmathbf{\beta}_1, \ldots, \tmmathbf{\beta}_M, \tmmathbf{x}_1, \ldots,
\tmmathbf{x}_M$. Throughout this paper, we will also make the following
convention.

\begin{convention}
  \label{convention}We assume that the distribution of $\tmmathbf{\beta}_i$ is
  supported on $(0, 1)$, and the distribution of $\tmmathbf{x}_i$ is supported
  on $(\delta, 2 - \delta)$, for some $\delta \in (0, 2)$, and all $i = 1,
  \ldots, M$.
\end{convention}

The dimer coverings of the Aztec diamond are determined by a height function
{\cite{B86}}. Our asymptotic results in the next sections can be described in
terms of this function. For a positive integer $N < M$, we will encode the
random signature $\lambda^{(N)}$ of (\ref{PRWTO}), that describes the
positions of dimers at the $N$th level of the graph, via the random empirical
measure
\[ m [\lambda^{(N)}] \assign \frac{1}{N} \sum_{i = 1}^N \delta \left(
   \frac{\lambda_i^{(N)} + N - i}{N} \right) . \]
Let $F_{m [\lambda^{(N)}]} (x)$ be the cumulative distribution function of $m
[\lambda^{(N)}]$. Then, the (integer-valued) height function of the dimer
coverings over the Aztec diamond along the $N$th level, is given by $N \text{}
F_{m [\lambda^{(N)}]} (x / N)$. In the following, we will focus on the
asymptotic behavior of the random measures $m [\lambda^{(N)}]$, as $N, M
\rightarrow \infty$. In our Central Limit Theorems (see Theorems
\ref{nzgamskl}, \ref{alloena}, \ref{PoissonCLT},  \ref{NMNm}) we will study the asymptotic
behavior of the moments of the random height function
\[ \int_{- \infty}^{\infty} t^k d [N (F_{m [\lambda^{(N)}]} (t) -\mathbb{E}
   [F_{m [\lambda^{(N)}]} (t)])] = \sum_{i = 1}^N \left[ \left(
   \frac{\lambda_i^{(N)} + N - i}{N} \right)^k -\mathbb{E} \left[ \left(
   \frac{\lambda_i^{(N)} + N - i}{N} \right)^k \right] \right] . \]
The expectations in the above formula can be either quenched (i.e. for a given
realization of disorder), or annealed (i.e. when we also average with respect
to the randomness of the edge weights).

\subsection{The one-periodic Aztec diamond and Schur functions}

Given a signature $\lambda \in \mathbb{G}\mathbb{T}_N$, its Schur function
$s_{\lambda}$ is a symmetric Laurent polynomial in $x_1, \ldots, x_N$ of
degree $| \lambda |$, defined by
\[ s_{\lambda} (x_1, \ldots, x_N) \assign \frac{\det (x_i^{\lambda_j + N -
   j})_{1 \leq i, j \leq N}}{\det (x_i^{N - j})_{1 \leq i, j \leq N}} =
   \frac{\det (x_i^{\lambda_j + N - j})_{1 \leq i, j \leq N}}{\prod_{1 \leq i
   < j \leq N} (x_i - x_j)} . \]
The one-periodic probability measure $\mathbb{P}_{x, w, u, v}$ of Proposition
\ref{-+} is strongly related to Schur functions. More precisely,
$\mathbb{P}_{x, w, u, v}$ is an example of a Schur process {\cite{B87}},
{\cite{B88}}, where the corresponding skew Schur functions are functions of
one variable, which is either $\frac{w_i}{x_i}$ or $\frac{u_i}{v_i}$
{\cite{B73}}. A standard computation implies that
\[ \sum_{\theta^{(1)}, \lambda^{(1)}, \ldots, \theta^{(N - 1)}, \lambda^{(N -
   1)}, \theta^{(N)}} \prod_{i = 1}^N a_i^{| \theta^{(i)} | - | \lambda^{(i)}
   |} \cdot \prod_{i = 1}^N b_i^{| \theta^{(i)} | - | \lambda^{(i - 1)} |}
   \text{{\hspace{12em}}} \]
\begin{equation}
  \text{{\hspace{12em}}} = \prod_{1 \leq i \leq j \leq N} (1 + b_i a_j) \cdot
  s_{\lambda^{(N)}} (b_1, \ldots, b_N), \label{2+2=4}
\end{equation}
where the sum is taken with respect to all $\theta^{(1)}, \lambda^{(1)},
\ldots, \theta^{(N - 1)}, \lambda^{(N - 1)}, \theta^{(N)}$, such that
$\lambda^{(i)}, \theta^{(i)} \in \mathbb{G}\mathbb{T}_i$ and $\emptyset \prec
\theta^{(1)} \succ_{\nu} \lambda^{(1)} \prec \cdots \prec \theta^{(N)}
\succ_{\nu} \lambda^{(N)}$, where $\lambda^{(N)} \in \mathbb{G}\mathbb{T}_N$
is fixed. This implies that for $N < M$, the marginal distribution of
$\lambda^{(N)}$ with respect to $\mathbb{P}_{\beta, y}$, is given by the Schur
measure
\[ \rho_{\beta, y} [\lambda] \assign \prod_{i = 1}^N \prod_{j = N + 1}^M (1 +
   w_j x_i)^{- 1} \cdot s_{\lambda} (x_1, \ldots, x_N) s_{\lambda^{'}} \left(
   w_{N + 1} {, \ldots, w_M}  \right) \text{, \quad for } \lambda \in
   \mathbb{G}\mathbb{T}_N . \]
In the next sections, we will study the global asymptotic behavior of such
random signatures, where the parameters $\{ x_i \}_{i = 1}^M, \{ w_i \}_{i =
1}^M$ are random.

\section{The Law of Large Numbers}\label{S3}

In the current section, we prove the Law of Large Numbers. A key feature of
our result is that we do not assume independence for the random weights. This
allows us to analyze the one-periodic Aztec diamond when all the weights are
random and in general, the randomness of the weights is not fixed.

Our proof relies on the determinantal form of the Schur functions. The method
that we use was developed in {\cite{B1}}, {\cite{B3}} and it is related to
Schur generating functions of probability measures on signatures. Schur
measures such that all their parameters are random cannot be analyzed by Schur
generating functions, but it is still possible to make use of the Cauchy
identity.

\subsection{Differential operators}

We recall that we focus on the asymptotics of the random empirical measure
\[ m_N [\overline{\rho_{\tmmathbf{\beta}, \tmmathbf{y}}}] \assign \frac{1}{N}
   \sum_{i = 1}^N \delta \left( \frac{\lambda_i + N - i}{N} \right), \]
where $\lambda = (\lambda_1 \geq \cdots \geq \lambda_N)$ is a random signature
distributed according to $\overline{\rho_{\tmmathbf{\beta}, \tmmathbf{y}}}
[\lambda] \assign \mathbb{E} [\rho_{\tmmathbf{\beta}, \tmmathbf{y}}
[\lambda]]$. The parameters $\{\tmmathbf{x}_i \}_{i = 1}^N$,
$\{\tmmathbf{\beta}_j \}_{j = N + 1}^M$ are random, and $\tmmathbf{y}_i = 1
-\tmmathbf{x}_i$. The moments of the empirical measure can be computed using
differential operators acting on analytic functions of $N$ variables,
\[ \mathcal{D}_k (f) (u_1, \ldots, u_N) \assign \frac{1}{V_N (\vec{u})}
   \sum_{i = 1}^N (u_i \partial_i)^k [V_N (\vec{u}) f (u_1, \ldots, u_N)], \]
where $V_N (\vec{u}) \assign \prod_{1 \leq i < j \leq N} (u_i - u_j)$, and
$\partial_i$ denotes the derivative with respect to $u_i$. Note that the
operators $\mathcal{D}_k$ can be simplified using that
\begin{equation}
  \sum_{i = 1}^N (u_i \partial_i)^k = \sum_{i = 1}^N \sum_{m = 1}^k S (k, m)
  u_i^m \partial_i^m, \label{Stirling}
\end{equation}
where $S (k, m)$ are the Stirling numbers of second kind. These operators are
useful because of the following fact.

\begin{proposition}
  \label{Prop1}For $\lambda \in \mathbb{G}\mathbb{T}_N$, we have
  \[ \mathcal{D}_k [s_{\lambda}] = \sum_{i = 1}^N (\lambda_i + N - i)^k
     s_{\lambda} . \]
\end{proposition}

\begin{proof}
  Immediate from the definition of Schur functions.
\end{proof}

For the following, it is important to understand how $\mathcal{D}_k$ acts on
analytic and symmetric functions. Although, due to (\ref{Stirling}) it
suffices to do it for the differential operators $D_k$, where
\[ D_k (f) (u_1, \ldots, u_N) \assign \frac{1}{V_N (\vec{u})} \sum_{i = 1}^N
   u_i^k \partial_i^k [V_N (\vec{u}) f (u_1, \ldots, u_N)] . \]

\subsection{Some technical lemmas}

Throughout this paper, given positive integers $N, M$ with $N < M$, we will
denote by $[N]$ the finite set $\{ 1, 2, \ldots, N \}$, and we will denote by
$[N ; M]$ the finite set $\{ N + 1, N + 2, \ldots, M \}$. Before we prove our
main result for Schur measures with random parameters, we formulate and prove
some lemmas about Schur measures $\rho_{\beta, y}$, but for the special case
where the parameters $\beta = \{\beta_j \}_{j = N + 1}^M$, $y = \{y_i \}_{i =
1}^N$ are deterministic and such that $0 < \beta_j < 1$ and $- 1 < y_i < 1$.

We recall the (dual) Cauchy identity
\begin{equation}
  \sum_{\lambda \in \mathbb{G}\mathbb{T}_N} \prod_{j = N + 1}^M (1 -
  \beta_j)^N s_{\lambda} (x_1, \ldots, x_N) s_{\lambda'} \left(
  \frac{\beta_M}{1 - \beta_M}, \ldots, \frac{\beta_{N + 1}}{1 - \beta_{N + 1}}
  \right) = \prod_{i = 1}^N \prod_{j = N + 1}^M (1 - \beta_j + \beta_j x_i) .
  \label{CI}
\end{equation}
In the following, we will denote $f (x_1, \ldots, x_N) \assign \prod_{i = 1}^N
\prod_{j = N + 1}^M (1 - \beta_j + \beta_j x_i)$. Due to Proposition
\ref{Prop1}, the expectation of $\sum_{i = 1}^N (\lambda_i + N - i)^k$ with
respect to $\rho_{\beta, y}$ is given by $f^{- 1} \mathcal{D}_k f$, where
$\mathcal{D}_k$ acts on the variables $x_1, \ldots, x_N$. The following lemma
gives a formula for this expectation which will be appropriate for asymptotic
analysis later on.

\begin{lemma}
  \label{Lem2}Consider the operators $D_k$ that act on the variables $x_1,
  \ldots, x_N$, and $y_i = 1 - x_i$. Then, we have that
  \[ \mathcal{F}_k \assign \frac{D_k f}{f} = \frac{1}{2 \pi i} \sum_{m = 0}^k
     \binom{k}{m} (- 1)^m m! \sum_{\{l_0, \ldots, l_m \} \subseteq [N]}
     \sum_{\underset{i_s \neq i_r \text{ for } s \neq r}{i_1, \ldots, i_{k -
     m} \in [N ; M]}} \]
  \begin{equation}
    \oint \frac{(1 - z)^k}{(z - y_{l_0}) \ldots (z - y_{l_m})} 
    \frac{\beta_{i_1} \ldots \beta_{i_{k - m}}}{(1 - \beta_{i_1} z) \ldots (1
    - \beta_{i_{k - m}} z)} d \text{} z, \label{vfsgmw}
  \end{equation}
  where the contour is positively oriented, it encircles all the $y_i$ but it
  does not encircle any of the $\beta_j^{- 1}$.
\end{lemma}

\begin{proof}
  It is immediate that
  \[ D_k f = \sum_{m = 0}^k \binom{k}{m} \sum_{\underset{l_s \neq l_r \text{
     for } s \neq r}{l_0, \ldots, l_m = 1}}^N \frac{x_{l_0}^k
     \partial_{l_0}^{k - m} f}{(x_{l_0} - x_{l_1}) \ldots (x_{l_0} - x_{l_m})}
     . \]
  But
  \[ \frac{\partial_{l_0}^{k - m} f}{f} = \sum_{\underset{i_s \neq i_r}{i_1,
     \ldots, i_{k - m} \in [N ; M]}} \frac{\beta_{i_1} \ldots \beta_{i_{k -
     m}}}{(1 - \beta_{i_1} y_{l_0}) \ldots (1 - \beta_{i_{k - m}} y_{l_0})},
  \]
  and by the residue theorem we have
  \[ \frac{(1 - y_{l_0})^k \partial_{l_0}^{k - m} f}{f (y_{l_0} - y_{l_1})
     \ldots (y_{l_0} - y_{l_m})} + \cdots + \frac{(1 - y_{l_m})^k
     \partial_{l_m}^{k - m} f}{f (y_{l_m} - y_{l_0}) \ldots (y_{l_m} - y_{l_{m
     - 1}})} \]
  \[ = \frac{1}{2 \pi i} \sum_{\underset{i_s \neq i_r}{i_1, \ldots, i_{k - m}
     \in [N ; M]}} \oint \frac{(1 - z)^k}{(z - y_{l_0}) \ldots (z - y_{l_m})} 
     \frac{\beta_{i_1} \ldots \beta_{i_{k - m}}}{(1 - \beta_{i_1} z) \ldots (1
     - \beta_{i_{k - m}} z)} d \text{} z. \]
  This proves the claim.
\end{proof}

Applying finitely many times differential operators $\mathcal{D}_k$ in
(\ref{CI}), we get expectations of
\[ \prod_{j = 1}^{\nu} \left( \sum_{i = 1}^N \left( \frac{\lambda_i + N -
   i}{N} \right)^{k_j} \right), \]
with respect to the Schur measure $\rho_{\beta, y}$. Such expectations can
also be expressed as appropriate contour integrals for $\nu \geq 2$.

\begin{lemma}
  \label{Lemma3}Consider the operators $D_k$ that act on the variables $x_1,
  \ldots, x_N$, and $y_i = 1 - x_i$. Then, we have that
  \[ \frac{D_{k_2} D_{k_1} f}{f} =\mathcal{F}_{k_1} \cdot \mathcal{F}_{k_2}
     +\mathcal{G}_{k_1, k_2} +\mathcal{H}_{k_1, k_2}, \]
  where
  \[ \mathcal{G}_{k_1, k_2} \assign \frac{- 1}{4 \pi^2} \sum_{m = 0}^{k_1}
     \binom{k_1}{m} (- 1)^m m! \sum_{\mu = 0}^{k_2 - 1} \binom{k_2}{\mu} (k_2
     - \mu) (- 1)^{\mu - 1} \mu ! \sum_{\underset{i_s \neq i_r \text{ for } s
     \neq r}{i_1, \ldots, i_{k_1 - m} \in [N ; M]}} \sum_{\{l_1, \ldots, l_m
     \} \subseteq [N]} \]
  \[ \sum_{\underset{\iota_s \neq \iota_r \text{ for } s \neq r}{\iota_1,
     \ldots, \iota_{k_2 - \mu - 1} \in [N ; M]}} \sum_{\{\lambda_0, \ldots,
     \lambda_{\mu} \} \subseteq [N]} \oint \oint \frac{(1 - z)^{k_1} (1 -
     w)^{k_2}}{(z - w)^2}  \frac{\beta_{i_1} \ldots \beta_{i_{k_1 - m}}}{(1 -
     \beta_{i_1} z) \ldots (1 - \beta_{i_{k_1 - m}} z)}  \]
  \[ \times \frac{1}{(z - y_{l_1}) \ldots (z - y_{l_m})} 
     \frac{\beta_{\iota_1} \ldots \beta_{\iota_{k_2 - \mu - 1}}}{(1 -
     \beta_{\iota_1} w) \ldots (1 - \beta_{\iota_{k_2 - \mu - 1}} w)} 
     \frac{1}{(w - y_{\lambda_0}) \ldots (w - y_{\lambda_{\mu}})} d \text{} z
     \text{} d \text{} w, \]
  where the contours $C_z, C_w$ are positively oriented, they encircle all the
  $y_i$ but they do not encircle any of the $\beta_j^{- 1}$, and $C_z$
  encircles $C_w$. Moreover, $\mathcal{H}_{k_1, k_2}$ is a linear combination
  of double contour integrals of the form
  \[ \sum_{J_1 \subseteq [N ; M] \of |J_1 | = \alpha_1} \sum_{I_1 \subseteq
     [N] \of |I_1 | = \gamma_1} \sum_{J_2 \subseteq [N ; M] \of |J_2 | =
     \alpha_2} \sum_{I_2 \subseteq [N] \of |I_2 | = \gamma_2} \oint \oint
     \frac{(1 - z)^{k_1} (1 - w)^{k_2}}{(z - w)^{\alpha}} \]
  \begin{equation}
    \prod_{j \in J_1} \frac{\beta_j}{1 - \beta_j z} \cdot \prod_{i \in I_1}
    \frac{1}{z - y_i} \cdot \prod_{j \in J_2} \frac{\beta_j}{1 - \beta_j w}
    \cdot \prod_{i \in I_2} \frac{1}{w - y_i} d \text{} z \text{} d \text{} w,
    \label{ast}
  \end{equation}
  where the contours $C_z$, $C_w$ are as before, $\alpha \geq 3$, and
  $\alpha_1 + \gamma_1 + \alpha_2 + \gamma_2 < k_1 + k_2$.
\end{lemma}

\begin{proof}
  We write
  \begin{equation}
    \frac{1}{f} D_{k_2} D_{k_1} f = \frac{1}{f} D_{k_2} \left[ f \cdot
    \frac{D_{k_1} f}{f} \right] . \label{operat}
  \end{equation}
  We have shown that $f^{- 1} D_{k_1} f$ is a sum of contour integrals and the
  factors of the integrands that depend on $x_1, \ldots, x_N$ are
  \[ g_{m + 1} (z ; x_1, \ldots, x_N) = \sum_{\{l_0, \ldots, l_m \} \subseteq
     [N]} \frac{1}{(z - 1 + x_{l_0}) \ldots (z - 1 + x_{l_m})} . \]
  Note that $g_m (z ; x_1, \ldots, x_N)$ is symmetric in $x_1, \ldots, x_N$.
  On the right hand side of (\ref{operat}), in the case that $D_{k_2}$ acts
  only on $f$ we get $\mathcal{F}_{k_1} \cdot \mathcal{F}_{k_2}$. We also have
  to compute the summand of (\ref{operat}) that emerges when $D_{k_2}$ acts on
  $g_m (z ; x_1, \ldots, x_N)$ as well. We have
  \[ \frac{D_{k_2} (fg_{m + 1})}{f} = \text{} \]
  \begin{equation}
    \sum_{\mu = 0}^{k_2} \sum_{a = 0}^{k_2 - \mu} \binom{k_2}{\mu} \binom{k_2
    - \mu}{a} \sum_{\underset{\lambda_s \neq \lambda_r \text{ for } s \neq
    r}{\lambda_0, \ldots, \lambda_{\mu} \in [N]}} \sum_{\underset{\iota_s \neq
    \iota_r \text{ for } s \neq r}{\iota_1, \ldots, \iota_{k_2 - \mu - a} \in
    [N ; M]}} \frac{x_{\lambda_0}^{k_2} \partial_{\lambda_0}^a g_{m + 1} (z ;
    \vec{x})}{(x_{\lambda_0} - x_{\lambda_1}) \ldots (x_{\lambda_0} -
    x_{\lambda_{\mu}})} \label{4}
  \end{equation}
  \begin{equation}
    \times \frac{\beta_{\iota_1} \ldots \beta_{\iota_{k_2 - \mu - a}}}{(1 -
    \beta_{\iota_1} + \beta_{\iota_1} x_{\lambda_0}) \ldots (1 -
    \beta_{\iota_{k_2 - \mu - a}} + \beta_{\iota_{k_2 - \mu - a}}
    x_{\lambda_0})} . \label{formu}
  \end{equation}
  Moreover,
  \[ \partial_{\lambda_0}^a g_{m + 1} (z ; \vec{x}) = \partial_{\lambda_0}^a
     \left( \frac{1}{z - 1 + x_{\lambda_0}} \right) \sum_{\{l_1, \ldots, l_m
     \} \subseteq [N] \backslash \{\lambda_0 \}} \frac{1}{(z - 1 + x_{l_1})
     \ldots (z - 1 + x_{l_m})} \]
  \[ = (- 1)^a a! \left( \frac{1}{z - y_{\lambda_0}} \right)^{a + 1}
     \sum_{\sigma = 0}^m (- 1)^{m - \sigma} \left( \frac{1}{z - y_{\lambda_0}}
     \right)^{m - \sigma} g_{\sigma} (z ; 1 - y_1, \ldots, 1 - y_N), \]
  where $y_i = 1 - x_i$, and $g_0 \equiv 1$. Since $g_{\sigma} (z ; 1 - y_1,
  \ldots, 1 - y_N)$ is symmetric in $y_1, \ldots, y_N$, the sum with respect
  to $\lambda_0, \ldots, \lambda_{\mu}$ in (\ref{4}) will give rise to
  symmetric terms of the form
  \[ \frac{(1 - y_{\lambda_0})^{k_2}  (z - y_{\lambda_0})^{\sigma - m - a - 1}
     \cdot \partial_{\lambda_0}^{k_2 - \mu - a} f}{f (y_{\lambda_0} -
     y_{\lambda_1}) \ldots (y_{\lambda_0} - y_{\lambda_{\mu}})} + \cdots +
     \frac{(1 - y_{\lambda_{\mu}})^{k_2}  (z - y_{\lambda_{\mu}})^{\sigma - m
     - a - 1} \cdot \partial_{\lambda_{\mu}}^{k_2 - \mu - a} f}{f
     (y_{\lambda_{\mu}} - y_{\lambda_0}) \ldots (y_{\lambda_{\mu}} -
     y_{\lambda_{\mu - 1}})} . \]
  By the residue theorem, writing these symmetric terms as contour integrals
  we deduce that the summand of (\ref{operat}) that emerges when $D_{k_2}$
  acts on $f^{- 1} D_{k_1} f$ will be the sum of $\mathcal{G}_{k_1, k_2}$
  (which corresponds to $a = 1$ and $\sigma = m$) and a linear combination of
  double contour integrals (\ref{ast}).
\end{proof}

In the following, by abuse of notation we will denote by $\mathcal{I}
(\{\alpha_i, \gamma_i \}_{i = 1}^{\mu})$ contour integrals of the form
\[ C \sum_{J_1 \subseteq [N ; M] : |J_1 | = \alpha_1} \sum_{I_1 \subseteq [N]
   \of |I_1 | = \gamma_1} \ldots \sum_{J_{\mu} \subseteq [N ; M] : |J_{\mu} |
   = \alpha_{\mu}} \sum_{I_{\mu} \subseteq [N] \of |I_{\mu} | = \gamma_{\mu}}
\]
\[ \oint \ldots \oint (1 - z_1)^{\kappa_1} \ldots (1 - z_{\mu})^{\kappa_{\mu}}
   \prod_{i < j} \left( \frac{1}{z_i - z_j} \right)^{k_{i, j}} \prod_{j \in
   J_1} \frac{\beta_j}{1 - \beta_j z_1} \cdot \prod_{i \in I_1} \frac{1}{z_1 -
   y_i} \]
\[ \times \cdots \times \prod_{j \in J_{\mu}} \frac{\beta_j}{1 - \beta_j
   z_{\mu}} \cdot \prod_{i \in I_{\mu}} \frac{1}{z_{\mu} - y_i} d \text{} z_1
   \ldots d \text{} z_{\mu}, \]
where $C$ can be any constant and $\kappa_i$, $k_{i, j}$ can be any
non-negative integers. Furthermore, the contours $C_{z_i}$ are positively
oriented, they encircle all the $y_i$ but they do not encircle any of the
$\beta_j^{- 1}$, and $C_{z_i}$ encircles $C_{z_j}$ for $i < j$. Note that for
$\mu = 1, 2$, such integrals describe $f^{- 1} D_{k_1} f$, $f^{- 1} D_{k_2}
D_{k_1} f$. Before we state the analogous of Lemma \ref{Lemma3} and Lemma
\ref{Lem2} for $f^{- 1} D_{k_{\nu}} \ldots D_{k_1} f$, we prove the following.

\begin{lemma}
  \label{nbcf}Consider the operators $D_k$ that act on variables $x_1, \ldots,
  x_N$, and $y_i = 1 - x_i$. Then, we have that
  \begin{enumerate}
    {\item \label{hghmn}The term $f^{- 1} D_{k_{l + 1}}
    [f\mathcal{F}_{k_1} \ldots \mathcal{F}_{k_l}] -\mathcal{F}_{k_1} \ldots
    \mathcal{F}_{k_{l + 1}}$ is a linear combination of contour integrals
    \[ \mathcal{F}_{k_{i_1}} \ldots \mathcal{F}_{k_{i_{\nu}}} \mathcal{I}
       (\{\alpha_i, \gamma_i \}_{i = 1}^{l + 1 - \nu}), \]
    where $\nu \leq l - 1$, $\{i_1, \ldots, i_{\nu} \} \subseteq [l]$, and
    \[ \sum_{j = 1}^{l + 1 - \nu} (\alpha_j + \gamma_j) \leq \sum_{j = 1}^{l +
       1} k_j - \sum_{j = 1}^{\nu} k_{i_j} - l + 1 + \nu . \]}{\item
    \label{vcx}The term
    \[ f^{- 1} D_{k_{l + 1}} [f\mathcal{F}_{k_1} \ldots \mathcal{F}_{k_l}
       \mathcal{I}(\{\alpha_i, \gamma_i \}_{i = 1}^{\mu})] -\mathcal{F}_{k_1}
       \ldots \mathcal{F}_{k_{l + 1}} \mathcal{I} (\{\alpha_i, \gamma_i \}_{i
       = 1}^{\mu}) \]
    is a linear combination of contour integrals $\mathcal{F}_{k_{i_1}} \ldots
    \mathcal{F}_{k_{i_{\nu}}} \mathcal{I} (\{\alpha_i', \gamma_i' \}_{i =
    1}^{l + 1 + \mu - \nu})$, where $\nu \leq l - 1$, $\{i_1, \ldots, i_{\nu}
    \} \subseteq [l]$,
    \[ \sum_{i = 1}^{l + 1 + \mu - \nu} (\alpha_i' + \gamma_i') \leq \sum_{i =
       1}^{\mu} (\alpha_i + \gamma_i) + \sum_{j = 1}^{l + 1} k_j - \sum_{j =
       1}^{\nu} k_{i_j} - l + 1 + \nu, \]
    and of contour integrals $\mathcal{F}_{k_1} \ldots \mathcal{F}_{k_l}
    \mathcal{I} (\{\alpha_i', \gamma_i' \}_{i = 1}^{\mu + 1})$, where
    \[ \sum_{i = 1}^{\mu + 1} (\alpha_i' + \gamma_i') \leq k_{l + 1} - 1 +
       \sum_{i = 1}^{\mu} (\alpha_i + \gamma_i) . \]}
  \end{enumerate}
\end{lemma}

\begin{proof}
  We start with the proof of \ref{hghmn}. To show the claim, we consider the
  case that $D_{k_{l + 1}}$ acts on both the product of symmetric functions
  \[ g_m (z ; \vec{x}) = \sum_{I \subseteq [N] \of |I| = m} \prod_{i \in I}
     \frac{1}{z - 1 + x_i} \text{\quad and\quad} f (\vec{x}) = \prod_{i = 1}^N
     \prod_{j = N + 1}^M (1 - \beta_j + \beta_j x_i) . \]
  By the Leibniz rule, we have
  \[ f^{- 1} D_{k_{l + 1}} \left[ f \prod_{j = 1}^l g_{m_j + 1} (z_j ;
     \vec{x}) \right] = \sum_{n = 0}^{k_{l + 1}} \binom{k_{l + 1}}{n}
     \sum_{\varepsilon_0 + \cdots + \varepsilon_l = k_{l + 1} - n} \binom{k_{l
     + 1} - n}{\varepsilon_0, \ldots, \varepsilon_l} \sum_{\underset{\lambda_s
     \neq \lambda_r \text{ for } s \neq r}{\lambda_0, \ldots, \lambda_n \in
     [N]}} \]
  \begin{equation}
    \frac{x_{\lambda_0}^{k_{l + 1}} \partial_{\lambda_0}^{\varepsilon_0} f
    \partial_{\lambda_0}^{\varepsilon_1} g_{m_1 + 1} (z_1 ; \vec{x}) \cdot
    \cdots \cdot \partial_{\lambda_0}^{\varepsilon_l} g_{m_l + 1} (z_l ;
    \vec{x})}{f (x_{\lambda_0} - x_{\lambda_1}) \ldots (x_{\lambda_0} -
    x_{\lambda_n})}, \label{fhd}
  \end{equation}
  where $m_j \leq k_j$. To compute the desired term, we must consider the case
  where $\varepsilon_0 \neq k_{l + 1} - n$. We also assume that only for
  $\{i_1, \ldots, i_{\nu} \} \subseteq [l]$, we have $\varepsilon_{i_1} =
  \cdots = \varepsilon_{i_{\nu}} = 0$, and let $\{\iota_1, \ldots, \iota_{l -
  \nu} \} = [l] \backslash \{i_1, \ldots, i_{\nu} \}$. Writing
  \begin{equation}
    \partial_{\lambda_0}^{\varepsilon} g_{m + 1} (z ; \vec{x}) = (-
    1)^{\varepsilon} \varepsilon ! \left( \frac{1}{z - y_{\lambda_0}}
    \right)^{\varepsilon + 1}  \sum_{\sigma = 0}^m (- 1)^{m - \sigma}  \left(
    \frac{1}{z - y_{\lambda_0}} \right)^{m - \sigma} g_{\sigma} (z ; 1 - y_1,
    \ldots, 1 - y_N), \label{vbvn}
  \end{equation}
  and using the residue theorem for the sum with respect to $\lambda_0,
  \ldots, \lambda_n$, we obtain linear combinations of
  \[ \mathcal{F}_{k_{i_1}} \ldots \mathcal{F}_{k_{i_{\nu}}} \mathcal{I}
     (\{\alpha_i, \gamma_i \}_{i = 1}^{l + 1 - \nu}) . \]
  Due to (\ref{vbvn}) and the presence of $f^{- 1}
  \partial_{\lambda_0}^{\varepsilon_0} f$, to maximize $\sum_{i = 1}^{l + 1 -
  \nu} (\alpha_i + \gamma_i)$, in (\ref{fhd}) we have to differentiate
  $g_{m_{\iota_1} + 1}, \ldots, g_{m_{\iota_{l - \nu}} + 1}$ as less as
  possible. This corresponds to the case where $\varepsilon_{\iota_1} = \cdots
  = \varepsilon_{\iota_{l - \nu}} = 1$ and $\varepsilon_0 = k_{l + 1} - n - l
  + \nu$. Therefore, we deduce that
  \[ \sum_{i = 1}^{l + 1 - \nu} (\alpha_i + \gamma_i) \leq k_{l + 1} - n - l +
     \nu + n + 1 + \sum_{j = 1}^{\nu - l} k_{\iota_j} = \sum_{j = 1}^{l + 1}
     k_j - \sum_{j = 1}^{\nu} k_{i_j} - l + 1 + \nu . \]
  The proof of \ref{vcx} is similar and it is omitted.
\end{proof}

Our goal is to prove formulas for $f^{- 1} D_{k_{\nu}} \ldots D_{k_1} f$ in
which we distinguish the summands that are products of $\mathcal{F}_k$ and
$\mathcal{G}_{k, l}$, $k, l \in \mathbb{N}$. This is important because the
terms $\mathcal{F}_k$ will give the limit in the LLN, while both terms
$\mathcal{F}_k$ and $\mathcal{G}_{k, l}$ will give the covariance of the
limiting (Gaussian) random vector (see Sections \ref{LLL}, \ref{S5},
\ref{S6}).

\begin{lemma}
  \label{Lemapent}Consider the operators $D_k$ that act on variables $x_1,
  \ldots, x_N$, and $y_i = 1 - x_i$. Then, for $\nu \geq 2$ we have that
  \begin{equation}
    \frac{D_{k_{\nu}} \ldots D_{k_1} f}{f} =\mathcal{F}_{k_1} \ldots
    \mathcal{F}_{k_{\nu}} + \sum_{\underset{\pi_1 \neq \emptyset}{\pi_1 \sqcup
    \pi_2 = [\nu]}} \left( \prod_{(i, j) \in \pi_1} \mathcal{G}_{k_i, k_j}
    \right) \left( \prod_{i \in \pi_2} \mathcal{F}_{k_i} \right)
    +\mathcal{H}_{k_1, \ldots, k_{\nu}} . \label{adfh}
  \end{equation}
  The above sum is taken with respect to all non-empty sets $\pi_1$ of pairs
  $(i, j)$, with $i < j$, of elements of $[\nu]$ and all subsets $\pi_2$ of
  $[\nu]$ (that could be the empty set), such that $\pi_1$, $\pi_2$ are
  disjoint and their union is equal to $[\nu]$. Moreover, $\mathcal{H}_{k_1,
  \ldots, k_{\nu}}$ is a linear combination of terms
  \[ \mathcal{F}_{k_{i_1}} \ldots \mathcal{F}_{k_{i_{\lambda}}} \cdot
     \mathcal{I} (\{\alpha_i, \gamma_i \}_{i = 1}^{\nu - \lambda}), \]
  where $0 \leq \lambda \leq \nu - 2$, $i_1, \ldots, i_{\lambda} \in [\nu]$
  and $\sum_{i = 1}^{\nu - \lambda} (\alpha_i + \gamma_i) < \sum_{i = 1}^{\nu}
  k_i - \sum_{j = 1}^{\lambda} k_{i_j}$.
\end{lemma}

\begin{proof}
  We will show the claim by induction on $\nu$. For $\nu = 2$ this is the
  result of Lemma \ref{Lemma3}. We assume that it holds for $\nu > 2$. For
  positive integers $k_1, \ldots, k_{\nu + 1}$, we write
  \[ \frac{D_{k_{\nu + 1}} \ldots D_{k_1} f }{f} = \frac{1}{f} D_{k_{\nu + 1}}
     \left( f \cdot \frac{D_{k_{\nu}} \ldots D_{k_1} f}{f} \right) . \]
  \begin{equation}
    = \frac{1}{f} D_{k_{\nu + 1}} \left( f \cdot \mathcal{F}_{k_1} \ldots
    \mathcal{F}_{k_{\nu}} + f \sum_{\underset{\pi_1 \neq \emptyset}{\pi_1
    \sqcup \pi_2 = [\nu]}} \left( \prod_{(i, j) \in \pi_1} \mathcal{G}_{k_i,
    k_j} \right) \left( \prod_{i \in \pi_2} \mathcal{F}_{k_i} \right) + f
    \cdot \mathcal{H}_{k_1, \ldots, k_{\nu}} \right) . \label{sdsf}
  \end{equation}
  Note that in order to get the corresponding sum (with respect to $\pi_1,
  \pi_2$) of (\ref{adfh}), every time that we apply a differential operator
  $D_{k_j}$, this must act either only to $f$ (which will give us
  $\mathcal{F}_{k_j}$) or to $f$ and only one of the $\mathcal{F}_{k_i}$, $i <
  j$. The pairs $(i, j) \in \pi_1$ express the situation where the $j$-th time
  that we apply the differential operator this acts on $f$ and
  $\mathcal{F}_{k_i}$.
  
  In (\ref{sdsf}), we first consider the case where $D_{k_{\nu + 1}}$ acts
  only to $f$. This will give us the desired summand $\mathcal{F}_{k_1} \ldots
  \mathcal{F}_{k_{\nu + 1}}$ and sums of $\left( \prod_{(i, j) \in \pi_1}
  \mathcal{G}_{k_i, k_j} \right) \left( \prod_{i \in \pi_2} \mathcal{F}_{k_i}
  \right)$ with respect to all $\pi_1, \pi_2$ as before, but such that $\nu +
  1 \in \pi_2$. The summands that correspond to $\pi_1, \pi_2$, such that $(i,
  \nu + 1) \in \pi_1$ will emerge from the second summand of (\ref{sdsf}) when
  $D_{k_{\nu + 1}}$ acts to $f$ and to only one of the $\mathcal{F}_{k_i}$, $i
  \in \pi_2$. In order to get such terms, when we apply $D_{k_{\nu + 1}}$ we
  have to differentiate once $\mathcal{F}_{k_i}$ and $k_{\nu} - 1$ times the
  product of $f$ and of the Vandermonde determinant. Note that if we
  differentiate $\mathcal{F}_{k_i}$ more than once, then we get
  $\mathcal{H}_{k_1, \ldots, k_{\nu + 1}}$. It remains to show that all the
  other cases give terms $\mathcal{H}_{k_1, \ldots, k_{\nu + 1}}$. We treat
  each of the three summands inside the parenthesis of (\ref{sdsf})
  separately.
  
  We start with $f \cdot \mathcal{F}_{k_1} \ldots \mathcal{F}_{k_{\nu}}$. Due
  to Lemma \ref{nbcf}, for the summands $\mathcal{F}_{k_{i_1}} \ldots
  \mathcal{F}_{k_{i_{\lambda}}} \mathcal{I} (\{\alpha_i, \gamma_i \}_{i =
  1}^{\nu + 1 - \lambda})$ of $f^{- 1} D_{k_{\nu + 1}} (f\mathcal{F}_{k_1}
  \ldots \mathcal{F}_{k_{\nu}})$, such that $\lambda < \nu - 1$, we have that
  \[ \sum_{i = 1}^{\nu + 1 - \lambda} (\alpha_i + \gamma_i) \leq \sum_{j =
     1}^{\nu + 1} k_j - \sum_{j = 1}^{\lambda} k_{i_j} - \nu + 1 + \lambda <
     \sum_{j = 1}^{\nu + 1} k_j - \sum_{j = 1}^{\lambda} k_{i_j}, \]
  i.e., they are of the form $\mathcal{H}_{k_1, \ldots, k_{\nu + 1}}$.
  
  Let $\pi_1, \pi_2$ fixed, with $\pi_1 \sqcup \pi_2 = [\nu]$, and $\pi_1 \neq
  \emptyset$. We also assume that $\pi_2 = \{\iota_1, \ldots, \iota_{\lambda}
  \}$. We consider the summands of $f^{- 1} D_{k_{\nu + 1}} \left(
  f\mathcal{F}_{k_{\iota_1}} \ldots \mathcal{F}_{k_{\iota_{\lambda}}} 
  \prod_{(i, j) \in \pi_1} \mathcal{G}_{k_i, k_j} \right)$ of the form
  $\mathcal{F}_{k_{\theta_1}} \ldots \mathcal{F}_{k_{\theta_{\mu}}}
  \mathcal{I} (\{\alpha_i, \gamma_i \}_{i = 1}^{\nu + 1 - \mu})$, where
  $\{\theta_1, \ldots, \theta_{\mu} \} \subsetneq \{\iota_1, \ldots,
  \iota_{\lambda} \}$, and $\mathcal{I} (\{\alpha_i, \gamma_i \}_{i = 1}^{\nu
  + 1 - \mu})$ has at least one factor different from $\mathcal{G}_{k, l}$.
  This implies that
  \[ \sum_{i = 1}^{\nu + 1 - \mu} (\alpha_i + \gamma_i) < k_{\nu + 1} +
     \sum_{j = 1}^{\lambda} k_{\iota_j} - \sum_{j = 1}^{\mu} k_{\theta_j} +
     \sum_{(i, j) \in \pi_1} (k_i + k_j) - \lambda + 1 + \mu \leq \sum_{j =
     1}^{\nu + 1} k_j - \sum_{j = 1}^{\mu} k_{\theta_j} . \]
  We also have to consider the summands of the form $\mathcal{F}_{k_{\iota_1}}
  \ldots \mathcal{F}_{k_{\iota_{\lambda}}} \mathcal{I} (\{\alpha_i, \gamma_i
  \}_{i = 1}^{2| \pi_1 | + 1})$. As a corollary of Lemma \ref{nbcf}, for those
  we have
  \[ \sum_{i = 1}^{2 | \pi_1 | + 1} (\alpha_i + \gamma_i) \leq k_{\nu + 1} - 1
     + \sum_{(i, j) \in \pi_1} (k_i + k_j) < \sum_{j = 1}^{\nu + 1} k_j -
     \sum_{j = 1}^{\lambda} k_{\iota_j} . \]
  Finally, we consider $f \cdot \mathcal{H}_{k_1, \ldots, k_{\nu}}$, or
  equivalently $f\mathcal{F}_{k_{i_1}} \ldots \mathcal{F}_{k_{i_{\lambda}}}
  \mathcal{I} (\{\alpha_i, \gamma_i \}_{i = 1}^{\nu - \lambda})$ for
  appropriate $\alpha_i, \gamma_i$. Applying again Lemma \ref{nbcf}, $f^{- 1}
  D_{k_{\nu + 1}} \left( f\mathcal{F}_{k_{i_1}} \ldots
  \mathcal{F}_{k_{i_{\lambda}}} \mathcal{I}(\{\alpha_i, \gamma_i \}_{i =
  1}^{\nu - \lambda}) \right)$ is a big sum where some of its summands have
  the form $\mathcal{F}_{k_{\iota_1}} \ldots \mathcal{F}_{k_{\iota_{\mu}}}
  \mathcal{I} (\{\alpha_i', \gamma_i' \}_{i = 1}^{\nu + 1 - \mu})$, where $\mu
  \leq \lambda - 1$, $\{\iota_1, \ldots, \iota_{\mu} \} \subseteq \{i_1,
  \ldots, i_{\lambda} \}$, and
  \[ \sum_{i = 1}^{\nu + 1 - \mu} (\alpha_i' + \gamma_i') \leq \sum_{i =
     1}^{\nu - \lambda} (\alpha_i + \gamma_i) + k_{\nu + 1} + \sum_{j =
     1}^{\lambda} k_{i_j} - \sum_{j = 1}^{\mu} k_{\iota_j} - \lambda + 1 + \mu
  \]
  \[ < \sum_{j = 1}^{\nu} k_j - \sum_{j = 1}^{\lambda} k_{i_j} + k_{\nu + 1} +
     \sum_{j = 1}^{\lambda} k_{i_j} - \sum_{j = 1}^{\mu} k_{\iota_j} = \sum_{j
     = 1}^{\nu + 1} k_j - \sum_{j = 1}^{\mu} k_{\iota_j} . \]
  In the last inequality we used our induction hypothesis. For the same
  reason, the rest of its summands have the form $\mathcal{F}_{k_{i_1}} \ldots
  \mathcal{F}_{k_{i_{\lambda}}} \mathcal{I} (\{\alpha_i', \gamma_i' \}_{i =
  1}^{\nu - \lambda + 1})$, where
  \[ \sum_{i = 1}^{\nu - \lambda + 1} (\alpha_i' + \gamma_i') \leq k_{\nu + 1}
     - 1 + \sum_{i = 1}^{\nu - \lambda} (\alpha_i + \gamma_i) < k_{\nu + 1} -
     1 + \sum_{j = 1}^{\nu} k_j - \sum_{j = 1}^{\lambda} k_{i_j} < \sum_{j =
     1}^{\nu + 1} k_j - \sum_{j = 1}^{\lambda} k_{i_j} . \]
  This proves the claim.
\end{proof}

\subsection{Random weights and limit shape}

Now, we consider the parameters $\tmmathbf{x}_i, \tmmathbf{\beta}_j$ of the
Schur measure $\rho_{\tmmathbf{\beta}, \tmmathbf{y}}$ to be random. We
investigate sufficient conditions to obtain a limit shape for the random
measure $m_N [\overline{\rho_{\tmmathbf{\beta}, \tmmathbf{y}}}]$. We will show
that the existence of limit shapes for the empirical measures of
$\{\tmmathbf{x}_i \}_{i = 1}^M$ and $\{\tmmathbf{\beta}_j \}_{j = 1}^M$,
allows to extract a limit shape for $m_N [\overline{\rho_{\tmmathbf{\beta},
\tmmathbf{y}}}]$.

\begin{definition}
  \label{Def}Let $\{\tmmathbf{y}_i \}_{i = 1}^M$, $\{\tmmathbf{\beta}_j \}_{j
  = 1}^M$ be sequences of random variables. We will call such sequences
  (jointly) LLN appropriate if the following hold:
  \begin{itemize}
    {\item Let $r, l$ be non-negative integers and $0 < \gamma_1,
    \ldots, \gamma_r, \varepsilon_1, \ldots, \varepsilon_l \leq 1$. Then, for
    any collection of non-negative integers $k_1, \ldots, k_r, m_1, \ldots,
    m_l$, we have that
    \[ \lim_{M \rightarrow \infty} \mathbb{E} \left[ \prod_{i = 1}^r \sum_{j =
       1}^{\lfloor \gamma_i M \rfloor} \frac{\tmmathbf{y}_j^{k_i}}{\lfloor
       \gamma_i M \rfloor} \cdot \prod_{i = 1}^l \sum_{j = 1}^{\lfloor
       \varepsilon_i M \rfloor} \frac{\tmmathbf{\beta}_j^{m_i}}{\lfloor
       \varepsilon_i M \rfloor} \right] \text{{\hspace{10em}}} \]
    \[ \text{{\hspace{7em}}} = \prod_{i = 1}^r \prod_{j = 1}^l \lim_{M
       \rightarrow \infty} \mathbb{E} \left[ \frac{\tmmathbf{y}_1^{k_i} +
       \cdots +\tmmathbf{y}_M^{k_i}}{M} \right] \mathbb{E} \left[
       \frac{\tmmathbf{\beta}_1^{m_j} + \cdots +\tmmathbf{\beta}_M^{m_j}}{M}
       \right] . \]}{\item Let $\mathfrak{c}_m = \lim_{M \rightarrow \infty}
    M^{- 1} \mathbb{E} [\tmmathbf{y}_1^m + \cdots +\tmmathbf{y}_M^m]$, and
    $\mathfrak{g}_m = \lim_{M \rightarrow \infty} M^{- 1} \mathbb{E}
    [\tmmathbf{\beta}_1^m + \cdots +\tmmathbf{\beta}_M^m]$. Then, the power
    series
    \[ \sum_{m = 0}^{\infty} \mathfrak{c}_m z^m \text{\quad and\quad} \sum_{m
       = 0}^{\infty} \mathfrak{g}_m z^m, \]
    converge uniformly in a neighborhood of $0$.}
  \end{itemize}
\end{definition}

Note that the convergence in expectation condition of Definition \ref{Def}
implies that the random empirical measures
\[ \frac{1}{N} \sum_{i = 1}^N \delta_{\tmmathbf{y}_i} \text{\quad and\quad}
   \frac{1}{N} \sum_{i = 1}^N \delta_{\tmmathbf{\beta}_i}, \]
converge in probability (and in the sense of moments) to deterministic
probability measures on $\mathbb{R}$. In the next theorem we show that such
asymptotic condition for the random parameters of the Schur measure
$\rho_{\tmmathbf{\beta}, \tmmathbf{y}}$ passes to $m_N
[\overline{\rho_{\tmmathbf{\beta}, \tmmathbf{y}}}]$.

\begin{theorem}
  \label{LLN!}Assume that $\{\tmmathbf{x}_i \}_{i = 1}^M, \{\tmmathbf{\beta}_j
  \}_{j = 1}^M$ are LLN appropriate sequences of random variables. We also
  denote $\tmmathbf{y}_i = 1 -\tmmathbf{x}_i$, and
  \[ \mathbf{F}_1 (z) \assign \sum_{i = 0}^{\infty} \mathfrak{c}_i z^i
     \text{\quad and\quad} \mathbf{F}_2 (z) \assign \sum_{i = 0}^{\infty}
     \mathfrak{g}_i z^i, \]
  where $\mathfrak{c}_i \assign \lim_{M \rightarrow \infty} \mathbb{E} [M^{-
  1} (\tmmathbf{y}_1^i + \cdots +\tmmathbf{y}_M^i)]$ and $\mathfrak{g}_i
  \assign \lim_{M \rightarrow \infty} \mathbb{E} [M^{- 1}
  (\tmmathbf{\beta}_1^i + \cdots +\tmmathbf{\beta}_M^i)]$. Then, for the
  random measure $m_N [\overline{\rho_{\tmmathbf{\beta}, \tmmathbf{y}}}]$ we
  have that
  \begin{equation}
    \lim_{\underset{N / M \rightarrow \gamma}{N, M \rightarrow \infty}}
    \mathbb{E}  \left[ \prod_{j = 1}^{\nu} \left( \frac{1}{N} \sum_{i = 1}^N
    \left( \frac{\lambda_i + N - i}{N} \right)^{k_j} \right) \right] =
    \prod_{j = 1}^{\nu} \mathfrak{m}_{k_j}, \label{genca}
  \end{equation}
  where $k_1, \ldots, k_{\nu}$ are arbitrary non-negative integers, and
  \begin{equation}
    \mathfrak{m}_k \assign \frac{1}{2 \pi i (k + 1)} \oint_{|z| = \varepsilon}
    \frac{1}{z - 1} \left( \left( \frac{1}{\gamma} - 1 \right) \left( \frac{1
    - z}{z} \mathbf{F}_2 (z) - \frac{1 - z}{z} \right) + \frac{z - 1}{z}
    \mathbf{F}_1 \left( \frac{1}{z} \right) \right)^{k + 1} d \text{} z.
    \label{momenh}
  \end{equation}
\end{theorem}

\begin{proof}
  We first show the claim for $\nu = 1$. Initially, we consider the parameters
  $x_i, \beta_j$ of the Schur measure to be deterministic, and at the end we
  will integrate with respect to their joint distribution. Applying
  $\mathcal{D}_{k_1}$ (that acts on the variables $x_i$) to both sides of the
  Cauchy identity (\ref{CI}), and dividing by $f (\vec{x}) = \prod_{i = 1}^N
  \prod_{j = N + 1}^M (1 - \beta_j + \beta_j x_i)$, we get an expression for
  \[ \sum_{\lambda \in \mathbb{G}\mathbb{T}_N} \left( \sum_{i = 1}^N
     (\lambda_i + N - i)^k \right) \rho_{\beta, y} [\lambda] , \]
  which can be analyzed as $N, M \rightarrow \infty$. First, we do it for its
  summand $\mathcal{F}_{k_1} = f^{- 1} D_{k_1} f$. We recall that
  \[ \mathcal{F}_{k_1} = \frac{1}{2 \pi i} \sum_{m = 0}^{k_1} \binom{k_1}{m}
     (- 1)^m m! \sum_{\{l_0, \ldots, l_m \} \subseteq [N]} \sum_{\underset{i_s
     \neq i_r \text{ for } s \neq r}{i_1, \ldots, i_{k_1 - m} \in [N ; M]}} \]
  \[ \oint \frac{(1 - z)^{k_1}}{(z - y_{l_0}) \ldots (z - y_{l_m})} 
     \frac{\beta_{i_1} \ldots \beta_{i_{k_1 - m}}}{(1 - \beta_{i_1} z) \ldots
     (1 - \beta_{i_{k_1 - m}} z)} d \text{} z, \]
  where the contour is a circle of radius smaller than 1, that encircles all
  the variables $y_i$. Using that
  \[ \sum_{\underset{i_s \neq i_r \text{ for } s \neq r}{i_1, \ldots, i_{k_1 -
     m} \in [N ; M]}} \frac{\beta_{i_1} \ldots \beta_{i_{k_1 - m}}}{(1 -
     \beta_{i_1} z) \ldots (1 - \beta_{i_{k_1 - m}} z)} = \frac{(- 1)^{k - m}
     \partial_z^{k - m} \left( \prod_{j = N + 1}^M (1 - \beta_j z) \right)
     }{\prod_{j = N + 1}^M (1 - \beta_j z)} \]
  \begin{equation}
    = (- 1)^{k - m} \prod_{j = N + 1}^M (1 - \beta_j z)^{- 1} \partial_z^{k -
    m} \exp \left( - (M - N) \sum_{j = 1}^{\infty} \frac{p_j (\beta_{N + 1},
    \ldots, \beta_M)}{M - N}  \frac{z^j}{j} \right), \label{asx}
  \end{equation}
  we write this sum as a linear combination of products of $\frac{1}{M - N}
  \sum_{j = 1}^{\infty} p_j (\beta_{N + 1}, \ldots, \beta_M) z^{j - 1}$ and
  its derivatives, where the coefficients are monomials of $M - N$. By $p_j$,
  we denote the power sum symmetric polynomial of degree $j$. Similarly, by
  \[ \sum_{\{l_0, \ldots, l_m \} \subseteq [N]} \frac{1}{(z - y_{l_0}) \ldots
     (z - y_{l_m})} = \frac{1}{(m + 1) !}  \prod_{i = 1}^N (z - y_i)^{- 1} 
     \text{{\hspace{11em}}} \]
  \begin{equation}
    \text{{\hspace{14em}}} \times \partial_z^{m + 1} \exp \left( N \text{}
    \log \text{} z - N \sum_{j = 1}^{\infty} \frac{p_j (y_1, \ldots, y_N)}{N} 
    \frac{1}{jz^j} \right), \label{zxcv}
  \end{equation}
  we write this sum as a linear combination of products of $\frac{1}{N}
  \sum_{j = 0}^{\infty} p_j (y_1, \ldots, y_N) z^{- j - 1}$ and its
  derivatives, where the coefficients are monomials of $N$. Since our random
  weights are LLN appropriate, only the summands of (\ref{asx}), (\ref{zxcv})
  that correspond to the largest powers of $M - N$, $N$ respectively, will
  contribute to the limit. For the same reason, the summands $f^{- 1} D_k f$
  of $f^{- 1} \mathcal{D}_{k_1} f$, with $k < k_1$, will not contribute to the
  limit because they are linear combinations of contour integrals $\mathcal{I}
  (\alpha, \gamma)$, where $\alpha + \gamma \leq k_1$. Therefore, we deduce
  that
  \[ \lim_{\underset{N / M \rightarrow \gamma}{N, M \rightarrow \infty}}
     \mathbb{E} \left[ \frac{1}{N^{k_1 + 1}} \sum_{i = 1}^N (\lambda_i + N -
     i)^{k_1} \right] = \lim_{\underset{N / M \rightarrow \gamma}{N, M
     \rightarrow \infty}} \frac{1}{2 \pi i} \oint \sum_{m = 0}^{k_1}
     \binom{k_1}{m} \frac{(- 1)^m}{m + 1} \left( \frac{1}{\gamma} - 1
     \right)^{k_1 - m} \frac{(1 - z)^{k_1}}{z^{k_1 + 1}} \]
  \[ \times \mathbb{E}_{\tmmathbf{\beta}, \tmmathbf{y}} \left[ \left( \sum_{j
     = 1}^{\infty} \frac{p_j (\tmmathbf{\beta}_{N + 1}, \ldots,
     \tmmathbf{\beta}_M)}{M - N} z^j \right)^{k_1 - m} \left( \sum_{j =
     0}^{\infty} \frac{p_j (\tmmathbf{y}_1, \ldots, \tmmathbf{y}_N)}{N} 
     \frac{1}{z^j} \right)^{m + 1} \right] d \text{} z =\mathfrak{m}_{k_1} .
  \]
  For the general case (\ref{genca}), we follow a similar procedure. Namely,
  at the beginning we consider $x_i, \beta_j$ deterministic, and we apply
  $f^{- 1} \mathcal{D}_{k_{\nu}} \ldots \mathcal{D}_{k_1}$ at both sides of
  the Cauchy identity (\ref{CI}). Using (\ref{Stirling}), the right hand side
  of the Cauchy identity will give a linear combination of $f^{- 1}
  D_{l_{\nu}} \ldots D_{l_1} f$, where $l_i \leq k_i$. In Lemma \ref{Lemapent}
  we shown $f^{- 1} D_{l_{\nu}} \ldots D_{l_1} f$ is a sum of contour
  integrals $\mathcal{I} (\{\alpha_i, \gamma_i \}_{i = 1}^{\nu})$. Using
  (\ref{asx}), (\ref{zxcv}), and the LLN appropriateness, we have that only
  the contour integrals with maximal $\sum_{i = 1}^{\nu} \alpha_i + \gamma_i
  \leq \sum_{i = 1}^{\nu} k_i + \nu$ will contribute to the limit. But there
  is only one such integral, it emerges for $l_i = k_i$ and it is equal to
  $\mathcal{F}_{k_1} \ldots \mathcal{F}_{k_{\nu}}$. This proves the claim.
\end{proof}

\begin{example}
  \label{exampl8}In Theorem \ref{LLN!}, under special cases, we can have
  explicit formulas for the functions $\mathbf{F}_1$, $\mathbf{F}_2$ that
  appear in the moments of the limiting measure. The simplest case to consider
  is $\{\tmmathbf{y}_i, \tmmathbf{\beta}_j \}$ to be independent random
  variables, $\{\tmmathbf{\beta}_j \}$ to be identically distributed and
  $\{\tmmathbf{y}_i \}$ to be identically distributed as well. Then, we have
  that
  \[ \mathbf{F}_2 (z) =\mathbb{E}_{\tmmathbf{\beta}_1} \left[ \frac{1}{1
     -\tmmathbf{\beta}_1 z} \right] \text{\quad and\quad} \mathbf{F}_1 (z)
     =\mathbb{E}_{\tmmathbf{y}_1} \left[ \frac{1}{1 -\tmmathbf{y}_1 z}
     \right], \]
  where we assume that the above expectations converge. By the formula
  (\ref{momenh}) for the moments, we get
  \[ \mathfrak{m}_k = \frac{1}{2 \pi i (k + 1)} \oint_{|z| = \varepsilon} 
     \frac{1}{z - 1} \left( \left( \frac{1}{\gamma} - 1 \right)
     \mathbb{E}_{\tmmathbf{\beta}_1} \left[ \frac{\tmmathbf{\beta}_1
     -\tmmathbf{\beta}_1 z}{1 -\tmmathbf{\beta}_1 z} \right]
     +\mathbb{E}_{\tmmathbf{y}_1} \left[ \frac{z - 1}{z -\tmmathbf{y}_1}
     \right] \right)^{k + 1} d \text{} z. \]
  The above contour encircles the poles of $z \mapsto
  \mathbb{E}_{\tmmathbf{y}_1} [(z -\tmmathbf{y}_1)^{- 1}]$. For the case where
  the distribution of $\tmmathbf{y}_1$ is $\delta_0$, we recover the formula
  for the moments of Proposition 5.12 of {\cite{B52}}. The above formula
  provides a generalization of their result.
\end{example}

\begin{remark}
  The conditions in Convention \ref{convention}, for the random parameters of
  Theorem \ref{LLN!}, are important in order to find a contour $C_z$ such that
  the power series $\sum_{j \geq 0} z^{- j} p_j (\tmmathbf{y}_1, \ldots,
  \tmmathbf{y}_N)$, $\sum_{j \geq 0} p_j (\tmmathbf{\beta}_{N + 1}, \ldots,
  \tmmathbf{\beta}_M) z^j$ converge for very $z \in C_z$. For the case where
  the distributions of $\tmmathbf{x}_1, \ldots, \tmmathbf{x}_M$ are compactly
  supported, the same arguments work under slight modifications.
\end{remark}

\section{The Annealed Central Limit Theorem for one level}\label{LLL}

In the current section, we study global fluctuations around the limit shape,
for the annealed case, for the model with random edge weights. This is done
under certain conditions for the random weights. For the Law of Large Numbers,
we proved that if the random empirical measures $M^{- 1} \sum_{i = 1}^M
\delta_{\tmmathbf{\beta}_i}$, $M^{- 1} \sum_{i = 1}^M \delta_{\tmmathbf{x}_i}$
converge to a deterministic limit shape (in the sense of Definition
\ref{Def}), then the random empirical measures $m_N
[\overline{\rho_{\tmmathbf{\beta}, \tmmathbf{y}}}]$ will also converge to a
deterministic limit shape. We will show that the same phenomenon occurs for
the fluctuations, namely if under proper scaling the empirical measures of the
random edge weights have Gaussian fluctuations, then under the same scaling
$m_N [\overline{\rho_{\tmmathbf{\beta}, \tmmathbf{y}}}]$ will have Gaussian
fluctuations as well. We provide explicit formulas for the covariance. We also
consider two different limit regimes, where we have Gaussian fluctuations
under different scalings. These results generalize Theorem 4.5 and Theorem 7.1
of {\cite{B52}}.

\subsection{First limit regime: classical scaling}\label{hjklm}

We start with the first limit regime that corresponds to scaling $\sqrt{N}$.
This is compatible to the classical central limit theorem. The next definition
makes clear the difference between the two limit regimes that we consider.
Before we state it, to simplify notation, in the following given a random
variable $Z$ we will denote $\dot{Z} = Z -\mathbb{E} [Z]$, and $\bar{Z}
=\mathbb{E} [Z]$.

\begin{definition}
  \label{BBBB}Let $\{\tmmathbf{y}_i \}_{i = 1}^M$, $\{\tmmathbf{\beta}_j \}_{j
  = 1}^M$ be sequences of random variables and $\varepsilon \in (0, 1)$. We
  will call such sequences $\varepsilon$-CLT appropriate if the following
  hold:
  \begin{itemize}
    {\item The collection of random variables
    \[ \left\{ \frac{1}{M^{1 - \varepsilon}} \sum_{i = 1}^{\left\lfloor M
       \text{} t \right\rfloor} (\tmmathbf{y}_i^k -\mathbb{E}
       [\tmmathbf{y}_i^k]), \frac{1}{M^{1 - \varepsilon}} \sum_{i =
       1}^{\left\lfloor M \text{} s \right\rfloor} (\tmmathbf{\beta}_i^l
       -\mathbb{E} [\tmmathbf{\beta}_i^l]) \right\}_{0 < s, t \leq 1, \text{ }
       k, l \geq 1} \]
    converges as $M \rightarrow \infty$, in the sense of moments to a Gaussian
    random vector.}
    
    \item Let
    \[ \mathfrak{q}_{k, l}^{(t, s)} \assign \lim_{M \rightarrow \infty}
       \mathbb{E} \left[ \frac{1}{M^{2 - 2 \varepsilon}} \sum_{i =
       1}^{\left\lfloor M \text{} t \right\rfloor} (\tmmathbf{\beta}_i^k
       -\mathbb{E} [\tmmathbf{\beta}_i^k]) \cdot \sum_{i = 1}^{\left\lfloor M
       \text{} s \right\rfloor} (\tmmathbf{\beta}_i^l -\mathbb{E}
       [\tmmathbf{\beta}_i^l]) \right], \]
    \[ \mathfrak{p}_{k, l}^{(t, s)} \assign \lim_{M \rightarrow \infty}
       \mathbb{E} \left[ \frac{1}{M^{2 - 2 \varepsilon}} \sum_{i =
       1}^{\left\lfloor M \text{} t \right\rfloor} (\tmmathbf{y}_i^k
       -\mathbb{E} [\tmmathbf{y}_i^k]) \cdot \sum_{i = 1}^{\left\lfloor M
       \text{} s \right\rfloor} (\tmmathbf{y}_i^l -\mathbb{E}
       [\tmmathbf{y}_i^l]) \right], \]
    \[ \mathfrak{s}_{k, l}^{(t, s)} \assign \lim_{M \rightarrow \infty}
       \mathbb{E} \left[ \frac{1}{M^{2 - 2 \varepsilon}} \sum_{i =
       1}^{\left\lfloor M \text{} t \right\rfloor} (\tmmathbf{\beta}_i^k
       -\mathbb{E} [\tmmathbf{\beta}_i^k]) \cdot \sum_{i = 1}^{\left\lfloor M
       \text{} s \right\rfloor} (\tmmathbf{y}_i^l -\mathbb{E}
       [\tmmathbf{y}_i^l]) \right] . \]
    Then, the power series
    \[ \sum_{k, l \geq 1} \mathfrak{q}_{k, l}^{(t, s)} z^k w^l \text{, \quad}
       \sum_{k, l \geq 1} \mathfrak{p}_{k, l}^{(t, s)} z^k w^l \text{, \quad}
       \sum_{k, l \geq 1} \mathfrak{s}_{k, l}^{(t, s)} z^k w^l \]
    converge uniformly in a neighborhood of $(0, 0)$.
  \end{itemize}
\end{definition}

\subsubsection{Computation of the covariance for the first limit regime}

In the following, our $\varepsilon$-CLT appropriate random parameters
correspond to the cases where $\varepsilon = 1 / 2, 1$. The classical case
that we deal with in the current subsection is $\varepsilon = 1 / 2$. We will
assume that the random parameters of $\overline{\rho_{\tmmathbf{\beta},
\tmmathbf{y}}}$ are $2^{- 1}$-CLT appropriate. Before we show asymptotic
Gaussianity for $m_N [\overline{\rho_{\tmmathbf{\beta}, \tmmathbf{y}}}]$, we
first compute the covariance. The following notation will simplify our
formulas: For the random parameters of $\overline{\rho_{\tmmathbf{\beta},
\tmmathbf{y}}}$, we denote $\tmmathbf{X}_{j, N, M} \assign \frac{1}{M - N} p_j
(\tmmathbf{\beta}_{N + 1}, \ldots, \tmmathbf{\beta}_M)$, and $\tmmathbf{Y}_{j,
N} \assign \frac{1}{N} p_j (\tmmathbf{y}_1, \ldots, \tmmathbf{y}_N)$, where
$\tmmathbf{y}_i = 1 -\tmmathbf{x}_i$ and $N < M$.

\begin{theorem}
  \label{CLT1!}Assume that $\{\tmmathbf{x}_i \}_{i = 1}^M,
  \{\tmmathbf{\beta}_j \}_{j = 1}^M$ are $2^{- 1}$-CLT appropriate sequences
  of random variables. We also denote $\tmmathbf{y}_i = 1 -\tmmathbf{x}_i$,
  and
  \[ \mathbf{G} _1 (z, w) \assign \sum_{i, j \geq 1}
     \mathfrak{q}^{(\gamma)}_{i, j} z^i w^j \text{, \quad} \mathbf{G} _2 (z,
     w) \assign \sum_{i, j \geq 1} \mathfrak{p}^{(\gamma)}_{i, j} z^i w^j
     \text{, \quad} \mathbf{G}_3 (z, w) \assign \sum_{i, j \geq 1}
     \mathfrak{s}^{(\gamma)}_{i, j} z^i w^j, \]
  where
  \[ \mathfrak{q}_{i, j}^{(\gamma)} = \lim_{\underset{N / M \rightarrow
     \gamma}{N, M \rightarrow \infty}} (M - N) \mathbb{E}
     [\dot{\tmmathbf{X}}_{i, N, M}   \dot{\tmmathbf{X}}_{j, N, M}] \text{,
     \quad} \mathfrak{p}_{i, j}^{(\gamma)} = \lim_{\underset{N / M \rightarrow
     \gamma}{N, M \rightarrow \infty}} N\mathbb{E} [\dot{\tmmathbf{Y}}_{i, N} 
     \dot{\tmmathbf{Y}}_{j, N}], \]
  \[ \mathfrak{s}_{i, j}^{(\gamma)} = \lim_{\underset{N / M \rightarrow
     \gamma}{N, M \rightarrow \infty}} \sqrt{(M - N) N} \mathbb{E}
     [\dot{\tmmathbf{X}}_{i, N, M}  \dot{\tmmathbf{Y}}_{j, N}] . \]
  Then, we have that
  \[ \mathfrak{c}_{k_1, k_2} \assign \lim_{\underset{N / M \rightarrow
     \gamma}{N, M \rightarrow \infty}} \frac{1}{N^{k_1 + k_2 + 1}} \mathbb{E}
     [\dot{p_{k_1}} (\lambda_1 + N - 1, \ldots, \lambda_N) \dot{p_{k_2}}
     (\lambda_1 + N - 1, \ldots, \lambda_N)] \text{{\hspace{3em}}} \]
  \[ \text{{\hspace{6em}}} = \frac{- 1}{4 \pi^2} \oint_{|w| = \varepsilon}
     \oint_{|z| = 2 \varepsilon} \frac{1}{zw} \left( \left( \frac{1}{\gamma} -
     1 \right) \left( \frac{1 - z}{z} \mathbf{F}_2 (z) - \frac{1 - z}{z}
     \right) + \frac{z - 1}{z} \mathbf{F}_1 \left( \frac{1}{z} \right)
     \right)^{k_1} \]
  \[ \times \left( \left( \frac{1}{\gamma} - 1 \right) \left( \frac{1 - w}{w}
     \mathbf{F}_2 (w) - \frac{1 - w}{w} \right) + \frac{w - 1}{w} \mathbf{F}_1
     \left( \frac{1}{w} \right) \right)^{k_2} \left( \mathbf{G}_2 \left(
     \frac{1}{z}, \frac{1}{w} \right) - \sqrt{\frac{1}{\gamma} - 1}
     \mathbf{G}_3 \left( w, \frac{1}{z} \right) \right. \]
  \begin{equation}
    \left. - \sqrt{\frac{1}{\gamma} - 1} \mathbf{G}_3 \left( z, \frac{1}{w}
    \right) + \left( \frac{1}{\gamma} - 1 \right) \mathbf{G}_1 (z, w) \right)
    d \text{} z \text{} d \text{} w, \label{vbcg}
  \end{equation}
  where the functions $\mathbf{F}_1, \mathbf{F}_2$ are the same as in Theorem
  \ref{LLN!}.
\end{theorem}

\begin{proof}
  To prove the claim we will use the results of Lemma \ref{Lem2} and Lemma
  \ref{Lemma3}. Similarly to what we did in Theorem \ref{LLN!}, at the
  beginning, we consider $x_1, \ldots, x_N, \beta_{N + 1}, \ldots, \beta_M$
  deterministic and at the end we will consider the expectation with respect
  to their joint distribution. As we saw, we can get an expression for
  $\mathbb{E} [p_{k_1} (\lambda_1 + N - 1, \ldots, \lambda_N) p_{k_2}
  (\lambda_1 + N - 1, \ldots, \lambda_N)]$ computing $f^{- 1}
  \mathcal{D}_{k_2} \mathcal{D}_{k_1} f$. First, we only consider its summand
  $f^{- 1} D_{k_2} D_{k_1} f$ and we will show that the covariance is due to
  $\mathcal{F}_{k_1} \mathcal{F}_{k_2}$. More precisely, (\ref{vbcg}) emerges
  from the expectation of the product
  \[ \prod_{i = 1}^2 \left[ \frac{N^{k_i + 1}}{2 \pi i} \sum_{m = 0}^{k_i}
     \binom{k_i}{m} \frac{(- 1)^m}{m + 1} \oint (1 - z)^{k_i} \text{$\left(
     \left( \frac{1}{\gamma} - 1 \right) \sum_{j = 1}^{\infty} X_{j, N, M}
     z^{j - 1} \right)^{k_i - m} \left( \sum_{j = 0}^{\infty} \frac{Y_{j,
     N}}{z^{j + 1}} \right)^{m + 1} d \text{} z$} \right] . \]
  The contours are taken as in the proof of Theorem \ref{LLN!}, to guarantee
  that the power series converge. Note that above we replaced $(M - N) / N$ by
  $\gamma^{- 1} - 1$, since at the end we will have $N, M \rightarrow \infty$.
  Writing $\tmmathbf{X}_{j, N, M} = \dot{\tmmathbf{X}}_{j, N, M} +
  \bar{\tmmathbf{X}}_{j, N, M}$, $\tmmathbf{Y}_{j, N} = \dot{\tmmathbf{Y}}_{j,
  N} + \bar{\tmmathbf{Y}}_{j, N}$ and using the binomial theorem, the
  expectation of the above product becomes an explicit linear combination of
  double contour integrals
  \[ N^{k_1 + k_2 + 2} \oint \oint (1 - z_1)^{k_1} (1 - z_2)^{k_2}
     \mathbb{E}_{\tmmathbf{\beta}, \tmmathbf{y}} \left[ \prod_{i = 1}^2 \left(
     \sum_{j = 1}^{\infty} \dot{\tmmathbf{X}}_{j, N, M} z_i^{j - 1}
     \right)^{\eta_i} \left( \sum_{j = 1}^{\infty}
     \frac{\dot{\tmmathbf{Y}}_{j, N}}{z_i^{j + 1}} \right)^{\theta_i} \right]
  \]
  \begin{equation}
    \times \prod_{i = 1}^2 \left( \sum_{j = 1}^{\infty} \bar{\tmmathbf{X}}_{j,
    N, M} z_i^{j - 1} \right)^{k_i - m_i - \eta_i} \left( \sum_{j =
    0}^{\infty} \frac{\bar{\tmmathbf{Y}}_{j, N}}{z_i^{j + 1}} \right)^{m_i + 1
    - \theta_i} d \text{} z \text{}_1 d \text{} z_2 . \label{bnmj}
  \end{equation}
  We will show that the desired limit comes only from such contour integrals
  that correspond to $\eta_1 + \theta_1 = 1 = \eta_2 + \theta_2$. We will only
  consider the cases where $\eta_1 = \eta_2 = 1$, $\theta_1 = \theta_2 = 0$,
  and $\eta_1 = \theta_2 = 1$, $\eta_2 = \theta_1 = 0$, in order to compute
  the asymptotic contribution of (\ref{bnmj}). In the first case, (\ref{bnmj})
  gives the summand of (\ref{vbcg}) that corresponds to $\mathbf{G}_1$, and in
  the second case, (\ref{bnmj}) gives the summand of (\ref{vbcg}) that
  corresponds to $\mathbf{G}_3 (z, w^{- 1})$. Similarly, for $\eta_2 =
  \theta_1 = 1$, we get the summand of (\ref{vbcg}) that corresponds to
  $\mathbf{G}_3 (w, z^{- 1})$, and for $\theta_1 = \theta_2 = 1$, we get the
  summand of (\ref{vbcg}) that corresponds to $\mathbf{G}_2$.
  
  Let $\eta_1 = \eta_2 = 1$, $\theta_1, \theta_2 = 0$. Then, the $2^{- 1}$-CLT
  appropriateness implies that the normalized product of expectations
  converges to
  \[ \frac{(\gamma^{- 1} - 1)^{- 1}}{z^{k_1 + 1} w^{k_2 + 1}} \mathbf{G}_1 (z,
     w) (\mathbf{F}_2 (z) - 1)^{k_1 - m_1 - 1} (\mathbf{F}_2 (w) - 1)^{k_2 -
     m_2 - 1} \left( \mathbf{F}_1 \left( \frac{1}{z} \right) \right)^{m_1 + 1}
     \left( \mathbf{F}_1 \left( \frac{1}{w} \right) \right)^{m_2 + 1} . \]
  Therefore, the limit of the corresponding linear combination is
  \[ \frac{- 1}{4 \pi^2} \sum_{m_1 = 0}^{k_1 - 1} \sum_{m_2 = 0}^{k_2 - 1}
     \binom{k_1}{m_1} \binom{k_2}{m_2} \frac{(- 1)^{m_1 + m_2} (k_1 - m_1)
     (k_2 - m_2)}{(m_1 + 1) (m_2 + 1)} (\gamma^{- 1} - 1)^{k_1 + k_2 - m_1 -
     m_2 - 1} \]
  \[ \oint \oint \frac{(1 - z)^{k_1} (1 - w)^{k_2}}{z^{k_1 + 1} w^{k_2 + 1}}
     \mathbf{G}_1 (z, w) (\mathbf{F}_2 (z) - 1)^{k_1 - m_1 - 1} (\mathbf{F}_2
     (w) - 1)^{k_2 - m_2 - 1} \]
  \[ \left( \mathbf{F}_1 \left( \frac{1}{z} \right) \right)^{m_1 + 1} \left(
     \mathbf{F}_1 \left( \frac{1}{w} \right) \right)^{m_2 + 1} d \text{} z
     \text{} d \text{} w, \]
  which is equal to the term of (\ref{vbcg}) that depends on $\mathbf{G}_1$.
  
  Now, let $\eta_1 = \theta_2 = 1$, $\eta_2 = \theta_1 = 0$. Again, the $2^{-
  1}$-CLT appropriateness implies that the normalized product of expectations
  in (\ref{bnmj}) converges to
  \[ \frac{(\gamma^{- 1} - 1)^{- 1 / 2}}{z_1^{k_1 + 1} z_2^{k_2 + 1}}
     \mathbf{G}_3 (z, w^{- 1}) (\mathbf{F}_2 (z) - 1)^{k_1 - m_1 - 1}
     (\mathbf{F}_2 (w) - 1)^{k_2 - m_2} \left( \mathbf{F}_1 \left( \frac{1}{z}
     \right) \right)^{m_1 + 1} \left( \mathbf{F}_1 \left( \frac{1}{w} \right)
     \right)^{m_2}, \]
  and the limit of the corresponding linear combination is
  \[ \frac{- 1}{4 \pi^2} \sum_{m_1 = 0}^{k_1 - 1} \sum_{m_2 = 0}^{k_2}
     \binom{k_1}{m_1} \binom{k_2}{m_2} \frac{(- 1)^{m_1 + m_2} (k_1 -
     m_1)}{m_1 + 1} (\gamma^{- 1} - 1)^{k_1 + k_2 - m_1 - m_2 - 1 / 2} \]
  \[ \oint \oint \frac{(1 - z_1)^{k_1} (1 - z_2)^{k_2}}{z_1^{k_1 + 1} z_2^{k_2
     + 1}} \mathbf{G}_3 (z, w^{- 1}) (\mathbf{F}_2 (z) - 1)^{k_1 - m_1 - 1}
     (\mathbf{F}_2 (w) - 1)^{k_2 - m_2} \]
  \[ \left( \mathbf{F}_1 \left( \frac{1}{z} \right) \right)^{m_1 + 1} \left(
     \mathbf{F}_1 \left( \frac{1}{w} \right) \right)^{m_2} d \text{} z \text{}
     d \text{} w, \]
  which is equal to the term of (\ref{vbcg}) that corresponds to $\mathbf{G}_3
  (z, w^{- 1})$.
  
  To conclude the proof, we have to show that all the other terms of $f^{- 1}
  \mathcal{D}_{k_2} \mathcal{D}_{k_1} f$ and $- f^{- 2} \mathcal{D}_{k_1} f
  \cdot \mathcal{D}_{k_2} f$, either vanish as $N, M \rightarrow \infty$, or
  cancel out. This is done in detail and in the most general case, in the
  proof of Theorem \ref{normal1!}.
\end{proof}

\subsubsection{Asymptotic normality}\label{amrtli}

Now, we focus on proving asymptotic normality for $\{ N^{- (k + 1 / 2)}
(p_k^{(N, M)} -\mathbb{E} [p_k^{(N, M)}]) \}_{k \in \mathbb{N}}$, where
\[ p_k^{(N, M)} \assign \sum_{i = 1}^N (\lambda_i + N - i)^k, \]
and the random partition $(\lambda_1 \geq \cdots \geq \lambda_N)$ is
distributed according to $\overline{\rho_{\tmmathbf{\beta}, \tmmathbf{y}}}$.
This is done assuming $2^{- 1}$-CLT appropriateness for the random parameters
of the Schur measure. We will identify the Gaussian distribution by computing
the limiting moments and interpreting via the Wick's formula. We recall that
$\dot{p}_k^{(N, M)} \assign p_k^{(N, M)} -\mathbb{E} [p_k^{(N, M)}]$.

\begin{theorem}
  \label{normal1!}Assume that $\{\tmmathbf{x}_i \}_{i = 1}^M,
  \{\tmmathbf{\beta}_j \}_{j = 1}^M$ are $2^{- 1}$-CLT appropriate sequences
  of random variables. We also denote $\tmmathbf{y}_i = 1 -\tmmathbf{x}_i$,
  and we assume that the asymptotic conditions of Theorem \ref{CLT1!} are
  satisfied. Then, for any positive integers $k_1, \ldots, k_{\nu}$ we have
  \[ \lim_{\underset{N / M \rightarrow \gamma}{N, M \rightarrow \infty}}
     \frac{N^{- \nu / 2}}{{N^{k_1 + \cdots + k_{\nu}}} } \mathbb{E}
     [\dot{p}_{k_1}^{(N, M)} \ldots \dot{p}_{k_{\nu}}^{(N, M)}] = 0, \]
  if $\nu$ is odd, and
  \begin{equation}
    \lim_{\underset{N / M \rightarrow \gamma}{N, M \rightarrow \infty}}
    \frac{N^{- \nu / 2}}{{N^{k_1 + \cdots + k_{\nu}}} } \mathbb{E}
    [\dot{p}_{k_1}^{(N, M)} \ldots \dot{p}_{k_{\nu}}^{(N, M)}] = \sum_{\pi \in
    P_2 (\nu)} \prod_{(i, j) \in \pi} \mathfrak{c}_{k_i, k_j}, \label{limit!!}
  \end{equation}
  if $\nu$ is even, where $P_2 (\nu)$ is the set of all pairings of $[\nu]$,
  and $\mathfrak{c}_{k, l}$ are the same as in Theorem \ref{CLT1!}.
\end{theorem}

Before we prove the theorem we want to emphasize on how the desired limit
emerges. Expressions for the left-hand side of (\ref{limit!!}) can be obtained
by applying the differential operators $\mathcal{D}_k$ (that act on $x_i$) to
$f = \prod_{j = N + 1}^M \prod_{i = 1}^N (1 - \beta_j + \beta_j x_i)$, which
gives rise to contour integrals $\mathcal{I} (\{\alpha_i, \gamma_i \}_{i =
1}^n)$. Using the formulas (\ref{asx}), (\ref{zxcv}), and differentiating, we
can write $\mathcal{I} (\{\alpha_i, \gamma_i \}_{i = 1}^n)$ in a form where
the $2^{- 1}$-CLT appropriateness can be used. This differentiation creates
powers of $N$. The maximum power of $N$ emerges when $\sum_{i = 1}^n (\alpha_i
+ \gamma_i)$ is maximized. It is reasonable to expect that the limit comes
from the terms that give the maximum power of $N$.

For $f^{- 1} \mathcal{D}_{k_{\nu}} \ldots \mathcal{D}_{k_1} f$, due to Lemma
\ref{Lemapent}, we see that such terms are given by $\mathcal{F}_{k_1} \ldots
\mathcal{F}_{k_{\nu}}$, and more specifically, they are the product of
\[ F_k \assign \frac{N^{k + 1}}{2 \pi i} \oint \sum_{m = 0}^k \binom{k}{m}
   \frac{(- 1)^m}{m + 1} \left( \frac{1}{\gamma} - 1 \right)^{k - m} \left(
   \sum_{j = 1}^{\infty} X_{j, N, M} z^{j - 1} \right)^{k - m} \left( \sum_{j
   = 0}^{\infty}  \frac{Y_{j, N}}{z^{j + 1}} \right)^{m + 1} d \text{} z, \]
for $k = k_1, \ldots, k_{\nu}$. The maximum power of $N$ is $N^{k_1 + \cdots +
k_{\nu} + \nu}$.

\begin{proposition}
  \label{CLTProp}Assume that $\{\tmmathbf{x}_i \}_{i = 1}^M,
  \{\tmmathbf{\beta}_j \}_{j = 1}^M$ are $2^{- 1}$-CLT appropriate sequences
  of random variables. We also denote $\tmmathbf{y}_i = 1 -\tmmathbf{x}_i$,
  and we assume that the asymptotic conditions of Theorem \ref{CLT1!} are
  satisfied. Then, for any positive integers $k_1, \ldots, k_{\nu}$ we have
  \[ \lim_{\underset{N / M \rightarrow \gamma}{N, M \rightarrow \infty}}
     \frac{N^{- \nu / 2}}{{N^{k_1 + \cdots + k_{\nu}}} }
     \mathbb{E}_{\tmmathbf{\beta}, \tmmathbf{y}} [\dot{F}_{k_1} \ldots
     \dot{F}_{k_{\nu}}] = 0, \]
  if $\nu$ is odd, and
  \begin{equation}
    \lim_{\underset{N / M \rightarrow \gamma}{N, M \rightarrow \infty}}
    \frac{N^{- \nu / 2}}{{N^{k_1 + \cdots + k_{\nu}}} }
    \mathbb{E}_{\tmmathbf{\beta}, \tmmathbf{y}} [\dot{F}_{k_1} \ldots
    \dot{F}_{k_{\nu}}] = \sum_{\pi \in P_2 (\nu)} \prod_{(i, j) \in \pi}
    \mathfrak{c}_{k_i, k_j}, \label{D}
  \end{equation}
  if $\nu$ is even, where $P_2 (\nu)$ is the set of all pairings of $[\nu]$,
  and $\mathfrak{c}_{k, l}$ are the same as in Theorem \ref{CLT1!}.
\end{proposition}

\begin{proof}
  We will show that the limit is given only by specific terms of
  $\mathbb{E}_{\tmmathbf{\beta}, \tmmathbf{y}} [F_{k_1} \ldots F_{k_{\nu}}]$.
  Writing $\tmmathbf{X}_{j, N, M} = \dot{\tmmathbf{X}}_{j, N, M} +
  \bar{\tmmathbf{X}}_{j, N, M}$, $\tmmathbf{Y}_{j, N} = \dot{\tmmathbf{Y}}_{j,
  N} + \bar{\tmmathbf{Y}}_{j, N}$, and applying the binomial theorem, we have
  that $\mathbb{E}_{\tmmathbf{\beta}, \tmmathbf{y}} [F_{k_1} \ldots
  F_{k_{\nu}}]$ is a linear combination of contour integrals
  \[ \frac{N^{k_1 + \cdots + k_{\nu} + \nu}}{(2 \pi \mathi)^{\nu}} \oint
     \ldots \oint \mathbb{E}_{\tmmathbf{\beta}, \tmmathbf{y}} \left[ \prod_{i
     = 1}^{\nu} \left( \left( \frac{1}{\gamma} - 1 \right) \sum_{j =
     1}^{\infty} \dot{\tmmathbf{X}}_{j, N, M} z_i^{j - 1} \right)^{\eta_i}
     \left( \sum_{j = 1}^{\infty} \frac{\dot{\tmmathbf{Y}}_{j, N}}{z_i^{j +
     1}} \right)^{\theta_i} \right] \]
  \begin{equation}
    \text{{\hspace{4em}}} \times \prod_{i = 1}^{\nu} (1 - z_i)^{k_i} \left(
    \left( \frac{1}{\gamma} - 1 \right) \sum_{j = 1}^{\infty}
    \bar{\tmmathbf{X}}_{j, N, M} z_i^{j - 1} \right)^{k_i - m_i - \eta_i}
    \left( \sum_{j = 0}^{\infty} \frac{\bar{\tmmathbf{Y}}_{j, N}}{z_i^{j + 1}}
    \right)^{m_i + 1 - \theta_i} d \text{} z \text{}_1 \ldots d \text{}
    z_{\nu} . \label{A}
  \end{equation}
  The linear combinations of (\ref{A}), where $\eta_i + \theta_i = 1$ for
  every $i$, will give the desired limit. Taking into account all the possible
  cases for $\{(\eta_i, \theta_i)\}_{i = 1}^{\nu}$, such that $\eta_i +
  \theta_i = 1$, we see that our explicit linear combination is
  \[ \frac{N^{k_1 + \cdots + k_{\nu} + \nu}}{(2 \pi \mathi)^{\nu}} \oint
     \ldots \oint \sum_{m_1 = 0}^{k_1} \ldots \sum_{m_{\nu} = 0}^{k_{\nu}}
     \left( \prod_{i = 1}^{\nu} \left( \frac{1}{\gamma} - 1 \right)^{k_i -
     m_i} \binom{k_i}{m_i} \frac{(- 1)^{m_i}}{m_i + 1} \right)  \sum_{\lambda
     = 0}^{\nu} \sum_{\{\pi (i)\}_{i = 1}^{\lambda}, \{\sigma (i)\}_{i =
     1}^{\nu - \lambda}} \]
  \[ \prod_{i = 1}^{\lambda} (k_{\pi (i)} - m_{\pi (i)})  \prod_{i = 1}^{\nu -
     \lambda} (m_{\sigma (i)} + 1) \mathbb{E}_{\tmmathbf{\beta}, \tmmathbf{y}}
     \left[ \prod_{i = 1}^{\lambda} \left( \sum_{j = 1}^{\infty}
     \dot{\tmmathbf{X}}_{j, N, M} z_{\pi (i)}^{j - 1} \right)  \prod_{i =
     1}^{\nu - \lambda} \left( \sum_{j = 1}^{\infty}
     \frac{\dot{\tmmathbf{Y}}_{j, N}}{z_{\sigma (i)}^{j + 1}} \right) \right]
  \]
  \[ \times \prod_{i = 1}^{\lambda} \left( \sum_{j = 1}^{\infty}
     \bar{\tmmathbf{X}}_{j, N, M} z_{\pi (i)}^{j - 1} \right)^{k_{\pi (i)} -
     m_{\pi (i)} - 1} \left( \sum_{j = 0}^{\infty}
     \frac{\bar{\tmmathbf{Y}}_{j, N}}{z_{\pi (i)}^{j + 1}} \right)^{m_{\pi
     (i)} + 1} \]
  \begin{equation}
    \times \prod_{i = 1}^{\nu - \lambda} \left( \sum_{j = 1}^{\infty}
    \bar{\tmmathbf{X}}_{j, N, M} z_{\sigma (i)}^{j - 1} \right)^{k_{\sigma
    (i)} - m_{\sigma (i)}} \left( \sum_{j = 0}^{\infty}
    \frac{\bar{\tmmathbf{Y}}_{j, N}}{z_{\sigma (i)}^{j + 1}}
    \right)^{m_{\sigma (i)}}  \prod_{i = 1}^{\nu} (1 - z_i)^{k_i} d \text{}
    z_1 \ldots d \text{} z_{\nu}, \label{B}
  \end{equation}
  where the last sum in the first line is taken with respect to all disjoint
  sets $\{\pi (1) < \cdots < \pi (\lambda)\}$, $\{\sigma (1) < \cdots < \sigma
  (\nu - \lambda)\}$, such that their union is equal to $[\nu]$. An elementary
  computation shows that (\ref{B}) is equal to
  \[ \frac{N^{k_1 + \cdots + k_{\nu} + \nu}}{(2 \pi \mathi)^{\nu}} \oint
     \ldots \oint \prod_{i = 1}^{\nu} (1 - z_i)^{k_i} \left( \left(
     \frac{1}{\gamma} - 1 \right) \sum_{j = 1}^{\infty} \bar{\tmmathbf{X}}_{j,
     N, M} z_i^{j - 1} - \sum_{j = 0}^{\infty} \frac{\bar{\tmmathbf{Y}}_{j,
     N}}{z_i^{j + 1}} \right)^{k_i} \]
  \begin{equation}
    \times \mathbb{E}_{\tmmathbf{\beta}, \tmmathbf{y}} \left[ \prod_{i =
    1}^{\nu} \left( \sum_{j = 1}^{\infty} \frac{\dot{\tmmathbf{Y}}_{j,
    N}}{z_i^{j + 1}} - \left( \frac{1}{\gamma} - 1 \right) \sum_{j =
    1}^{\infty} \dot{\tmmathbf{X}}_{j, N, M} z_i^{j - 1} \right) \right] d
    \text{} z_1 \ldots d \text{} z_{\nu} . \label{C}
  \end{equation}
  Due to the $2^{- 1}$-CLT appropriateness, and the computation for the
  covariance that we did in Theorem \ref{CLT1!}, we deduce that the contour
  integral (\ref{C}) multiplied by $N^{- \nu / 2} / N^{k_1 + \cdots +
  k_{\nu}}$ will converge as $N, M \rightarrow \infty$, either to $0$, for
  $\nu$ odd, or to the right hand side of (\ref{D}), for $\nu$ even.
  
  To conclude the proof, we have to show that only (\ref{C}) contributes to
  the limit. To do so, instead of $\mathbb{E}_{\tmmathbf{\beta}, \tmmathbf{y}}
  [\dot{F}_{k_1} \ldots \dot{F}_{k_{\nu}}]$, we will consider
  \[ \mathbb{E}_{\tmmathbf{\beta}, \tmmathbf{y}}
     [(\dot{F}_{\kappa_1})^{\tau_1} \ldots (\dot{F}_{\kappa_r})^{\tau_r}] =
     \sum_{\sigma_1 = 0}^{\tau_1} \ldots \sum_{\sigma_r = 0}^{\tau_r} (-
     1)^{\tau_1 - \sigma_1} \ldots (- 1)^{\tau_r - \sigma_r}
     \binom{\tau_1}{\sigma_1} \ldots \binom{\tau_r}{\sigma_r} \]
  \begin{equation}
    \times \mathbb{E}_{\tmmathbf{\beta}, \tmmathbf{y}}
    [(F_{\kappa_1})^{\sigma_1} \ldots (F_{\kappa_r})^{\sigma_r}]
    (\bar{F}_{\kappa_1})^{\tau_1 - \sigma_1} \ldots
    (\bar{F}_{\kappa_r})^{\tau_r - \sigma_r}, \label{S}
  \end{equation}
  where $\kappa_i \neq \kappa_j$ for every $i \neq j$. Using an inductive
  argument for $\sigma_1 + \cdots + \sigma_r \leq \tau_1 + \cdots + \tau_r$,
  we will show that only the analogous of (\ref{C}) contributes to the limit.
  
  Our first inductive step will be for $\sigma_1 + \cdots + \sigma_r = \tau_1
  + \cdots + \tau_r$. Our goal is to show that the linear combinations of
  \[ \frac{N^{\tau_1 (\kappa_1 + 1) + \cdots + \tau_r (\kappa_r + 1)}}{(2 \pi
     i)^{\tau_1 + \cdots + \tau_r}} \oint \ldots \oint
     \mathbb{E}_{\tmmathbf{\beta}, \tmmathbf{y}} \left[ \prod_{i = 1}^r
     \prod_{j = 1}^{\tau_i} \left( \left( \frac{1}{\gamma} - 1 \right) \sum_{l
     = 1}^{\infty} \dot{\tmmathbf{X}}_{l, N, M} z_{i, j}^{l - 1}
     \right)^{\eta_{i, j}} \left( \sum_{l = 1}^{\infty}
     \frac{\dot{\tmmathbf{Y}}_{l, N}}{z_{i, j}^{l + 1}} \right)^{\theta_{i,
     j}} \right] \]
  \begin{equation}
    \times \prod_{i = 1}^r \prod_{j = 1}^{\tau_i} \left( \left(
    \frac{1}{\gamma} - 1 \right) \sum_{l = 1}^{\infty} \bar{\tmmathbf{X}}_{l,
    N, M} z_{i, j}^{l - 1} \right)^{\kappa_i - m_{i, j} - \eta_{i, j}} \left(
    \sum_{l = 0}^{\infty} \frac{\bar{\tmmathbf{Y}}_{l, N}}{z_{i, j}^{l + 1}}
    \right)^{m_{i, j} + 1 - \theta_{i, j}} (1 - z_{i, j})^{\kappa_i} d \text{}
    z_{i, j}, \label{E}
  \end{equation}
  where $\eta_{i, j} + \theta_{i, j} \neq 1$ for some $(i, j)$, do not
  contribute to the limit. We recall that we have to multiply (\ref{E}) by
  \[ N^{(\tau_1 + \cdots + \tau_r) / 2} N^{- \tau_1 (\kappa_1 + 1) - \cdots -
     \tau_r (\kappa_r + 1)}, \]
  before we let $N, M \rightarrow \infty$. It is immediate to see that if
  $\eta_{i, j} + \theta_{i, j} > 1$, for every $i, j$, then the $2^{- 1}$-CLT
  appropriateness guarantees that (\ref{E}) will converge to zero as $N, M
  \rightarrow \infty$. The problematic case is when a lot of the tuples
  $(\eta_{i, j}, \theta_{i, j})$ are equal to $(0, 0)$. Note that this is a
  problem because in the expectation of the product in (\ref{E}), we will have
  less than $\tau_1 + \cdots + \tau_r$ factors. As a consequence, the power
  $N^{(\tau_1 + \cdots + \tau_r) / 2}$ will not be absorbed (by the $2^{-
  1}$-CLT appropriateness) and these terms will diverge. Therefore, our goal
  is to show that such terms cancel out in (\ref{S}).
  
  Let's assume that $\{ (\eta_{1, 1}, \theta_{1, 1}), \ldots, (\eta_{1, n_1},
  \theta_{1, n_1})\}, \ldots, \{(\eta_{r, 1}, \theta_{r, 1}), \ldots,
  (\eta_{r, n_r}, \theta_{r, n_r})\}$ are equal to $(0, 0)$, where $n_i \leq
  \tau_i$. All the other tuples are fixed and not equal to $(0, 0)$. Let's
  denote by $V$ the linear combination of all such terms (\ref{E}), that
  appear in $\mathbb{E}_{\tmmathbf{\beta}, \tmmathbf{y}}
  [(F_{\kappa_1})^{\tau_1} \ldots (F_{\kappa_r})^{\tau_r}]$. But, we also have
  that $V$ appears in (\ref{S}), from terms that correspond to $\sigma_i \geq
  \tau_i - n_i$. This can be done choosing all the corresponding tuples
  $(\eta, \theta)$ of $(\bar{F}_{\kappa_1})^{\tau_1 - \sigma_1} \ldots
  (\bar{F}_{\kappa_r})^{\tau_r - \sigma_r}$, to be equal to $(0, 0)$, and some
  tuples $(\eta, \theta)$ that correspond to $\mathbb{E}_{\tmmathbf{\beta},
  \tmmathbf{y}} [(F_{\kappa_1})^{\sigma_1} \ldots (F_{\kappa_r})^{\sigma_r}]$,
  to be equal to $(0, 0)$ as well. As a corollary, the total number of times
  that $V$ appears in (\ref{S}) is equal to
  \[ \sum_{\sigma_1 = \tau_1 - n_1}^{\tau_1} \ldots \sum_{\sigma_r = \tau_r -
     n_r}^{\tau_r} (- 1)^{\tau_1 - \sigma_1} \ldots (- 1)^{\tau_r - \sigma_r}
     \binom{\tau_1}{\sigma_1} \ldots \binom{\tau_r}{\sigma_r}
     \binom{\sigma_1}{n_1 - (\tau_1 - \sigma_1)} \ldots \binom{\sigma_r}{n_r -
     (\tau_r - \sigma_r)} = 0, \]
  because
  \[ \sum_{\sigma = \tau - n}^{\tau} (- 1)^{\tau - \sigma}
     \binom{\tau}{\sigma} \binom{\sigma}{n - (\tau - \sigma)} =
     \binom{\tau}{n} (- 1)^n \sum_{\sigma = 0}^n \binom{n}{\sigma} (-
     1)^{\sigma} = 0 \text{, \quad for } n \neq 0. \]
  This shows that all (\ref{E}) that are not such that $\eta_{i, j} +
  \theta_{i, j} = 1$, do not contribute to the limit.
  
  To better illustrate how our inductive argument works, and what we have to
  show at each step, we also treat in detail the second step $\sigma_1 +
  \cdots + \sigma_r = \tau_1 + \cdots + \tau_r - 1$. Without loss of
  generality, let $\sigma_1 = \tau_1 - 1$ and $\sigma_i = \tau_i$, for $i \neq
  1$. Again, we will show cancellations, under certain assumptions for the tuples
  $\{(\eta_{1, j}, \theta_{1, j})\}_{j = 1}^{\tau_1 - 1}, {\{(\eta_{2,
  j}, \theta_{2, j})\}_{j = 1}^{\tau_2}} , \ldots, \{(\eta_{r, j}, \theta_{r,
  j})\}_{j = 1}^{\tau_r}$, of $\mathbb{E}_{\tmmathbf{\beta},
  \tmmathbf{y}} [(F_{\kappa_1})^{\tau_1 - 1} (F_{\kappa_2})^{\tau_2} \ldots
  (F_{\kappa_r})^{\tau_r}]$, and $(\eta, \theta)$ (that corresponds to
  $\bar{F}_{\kappa_1}$). Of course, such tuples are obtained in the same way
  as before. Again, the problem is when some of these tuples are equal to $(0,
  0)$. Note that for the case where $(\eta, \theta) = (0, 0)$, we have shown
  that such linear combinations cancel out, by the first inductive step.
  Therefore, it suffices to consider the case where some of the other tuples
  are equal to $(0, 0)$, and $(\eta, \theta) \neq (0, 0)$ is fixed. By abuse
  of notation, assume that the tuples $\{(\eta_{1, 1}, \theta_{1, 1}), \ldots,
  (\eta_{1, n_1}, \theta_{1, n_1})\}, \ldots, \{(\eta_{r, 1}, \theta_{r, 1}),
  \ldots, (\eta_{r, n_r}, \theta_{r, n_r})\}$ are equal to $(0, 0)$, where
  $n_1 \leq \tau_1 - 1$, and $n_i \leq \tau_i$ for $i \neq 1$. All the other
  tuples are fixed and not equal to $(0, 0)$. Then, the corresponding linear
  combinations of products of analogous contour integrals (\ref{E}), also
  appear in (\ref{S}) for $\tau_1 - n_1 - 1 \leq \sigma_1 \leq \tau_1 - 1$,
  $\tau_i - n_i \leq \sigma_i$, fixing appropriately the corresponding tuples
  $(\eta, \theta)$. The total number of times that such linear combinations
  appear in (\ref{S}) is equal to
  \[ \sum_{\sigma_1 = \tau_1 - 1 - n_1}^{\tau_1 - 1} \sum_{\sigma_2 = \tau_2 -
     n_2}^{\tau_2} \ldots \sum_{\sigma_r = \tau_r - n_r}^{\tau_r} (-
     1)^{\tau_1 - \sigma_1} \ldots (- 1)^{\tau_r - \sigma_r}
     \binom{\tau_1}{\sigma_1} \ldots \binom{\tau_r}{\sigma_r} (\tau_1 -
     \sigma_1) \binom{\sigma_1}{n_1 - (\tau_1 - \sigma_1 - 1)} \]
  \[ \times \binom{\sigma_2}{n_2 - (\tau_2 - \sigma_2)} \ldots
     \binom{\sigma_r}{n_r - (\tau_r - \sigma_r)} = 0. \]
  Concluding our inductive argument, assume that for every $l < m \leq \tau_1
  + \cdots + \tau_r$, we have cancellations when tuples $(\eta, \theta)$ are
  equal to $(0, 0)$. Let $s_1, \ldots, s_r$ be arbitrary positive integers,
  such that $s_1 + \cdots + s_r = l$. We have to show cancellations for the
  summands of
  \begin{equation}
    \mathbb{E}_{\tmmathbf{\beta}, \tmmathbf{y}} [(F_{\kappa_1})^{s_1} \ldots
    (F_{\kappa_r})^{s_r}] (\bar{F}_{\kappa_1})^{\tau_1 - s_1} \ldots
    (\bar{F}_{\kappa_r})^{\tau_r - s_r}, \label{Z}
  \end{equation}
  when some of the $\tau_1 + \cdots + \tau_r$ tuples $(\eta, \theta)$ are
  equal to $(0, 0)$. Due to our inductive argument, we only have to consider
  that some of the tuples of $\mathbb{E}_{\tmmathbf{\beta}, \tmmathbf{y}}
  [(F_{\kappa_1})^{s_1} \ldots (F_{\kappa_r})^{s_r}]$ are equal to $(0, 0)$.
  The tuples of $(\bar{F}_{\kappa_i})^{\tau_i - s_i}$ are fixed, and not equal
  to $(0, 0)$. By abuse of notation, assume that only the tuples $\{(\eta_{1,
  1}, \theta_{1, 1}), \ldots, (\eta_{1, n_1}, \theta_{1, n_1})\}, \ldots,
  \{(\eta_{r, 1}, \theta_{r, 1}), \ldots, (\eta_{r, n_r}, \theta_{r, n_r})\}$
  are equal to $(0, 0)$, where $n_i \leq s_i$. Then, the corresponding linear
  combinations that we get from (\ref{Z}) also appear in (\ref{S}) for $s_i -
  n_i \leq \sigma_i \leq s_i$, and the total number of times that they appear
  is
  \[ \sum_{\sigma_1 = s_1 - n_1}^{s_1} \ldots \sum_{\sigma_r = s_r -
     n_r}^{s_r} \prod_{i = 1}^r (- 1)^{\tau_i - \sigma_i}
     \binom{\tau_i}{\sigma_i} \binom{\tau_i - \sigma_i}{\tau_i - s_i}
     \binom{\sigma_i}{n_i - (s_i - \sigma_i)} = 0, \]
  for $n_i \neq 0$. This proves our induction hypothesis and our claim.
\end{proof}

Now, we prove Theorem \ref{normal1!}.

\begin{proof}[of Theorem \ref{normal1!}]
  Instead of $\mathbb{E} [\dot{p}_{k_1}^{(N, M)} \ldots \dot{p}_{k_{\nu}}^{(N,
  M)}]$, we will consider
  \[ \mathbb{E} [(\dot{p}_{\kappa_1}^{(N, M)})^{\tau_1} \ldots
     (\dot{p}_{\kappa_r}^{(N, M)})^{\tau_r}] = \sum_{\sigma_1 = 0}^{\tau_1}
     \ldots \sum_{\sigma_r = 0}^{\tau_r} (- 1)^{\tau_1 - \sigma_1} \ldots (-
     1)^{\tau_r - \sigma_r} \binom{\tau_1}{\sigma_1} \ldots
     \binom{\tau_r}{\sigma_r} \]
  \begin{equation}
    \times \mathbb{E} [(p_{\kappa_1}^{(N, M)})^{\sigma_1} \ldots
    (p_{\kappa_r}^{(N, M)})^{\sigma_r}] (\bar{p}_{\kappa_1}^{(N, M)})^{\tau_1
    - \sigma_1} \ldots (\bar{p}_{\kappa_r}^{(N, M)})^{\tau_r - \sigma_r},
    \label{MN}
  \end{equation}
  where $\kappa_i \neq \kappa_j$, for $i \neq j$. We recall that an expression
  for such expectation can be obtained using that
  \begin{equation}
    \sum_{\lambda \in \mathbb{G}\mathbb{T}_N} \left( \prod_{j = 1}^r
    (p_{\kappa_j}^{(N, M)})^{\tau_j} \right) \rho_{\beta, y} [\lambda] = f^{-
    1} (\mathcal{D}_{\kappa_r})^{\tau_r} \ldots
    (\mathcal{D}_{\kappa_1})^{\tau_1} f, \label{X}
  \end{equation}
  where $f = \prod_{j = N + 1}^M \prod_{i = 1}^N (1 - \beta_j + \beta_j x_i)$.
  In (\ref{X}), we treat $x_i, \beta_j$ like they are deterministic (the
  operators $\mathcal{D}_k$ act on the $x_i$ variables). At the end, we will
  consider the expectation with respect to their joint distribution.
  
  We showed in Proposition \ref{CLTProp} how the higher order terms of the
  summand $(\mathcal{F}_{\kappa_1})^{\tau_1} \ldots
  (\mathcal{F}_{\kappa_r})^{\tau_r}$ of $f^{- 1}
  (\mathcal{D}_{\kappa_r})^{\tau_r} \ldots (\mathcal{D}_{\kappa_1})^{\tau_1}
  f$ give the limit. Our goal is to show that the remaining terms do not
  contribute asymptotically. In Lemma \ref{Lemapent} we showed that such
  remaining terms (of higher order) for
  \[ \mathbb{E} [(p_{\kappa_1}^{(N, M)})^{s_1} \ldots (p_{\kappa_r}^{(N,
     M)})^{s_r}] (\bar{p}_{\kappa_1}^{(N, M)})^{\tau_1 - s_1} \ldots
     (\bar{p}_{\kappa_r}^{(N, M)})^{\tau_r - s_r} \]
  have the form
  \begin{equation}
    \mathbb{E}_{\tmmathbf{\beta}, \tmmathbf{y}} [(F_{\kappa_1})^{s_1 - s'_1}
    \ldots (F_{\kappa_r})^{s_r - s'_r} \mathcal{I} (\{\alpha_i, \gamma_i \}_{i
    = 1}^{s'_1 + \cdots + s'_r})] \cdot (\mathbb{E}_{\tmmathbf{\beta},
    \tmmathbf{y}} [F_{\kappa_1}])^{\tau_1 - s_1} \ldots
    (\mathbb{E}_{\tmmathbf{\beta}, \tmmathbf{y}} [F_{\kappa_r}])^{\tau_r -
    s_r}, \label{z!}
  \end{equation}
  where
  \[ \sum_{i = 1}^{s'_1 + \cdots + s'_r} (\alpha_i + \gamma_i) \leq \sum_{i =
     1}^r s'_i \kappa_i . \]
  Again, some cancellations must be shown to avoid potentially divergent
  terms. For this, we only have to make assumptions for the tuples $(\eta,
  \theta)$ of (\ref{z!}), that correspond to the terms $F_{\kappa}$. No
  further assumptions for $\mathcal{I} (\{\alpha_i, \gamma_i \}_{i = 1}^{s'_1
  + \cdots + s'_r})$ are needed. To make it clear, assume that for all the
  tuples $(\eta, \theta)$ of (\ref{z!}), we have $\eta + \theta = 1$. The
  maximum power of (\ref{z!}) (in the sense that we explained before) will be
  at most $N^{(\tau_1 - s'_1) (\kappa_1 + 1) + \cdots + (\tau_r - s'_r)
  (\kappa_r + 1)} \cdot N^{s'_1 \kappa_1 + \cdots + s'_r \kappa_r}$. Before we
  consider $N, M \rightarrow \infty$, we have to multiply by $N^{(\tau_1 +
  \cdots + \tau_r) / 2} \cdot N^{- \tau_1 (\kappa_1 + 1) - \cdots - \tau_r
  (\kappa_r + 1)}$. Due to our assumption for the tuples, we see that at the
  end we will multiply a convergent term by
  \[ N^{(\tau_1 - s'_1) (\kappa_1 + 1) + \cdots + (\tau_r - s'_r) (\kappa_r +
     1)} N^{s'_1 \kappa_1 + \cdots + s'_r \kappa_r} N^{(s'_1 + \cdots + s'_2)
     / 2} N^{- \tau_1 (\kappa_1 + 1) - \cdots - \tau_r (\kappa_r + 1)} \]
  \[ = N^{- (s'_1 + \cdots + s'_r) / 2} . \]
  Therefore, such terms do not contribute to the limit.
  
  We have to use a similar inductive argument as before, to show cancellations
  for these higher order terms of (\ref{z!}). The procedure is the same as
  that we followed in Proposition \ref{CLTProp}. For this reason, we do not
  present it in full detail. Due to this inductive argument, we only have to
  assume that some of the tuples $(\eta, \theta)$ of
  $\mathbb{E}_{\tmmathbf{\beta}, \tmmathbf{y}} [(F_{\kappa_1})^{s_1 - s'_1}
  \ldots (F_{\kappa_r})^{s_r - s'_r} \mathcal{I} (\{\alpha_i, \gamma_i \}_{i =
  1}^{s'_1 + \cdots + s'_r})]$ are equal to $(0, 0)$, and all of the tuples of
  $(\mathbb{E}_{\tmmathbf{\beta}, \tmmathbf{y}} [F_{\kappa_1}])^{\tau_1 - s_1}
  \ldots (\mathbb{E}_{\tmmathbf{\beta}, \tmmathbf{y}} [F_{\kappa_r}])^{\tau_r
  - s_r}$ are fixed and not equal to $(0, 0)$. Assume that $\{(\eta_{1, 1},
  \theta_{1, 1}), \ldots, (\eta_{1, n_1}, \theta_{1, n_1})\}, \ldots,
  \{(\eta_{r, 1}, \theta_{r, 1}), \ldots, (\eta_{r, n_r}, \theta_{r, n_r})\}$
  are equal to $(0, 0)$, where $n_i \leq s_i - s'_i$. The terms that we want
  to show that cancel out, also appear in (\ref{MN}), for $s_i' \leq s_i - n_i
  \leq \sigma_i \leq s_i$, and the total number of times that they appear is
  equal to
  \[ \sum_{\sigma_1 = s_1 - n_1}^{s_1} \ldots \sum_{\sigma_r = s_r -
     n_r}^{s_r} \prod_{i = 1}^r (- 1)^{\tau_i - \sigma_i}
     \binom{\tau_i}{\sigma_i} \binom{\sigma_i}{s'_i} \binom{\tau_i -
     \sigma_i}{\tau_i - s_i} \binom{\sigma_i - s_i'}{n_i - (s_i - \sigma_i)} =
     0, \]
  for $n_i \neq 0$.
  
  In what we did so far, we showed cancellations when necessary for terms of
  \[ \mathbb{E} [(p_{\kappa_1}^{(N, M)})^{\sigma_1} \ldots (p_{\kappa_r}^{(N,
     M)})^{\sigma_r}] (\bar{p}_{\kappa_1}^{(N, M)})^{\tau_1 - \sigma_1} \ldots
     (\bar{p}_{\kappa_r}^{(N, M)})^{\tau_r - \sigma_r} \]
  that correspond to the maximum power or $N$. There are also other terms that
  contribute to (\ref{MN}), for $N, M$ finite. These emerge, for example, if
  we do not consider the higher order summands in the differentiation
  procedures (\ref{asx}), (\ref{zxcv}) for the rational functions. In other
  words, if we consider a different summand of $\mathcal{F}_k$ than $F_k$. We
  treat such summands in the same way as $F_k$, i.e. writing $X_{j, N, M} =
  \dot{X}_{j, N, M} + \bar{X}_{j, N, M}$, $Y_{j, N} = \dot{Y}_{j, N} +
  \bar{Y}_{j, N}$ and applying the binomial theorem, we separate the
  deterministic and the stochastic part of zero mean. For the same reasons, we
  have cancellations (if necessary) if some of the analogous tuples that
  correspond to the stochastic part are equal to $(0, 0)$. The only difference
  is that we might need much more tuples to be equal to $(0, 0)$ in order to
  obtain something divergent. Notice that lower order terms can also emerge if
  we do not consider the summand $D_k$ of $\mathcal{D}_k$ (at the beginning
  that we apply the differential operators), but summands $C \cdot D_l$, where
  $l < k$, and $C$ is a constant. Then, one can use the same arguments as
  before. This concludes the proof.
\end{proof}

\subsection{Second limit regime: Gaussian free field}

In this subsection, we study (annealed) global fluctuations around the limit
shape, considering a different limit regime for the random weights. This will
lead to Gaussian fluctuations for $m_N [\overline{\rho_{\tmmathbf{\beta},
\tmmathbf{y}}}]$, but under a different scaling that will give a different
covariance structure. Similarly to subsection \ref{hjklm}, we are based on
$\varepsilon$-CLT-appropriateness for the random weights. The main difference
of the limit regime of the current subsection, is that the leading term (which
gives the covariance) of the asymptotic expansion of $\mathbb{E}
[\dot{p}_{k_1}^{(N, M)} \ldots \dot{p}_{k_{\nu}}^{(N, M)}]$ is not given only
by $F_{k_1} \ldots F_{k_{\nu}}$. The extra term will give rise to the Gaussian
Free Field.

\subsubsection{Computation of the covariance for the second limit regime}

In the current section, we consider $1$-CLT appropriate random weights. To
recall our notation, for the random weights $\{\tmmathbf{x}_i,
\tmmathbf{\beta}_i \}_{i = 1}^M$, we consider $\tmmathbf{y}_i = 1
-\tmmathbf{x}_i$, $\tmmathbf{X}_{j, N, M} = \frac{1}{M - N} \sum_{i = N + 1}^M
\tmmathbf{\beta}_i^j$, and $\tmmathbf{Y}_{j, N} = \frac{1}{N} \sum_{i = 1}^N
\tmmathbf{y}_i^j$. For 1-CLT appropriate random weights we have the following
limiting covariance.

\begin{theorem}
  \label{gmstrs}Assume that $\{\tmmathbf{x}_i \}_{i = 1}^M,
  \{\tmmathbf{\beta}_j \}_{j = 1}^M$ are $1$-CLT appropriate sequences of
  random variables. We also denote $\tmmathbf{y}_i = 1 -\tmmathbf{x}_i$, and
  \[ \widehat{\mathbf{G}_1} (z, w) = \sum_{i, j \geq 1} \hat{\mathfrak{q}}_{i,
     j}^{(\gamma)} z^i w^j \text{, \quad} \widehat{\mathbf{G}_2} (z, w)
     \assign \sum_{i, j \geq 1} \hat{\mathfrak{p}}_{i, j}^{(\gamma)} z^i w^j
     \text{, \quad} \widehat{\mathbf{G}_3} (z, w) \assign \sum_{i, j \geq 1}
     \hat{\mathfrak{s}}_{i, j}^{(\gamma)} z^i w^j, \]
  where
  \[ \hat{\mathfrak{q}}_{i, j}^{(\gamma)} = \lim_{\underset{N / M \rightarrow
     \gamma}{N, M \rightarrow \infty}} (M - N)^2 \mathbb{E}
     [\dot{\tmmathbf{X}}_{i, N, M}   \dot{\tmmathbf{X}}_{j, N, M}] \text{,
     \quad} \hat{\mathfrak{p}}^{(\gamma)}_{i, j} = \lim_{\underset{N / M
     \rightarrow \gamma}{N, M \rightarrow \infty}} N^2 \mathbb{E}
     [\dot{\tmmathbf{Y}}_{i, N}  \dot{\tmmathbf{Y}}_{j, N}], \]
  \[ \hat{\mathfrak{s}}_{i, j}^{(\gamma)} = \lim_{\underset{N / M \rightarrow
     \gamma}{N, M \rightarrow \infty}} N (M - N) \mathbb{E}
     [\dot{\tmmathbf{X}}_{i, N, M}  \dot{\tmmathbf{Y}}_{j, N}] . \]
  Then, we have that
  \[ \hat{\mathfrak{c}}_{k_1, k_2} \assign \lim_{\underset{N / M \rightarrow
     \gamma}{N, M \rightarrow \infty}} \frac{1}{N^{k_1 + k_2}} \mathbb{E}
     [\dot{p}^{(N, M)}_{k_1} \dot{p}^{(N, M)}_{k_2}] = \]
  \[ \frac{- 1}{4 \pi^2} \oint_{|w| = \varepsilon} \oint_{|z| = 2 \varepsilon}
     \frac{1}{zw} \left( \left( \frac{1}{\gamma} - 1 \right) \left( \frac{1 -
     z}{z} \mathbf{F}_2 (z) - \frac{1 - z}{z} \right) + \frac{z - 1}{z}
     \mathbf{F}_1 \left( \frac{1}{z} \right) \right)^{k_1} \]
  \[ \times \left( \left( \frac{1}{\gamma} - 1 \right) \left( \frac{1 - w}{w}
     \mathbf{F}_2 (w) - \frac{1 - w}{w} \right) + \frac{w - 1}{w} \mathbf{F}_1
     \left( \frac{1}{w} \right) \right)^{k_2} \left( \widehat{\mathbf{G}_2}
     \left( \frac{1}{z}, \frac{1}{w} \right) - \widehat{\mathbf{G}_3} \left(
     w, \frac{1}{z} \right) \right. \]
  \begin{equation}
    - \widehat{\mathbf{G}_3} \left( z, \frac{1}{w} \right) + \left.
    \widehat{\mathbf{G}_1} (z, w) + \frac{z \text{} w}{(z - w)^2} \right) d
    \text{} z \text{} d \text{} w, \label{mnbmn}
  \end{equation}
  where the functions $\mathbf{F}_1, \mathbf{F}_2$ are the same as in Theorem
  \ref{LLN!}.
\end{theorem}

\begin{proof}
  The computation of the limit is done using again the contour integral
  expressions that we get when we apply the operators $\mathcal{D}_k$ to $f =
  \prod_{i = 1}^N \prod_{j = N + 1}^M (1 - \beta_j + \beta_j x_i)$. The
  existence in the limit of the term of (\ref{mnbmn}) that does not depend on
  $(z - w)^{- 2}$ is due to the arguments of Theorem \ref{CLT1!}, namely, this
  will emerge by the limit of linear combinations of terms (\ref{bnmj}), for
  $\eta_1 + \theta_1 = \eta_2 + \theta_2 = 1$. The term that corresponds to
  $(z - w)^{- 2}$ comes from the summand $\mathcal{G}_{k_1, k_2}$ of $f^{- 1}
  D_{k_2} D_{k_1} f$. This did not contribute to the limit for $2^{- 1}$-CLT
  appropriate random weights.
  
  Due to Lemma \ref{Lemma3}, the higher order term of $\mathcal{G}_{k_1, k_2}$
  is
  \[ \frac{- N^{k_1 + k_2}}{4 \pi^2} \sum_{m = 0}^{k_1} \binom{k_1}{m} (- 1)^m
     \sum_{\mu = 0}^{k_2 - 1} \binom{k_2}{\mu} \frac{(k_2 - \mu) (- 1)^{\mu -
     1}}{\mu + 1} \left( \frac{1}{\gamma} - 1 \right)^{k_1 + k_2 - m - \mu -
     1} \oint \oint \frac{(1 - z)^{k_1} (1 - w)^{k_2}}{(z - w)^2} \]
  \[ \mathbb{E}_{\tmmathbf{\beta}, \tmmathbf{y}} \left[ \left( \sum_{j =
     1}^{\infty} \tmmathbf{X}_{j, N, M} z^{j - 1} \right)^{k_1 - m} \left(
     \sum_{j = 0}^{\infty} \frac{\tmmathbf{Y}_{j, N}}{z^{j + 1}} \right)^m
     \left( \sum_{j = 1}^{\infty} \tmmathbf{X}_{j, N, M} w^{j - 1}
     \right)^{k_2 - \mu - 1} \left( \sum_{j = 0}^{\infty}
     \frac{\tmmathbf{Y}_{j, N}}{w^{j + 1}} \right)^{\mu + 1} \right] d \text{}
     z \text{} d \text{} w. \]
  Dividing this term by $N^{k_1 + k_2}$ and letting $N, M \rightarrow \infty$,
  we get the desired limit.
  
  As we showed in Proposition \ref{CLTProp} and Theorem \ref{normal1!}, the
  potentially divergent terms of $\mathbb{E} [p_{k_1}^{(N, M)} p_{k_2}^{(N,
  M)}]$ (that emerge when we introduce the tuples $(\eta_1, \theta_1)$,
  $(\eta_2, \theta_2)$ for $F_{k_1}$, $F_{k_2}$) will cancel out with terms of
  $-\mathbb{E} [p_{k_1}^{(N, M)}] \mathbb{E} [p_{k_2}^{(N, M)}]$. This is also
  the case (as $N \rightarrow \infty$) for summands of $N^{- k_1 - k_2}
  \mathbb{E} [p_{k_1}^{(N, M)} p_{k_2}^{(N, M)}]$ that converge to non-zero
  terms, because these will emerge from summands of $\mathcal{F}_k
  \mathcal{F}_l$, which are not of higher order.
\end{proof}

\begin{remark}
  It is well known that the summand of the limiting covariance (\ref{mnbmn}),
  that involves $(z - w)^{- 2}$, gives rise to the Gaussian Free Field
  {\cite{B3}}. The remaining term arises from the fluctuations of the random
  environment.
\end{remark}

\subsubsection{Asymptotic normality for the second limit regime}

Now, we focus on showing asymptotic normality for the collection $\{ N^{- k}
(p_k^{(N, M)} -\mathbb{E} [p_k^{(N, M)}]) \}_{k \in \mathbb{N}}$, as $N, M
\rightarrow \infty$, for the case where the random weights $\tmmathbf{x}_i,
\tmmathbf{\beta}_i$ of the Schur measure are 1-CLT appropriate. We recall that
these are the random variables $p_k^{(N, M)} = \sum_{i = 1}^N  (\lambda_i + N
- i)^k$, where $(\lambda_1 \geq \cdots \geq \lambda_N) \in
\mathbb{G}\mathbb{T}_N$ is distributed according to
$\overline{\rho_{\tmmathbf{\beta}, \tmmathbf{y}}}$.

In the current case, we have a different scaling compared to the case that we
examined in \ref{amrtli}. In both of our limit regimes, the scaling of the
random environment passes on to the model. To show asymptotic normality, we
will compute the joint moments of the limiting distribution.

\begin{theorem}
  \label{normal2!}Assume that $\{\tmmathbf{x}_i \}_{i = 1}^M,
  \{\tmmathbf{\beta}_j \}_{j = 1}^M$ are $1$-CLT appropriate sequences of
  random variables. We also denote $\tmmathbf{y}_i = 1 -\tmmathbf{x}_i$, and
  we assume that the asymptotic conditions of Theorem \ref{gmstrs} are
  satisfied. Then, for any positive integers $k_1, \ldots, k_{\nu}$ we have
  \[ \lim_{\underset{N / M \rightarrow \gamma}{N, M \rightarrow \infty}}
     \frac{1}{{N^{k_1 + \cdots + k_{\nu}}} } \mathbb{E} [\dot{p}_{k_1}^{(N,
     M)} \ldots \dot{p}_{k_{\nu}}^{(N, M)}] = 0, \]
  if $\nu$ is odd, and
  \begin{equation}
    \lim_{\underset{N / M \rightarrow \gamma}{N, M \rightarrow \infty}}
    \frac{1}{{N^{k_1 + \cdots + k_{\nu}}} } \mathbb{E} [\dot{p}_{k_1}^{(N, M)}
    \ldots \dot{p}_{k_{\nu}}^{(N, M)}] = \sum_{\pi \in P_2 (\nu)} \prod_{(i,
    j) \in \pi} \hat{\mathfrak{c}}_{k_i, k_j}, \label{ktxsa}
  \end{equation}
  if $\nu$ is even, where $\hat{\mathfrak{c}}_{k, l}$ are the same as in
  Theorem \ref{gmstrs}.
\end{theorem}

\begin{proof}
  Again, in order to show the claim we rely on (\ref{X}) and the determination
  of $f^{- 1} D_{k_{\nu}} \ldots D_{k_1} f$ that was done in Lemma
  \ref{Lemapent}. We recall that $f = \prod_{i = 1}^N \prod_{j = N + 1}^M (1 -
  \beta_j + \beta_j x_i)$. Compared to Theorem \ref{normal1!}, the main
  difference now is that some of the $\mathcal{G}_{k, l}$ related terms of
  (\ref{adfh}) will contribute to the limit. Our goal is to clarify the terms
  that give the limit on the right-hand side of (\ref{ktxsa}). The rest terms
  either will cancel out or they will vanish as $N, M \rightarrow \infty$, for
  the same reasons as in Theorem \ref{normal1!} and Proposition \ref{CLTProp}.
  
  As we showed in Theorem \ref{gmstrs}, the covariance $\hat{\mathfrak{c}}_{k,
  l}$ consists of two independent terms, where the one is related to the
  Gaussian free field. The summand of $\sum_{\pi \in P_2 (\nu)} \prod_{(i, j)
  \in \pi} \hat{\mathfrak{c}}_{k_i, k_j}$ which does not involve terms related
  to the Gaussian free field, will arise from the limit in the left-hand side
  of (\ref{ktxsa}), due to the term (\ref{C}). Of course, the formula of
  (\ref{C}) and the 1-CLT appropriateness imply that this will converge to
  zero for $\nu$ odd.
  
  Now, we focus on the terms of
  \begin{equation}
    \sum_{\underset{\pi_1 \neq \emptyset}{\pi_1 \sqcup \pi_2 = [\nu]}} \left(
    \prod_{(i, j) \in \pi_1} \mathcal{G}_{k_i, k_j} \right) \left( \prod_{i
    \in \pi_2} \mathcal{F}_{k_i} \right) \label{111}
  \end{equation}
  that contribute to the limit (\ref{ktxsa}). This sum is the same as in
  (\ref{adfh}). We recall that it is responsible (in the case where $\nu = 2$)
  for the summand of $\hat{\mathfrak{c}}_{k, l}$ that is related to the
  Gaussian free field. We are only interested in on its higher order terms.
  Moreover, for the higher order terms of $\mathcal{F}_{k_i}$, $i \in \pi_2$,
  we introduce the tuples $(\eta, \theta)$ in the same way as in Theorem
  \ref{normal1!} and Proposition \ref{CLTProp}. We also consider the case
  where for each tuple we have $\eta + \theta = 1$. Let $\pi_1, \pi_2$ of
  (\ref{111}) be fixed, such that $\pi_1 = \{(i_1, j_1), \ldots, (i_{\rho},
  j_{\rho})\}$, and $i_n < j_n$, for $n = 1, \ldots, \rho$. Then, by a simple
  computation we deduce that the corresponding higher order term for such
  $\pi_1, \pi_2$ is
  \[ \left( \frac{1}{2 \pi i} \right)^{\nu} \prod_{a = 1}^{\rho} N^{k_{i_a} +
     k_{j_a}}  \prod_{b \in [\nu] \backslash \{i_n, j_n \}_{n = 1}^{\rho}}
     N^{k_b + 1}  \oint \ldots \oint \frac{(1 - z_1)^{k_1} \ldots (1 -
     z_{\nu})^{k_{\nu}}}{\prod_{a = 1}^{\rho} (z_{i_a} - z_{j_a})^2} \]
  \[ \times \mathbb{E}_{\tmmathbf{\beta}, \tmmathbf{y}} \left[ \prod_{b \in
     [\nu] \backslash \{i_n, j_n \}_{n = 1}^{\rho}} \left( \sum_{j =
     1}^{\infty} \frac{\dot{\tmmathbf{Y}}_{j, N}}{z_b^{j + 1}} - \left(
     \frac{1}{\gamma} - 1 \right) \sum_{j = 1}^{\infty} \dot{\tmmathbf{X}}_{j,
     N, M} z_b^{j - 1} \right) \right. \text{{\hspace{9em}}} \]
  \[ \left. \times \prod_{a \in \{i_n, j_n \}_{n = 1}^{\rho}} \left( \left(
     \frac{1}{\gamma} - 1 \right) \sum_{j = 1}^{\infty} \tmmathbf{X}_{j, N, M}
     z_a^{j - 1} - \sum_{j = 0}^{\infty} \frac{\tmmathbf{Y}_{j, N}}{z_a^{j +
     1}} \right)^{k_a} \right] \]
  \begin{equation}
    \times \prod_{b \in [\nu] \backslash \{i_n, j_n \}_{n = 1}^{\rho}} \left(
    \left( \frac{1}{\gamma} - 1 \right) \sum_{j = 1}^{\infty}
    \bar{\tmmathbf{X}}_{j, N, M} z_b^{j - 1} - \sum_{j = 0}^{\infty}
    \frac{\bar{\tmmathbf{Y}}_{j, N}}{z_b^{j + 1}} \right)^{k_b} d \text{} z_1
    \ldots d \text{} z_{\nu} . \label{qqq}
  \end{equation}
  Before we consider $N, M \rightarrow \infty$, we have to multiply
  (\ref{qqq}) by $N^{- k_1 - \cdots - k_{\nu}}$. It is clear from the formula
  (\ref{qqq}) that for $\nu$ odd the limit will be equal to zero.
  
  For $\nu = 2 \mu$, as a corollary of the $1$-CLT appropriateness and the
  computation of the covariance that we did in Theorem \ref{gmstrs}, we get
  that the limit as $N, M \rightarrow \infty$, will be equal to
  \begin{equation}
    \prod_{a = 1}^{\rho} \hat{\mathfrak{d}}_{k_{i_a}, k_{j_a}}  \sum_{\{
    (\tilde{i}_1, \tilde{j}_1), \ldots, (\tilde{i}_{\mu - \rho},
    \tilde{j}_{\mu - \rho})\}} \prod_{b = 1}^{\mu - \rho} \left(
    \hat{\mathfrak{e}}_{k_{\tilde{i}_b}, k_{\tilde{j}_b}} -
    {\hat{\mathfrak{d}}_{k_{\tilde{i}_b}, k_{\tilde{j}_b}}}  \right),
    \label{0909}
  \end{equation}
  where the sum is taken with respect to all pairings of $[2 \mu] \backslash
  \{i_n, j_n \}_{n = 1}^{\rho}$, and
  \[ \hat{\mathfrak{d}}_{k, l} = \left( \frac{1}{2 \pi i} \right)^2 \oint
     \oint \frac{1}{(z - w)^2} \left( \left( \frac{1}{\gamma} - 1 \right)
     \left( \frac{1 - z}{z} \mathbf{F}_2 (z) - \frac{1 - z}{z} \right) +
     \frac{z - 1}{z} \mathbf{F}_1 \left( \frac{1}{z} \right) \right)^k \]
  \[ \times \left( \left( \frac{1}{\gamma} - 1 \right) \left( \frac{1 - w}{w}
     \mathbf{F}_2 (w) - \frac{1 - w}{w} \right) + \frac{w - 1}{w} \mathbf{F}_1
     \left( \frac{1}{w} \right) \right)^l d \text{} z \text{} d \text{} w. \]
  Considering the sum of all terms (\ref{0909}) with respect to all pairings
  $\pi_1 = \{(i_1, j_1), \ldots, (i_{\rho}, j_{\rho})\}$ of elements of $[2
  \mu]$ ($\rho$ also varies), we get that the higher order terms of
  (\ref{111}) contribute to the desired limit by
  \begin{equation}
    \sum_{\pi \in P_2 (2 \mu)} \prod_{(i, j) \in \pi} \hat{\mathfrak{c}}_{k_i,
    k_j} - \sum_{\pi \in P_2 (2 \mu)} \prod_{(i, j) \in \pi}
    (\hat{\mathfrak{c}}_{k_i, k_j} - \hat{\mathfrak{d}}_{k_i, k_j}) .
    \label{33345}
  \end{equation}
  Subtracting $\hat{\mathfrak{d}}_{k_i, k_j}$ from $\hat{\mathfrak{c}}_{k_i,
  k_j}$ we get the summand of $\hat{\mathfrak{c}}_{k_i, k_j}$ that is not
  related to the Gaussian free field, and the second sum of (\ref{33345})
  cancels out with the limit of (\ref{C}) (for $\nu = 2 \mu$). Therefore, this
  gives us the desired limit on the right-hand side of (\ref{ktxsa}).
  
  For very similar reasons as in Theorem \ref{normal1!} and Proposition
  \ref{CLTProp}, the remaining terms of $\mathbb{E} [\dot{p}_{k_1}^{(N, M)}
  \ldots \dot{p}_{k_{\nu}}^{(N, M)}]$ do not contribute to the limit. Note
  that for (\ref{111}) we can get some convergent terms (which do not
  necessarily converge to $0$), if we also introduce the analogous tuples
  $(\eta, \theta)$ for some of the terms $\mathcal{G}_{k, l}$, in order to
  absorb powers of $N$. But then, in order to get a term that does not
  converge to zero, we must have that the corresponding tuple $(\eta, \theta)$
  of at least one of the terms $\mathcal{F}_k$ in the product of (\ref{111})
  must be equal to $(0, 0)$. Therefore, such terms of $\mathbb{E}
  [\dot{p}_{k_1}^{(N, M)} \ldots \dot{p}_{k_{\nu}}^{(N, M)}]$ will cancel out
  as well.
\end{proof}

\section{The Annealed Central Limit Theorem for several levels}\label{S5}

In the current section we use the arguments of Section \ref{LLL} in order to
prove the annealed CLT for several levels. We treat both limit regimes of the
previous section. In each case, we start from standard formulas for the Schur
process $\mathbb{P}_{\beta, y}$, in order to extract the moments at several
levels, and we analyze their asymptotics in the same way as in the one-level
case.

We will denote by $f_N$ the function $f_N (x_1, \ldots, x_N) \assign \prod_{j
= N + 1}^M \prod_{i = 1}^N (1 - \beta_j + \beta_j x_i)$, where $N < M$ is a
non-negative integer. Let $N_1 < \cdots < N_r < M$ be such integers. Due to
the standard property (\ref{2+2=4}) for Schur functions and Proposition
\ref{Prop1}, expressions for the joint moments of $\{ \dot{p}_k^{(N_i, M)}
\}_{k \geq 0, i = 1, \ldots, r}$ can be obtained by the equality
\[ \sum_{\theta^{(1)}, \lambda^{(1)}, \ldots, \theta^{(M)}, \lambda^{(M)}}
   \mathbb{P}_{\beta, y} (\{ \theta^{(i)}, \lambda^{(i)} \}_{i = 1}^M)
   \prod_{j = 1}^r \left( \sum_{i = 1}^{N_j} (\lambda_i^{(N_j)} + N_j -
   i)^{k_j} \right) \]
\begin{equation}
  = f_{N_1}^{- 1} \mathcal{D}_{k_1} [f_{N_1} \cdot f_{N_2}^{- 1}
  \mathcal{D}_{k_2} [f_{N_2} \cdot f^{- 1}_{N_3} \mathcal{D}_{k_3} [\ldots
  f_{N_{r - 1}}^{- 1} \mathcal{D}_{k_{r - 1}} [f_{N_{r - 1}} f_{N_r}^{- 1}
  \mathcal{D}_{k_r} [f_{N_r}]] \ldots]]], \label{090909}
\end{equation}
where for every $i = 1, \ldots, r$, the differential operator
$\mathcal{D}_{k_i}$ acts on the variables $x_1, \ldots, x_{N_i}$.

\begin{theorem}
  \label{nzgamskl}Assume that $\{ \tmmathbf{x}_i \}_{i = 1}^M$, $\{
  \tmmathbf{\beta}_j \}_{j = 1}^M$ are $2^{- 1}$-CLT appropriate sequences of
  random variables, and we denote $\tmmathbf{y}_i = 1 -\tmmathbf{x}_i$. Let
  $N_1 = \lfloor \gamma_1 M \rfloor$ and $N_2 = \lfloor \gamma_2 M \rfloor$,
  for $0 < \gamma_1 < \gamma_2 \leq 1$. We also consider the coefficients
  \[ \tilde{\mathfrak{q}}_{i, j}^{(\gamma_1, \gamma_2)} = \lim_{M \rightarrow
     \infty} M\mathbb{E} [\dot{\tmmathbf{X}}_{i, N_1, M}
     \dot{\tmmathbf{X}}_{j, N_2, M}] \text{, \quad} \tilde{\mathfrak{p}}_{i,
     j}^{(\gamma_1, \gamma_2)} = \lim_{M \rightarrow \infty} M\mathbb{E}
     [\dot{\tmmathbf{Y}}_{i, N_1} \dot{\tmmathbf{Y}}_{j, N_2}], \]
  \[ \tilde{\mathfrak{s}}_{i, j}^{(\gamma_1, \gamma_2)} = \lim_{M \rightarrow
     \infty} M\mathbb{E} [\dot{\tmmathbf{X}}_{i, N_1, M}
     \dot{\tmmathbf{Y}}_{j, N_2}], \]
  and the analytic functions
  \[ \mathbf{G}_1^{(\gamma_2, \gamma_1)} (z, w) = \sum_{i, j \geq 1}
     \tilde{\mathfrak{q}}_{i, j}^{(\gamma_1, \gamma_2)} w^i z^j \text{, \quad}
     \mathbf{G}_2^{(\gamma_2, \gamma_1)} (z, w) = \sum_{i, j \geq 1}
     \tilde{\mathfrak{p}}_{i, j}^{(\gamma_1, \gamma_2)} w^i z^j, \]
  \[ \mathbf{G}_3^{(\gamma_2, \gamma_1)} (z, w) = \sum_{i, j \geq 1}
     \tilde{\mathfrak{s}}_{i, j}^{(\gamma_1, \gamma_2)} w^i z^j . \]
  Then, we have the following:
  \begin{enumerate}
    {\item \label{nocpids}The limiting covariance is given by
    \[ \lim_{M \rightarrow \infty} \frac{1}{M^{k + l + 1}} \mathbb{E}
       [\dot{p}_l^{(N_1, M)} \dot{p}_k^{(N_2, M)}] = \frac{- 1}{4 \pi^2}
       \oint_{| w | = \varepsilon'} \oint_{| z | = \varepsilon} \frac{1}{z
       \text{} w} \left( (1 - \gamma_1) \left( \frac{1 - w}{w} \mathbf{F}_2
       (w) - \frac{1 - w}{w} \right) \right. \]
    \[ \left. + \gamma_1 \frac{w - 1}{w} \mathbf{F}_1 \left( \frac{1}{w}
       \right) \right)^l \left( (1 - \gamma_2) \left( \frac{1 - z}{z}
       \mathbf{F}_2 (z) - \frac{1 - z}{z} \right) + \gamma_2 \frac{z - 1}{z}
       \mathbf{F}_1 \left( \frac{1}{z} \right) \right)^k \]
    \[ \times \left( (1 - \gamma_1) (1 - \gamma_2) \mathbf{G}_1^{(\gamma_2,
       \gamma_1)} (z, w) - (1 - \gamma_1) \gamma_2 \mathbf{G}_3^{(\gamma_2,
       \gamma_1)} \left( \frac{1}{z}, w \right) - \gamma_1 (1 - \gamma_2)
       \mathbf{G}_3^{(\gamma_1, \gamma_2)} \left( \frac{1}{w}, z \right)
       \right. \]
    \[ + \left. \gamma_1 \gamma_2 \mathbf{G}_2^{(\gamma_2, \gamma_1)} \left(
       \frac{1}{z}, \frac{1}{w} \right) \right) d \text{} z \text{} d \text{}
       w, \]
    where the functions $\mathbf{F}_1, \mathbf{F}_2$ are the same as in
    Theorem \ref{LLN!}.}
    
    \item \label{mgmsn}The collection of random variables
    \[ \{ M^{- k - 1 / 2} (p_k^{(\lfloor \gamma M \rfloor, M)} -\mathbb{E}
       [p_k^{(\lfloor \gamma M \rfloor, M)}]) \}_{k \geq 1, 0 < \gamma \leq 1}
    \]
    converges as $M \rightarrow \infty$, in the sense of moments to a Gaussian
    vector.
  \end{enumerate}
\end{theorem}

\begin{proof}
  For the proof of \ref{nocpids}, to get an expression for $\mathbb{E}
  [\dot{p}_l^{(N_1, M)} \dot{p}_k^{(N_2, M)}]$ we use (\ref{090909}) for $\nu
  = 2$. We have already shown that $f_{N_1}^{- 1} \mathcal{D}_l [f_{N_1} \cdot
  f_{N_2}^{- 1} \mathcal{D}_k [f_{N_2}]]$ can be written as an explicit linear
  combination of double contour integrals
  \[ \sum_{J_2 \subseteq [N_1, M] : | J_2 | = \alpha_2} \sum_{I_2 \subseteq
     [N_1] : | I_2 | = \gamma_2} \sum_{J_1 \subseteq [N_2, M] \of | J_1 | =
     \alpha_1} \sum_{I_1 \subseteq [N_2] \of | I_1 | = \gamma_1} \oint \oint
     \frac{(1 - w)^l (1 - z)^k}{(z - w)^{\alpha}} \]
  \begin{equation}
    \prod_{n \in J_1} \frac{\beta_n}{1 - \beta_n z} \cdot \prod_{m \in I_1}
    \frac{1}{z - y_m} \cdot \prod_{n \in J_2} \frac{\beta_n}{1 - \beta_n w}
    \cdot \prod_{m \in I_2} \frac{1}{w - y_m} d \text{} z \text{} d \text{} w,
    \label{mcwhel}
  \end{equation}
  where $\alpha_2 + \gamma_2 + \alpha_1 + \gamma_1 \leq k + l + 2$. The term
  of the linear combination that corresponds to the highest power of $M$ is
  $\mathcal{F}_l^{(N_1)} \cdot \mathcal{F}_k^{(N_2)}$, where
  $\mathcal{F}_l^{(N_1)}$, $\mathcal{F}_k^{(N_2)}$ are the same as
  $\mathcal{F}_l$, $\mathcal{F}_k$ (see in (\ref{vfsgmw})), if we replace $N$
  by $N_1$, $N_2$ respectively. For the same reasons as in Theorem
  \ref{CLT1!}, only this double contour integral contributes to the desired
  limit, and gives the formula that we claimed.
  
  Let $0 < \gamma_1 \leq \cdots \leq \gamma_r \leq 1$, and $N_i = \lfloor
  \gamma_i M \rfloor$. In order to prove \ref{mgmsn}, using similar arguments
  as in Lemma \ref{Lemapent}, we see that the expectations $\mathbb{E}
  [\dot{p}_{k_1}^{(N_1, M)} \ldots \dot{p}_{k_r}^{(N_r, M)}]$ are linear
  combinations of $r$-fold contour integrals, analogous to (\ref{mcwhel}). The
  contour integral that corresponds to the highest power of $ M$ is
  $\mathcal{F}_{k_1}^{(N_1)} \cdot \ldots \cdot \mathcal{F}_{k_r}^{(N_r)}$,
  and more specifically its summand
  \[ \frac{1}{(2 \pi i)^r} \oint \ldots \oint \prod_{n = 1}^r (1 - z_n)^{k_{r
     - n + 1}} \left( (\gamma_{r - n + 1} - 1) \sum_{j = 1}^{\infty}
     \bar{\tmmathbf{X}}_{j, N_{r - n + 1}, M} z_n^{j - 1} - \gamma_{r - n + 1}
     \sum_{j = 0}^{\infty} \frac{\bar{\tmmathbf{Y}}_{j, N_{r - n + 1}}}{z_n^{j
     + 1}} \right)^{k_{r - n + 1}} \]
  \[ \times \mathbb{E}_{\tmmathbf{\beta}, \tmmathbf{y}} \left[ \prod_{n =
     1}^{\tmxspace r} \left( \gamma_{r - n + 1} \sum_{j = 1}^{\infty}
     \frac{\dot{\tmmathbf{Y}}_{j, N_{r - n + 1}}}{z_n^{j + 1}} - (\gamma_{r -
     n + 1} - 1) \sum_{j = 1}^{\infty} \dot{\tmmathbf{X}}_{j, N_{r - n + 1},
     M} z_n^{j - 1} \right) \right] d \text{} z_1 \ldots d \text{} z_r . \]
  Our CLT-appropriateness condition implies that the limit of the normalized
  contour integral will converge to the joint moment of the Gaussian vector
  that we claimed. The rest $\nu$-fold contour integrals of the linear
  combination do not contribute to the limit for the same reasons as in
  Theorem \ref{normal1!} and Proposition \ref{CLTProp}.
\end{proof}

For particular examples we compute the covariance of the previous theorem.

\begin{example}
  \label{ex21}For the case where the random vectors $\{ (\tmmathbf{x}_j,
  \tmmathbf{\beta}_j) \}_{j = 1}^M$ of Theorem \ref{nzgamskl} are i.i.d., the
  limiting covariance is
  \[ \lim_{M \rightarrow \infty} \frac{1}{M^{k + l + 1}} \mathbb{E}
     [\dot{p}_l^{(N_1, M)} \dot{p}_k^{(N_2, M)}] = \frac{- 1}{4 \pi^2}
     \oint_{| w | = \varepsilon'} \oint_{| z | = \varepsilon} \frac{1}{z
     \text{} w} \left( (1 - \gamma_2) \mathbb{E} \left[
     \frac{\tmmathbf{\beta}_1 -\tmmathbf{\beta}_1 z}{1 -\tmmathbf{\beta}_1 z}
     \right] + \gamma_2 \mathbb{E} \left[ \frac{z - 1}{z -\tmmathbf{y}_1}
     \right] \right)^k \]
  \[ \times \left( (1 - \gamma_1) \mathbb{E} \left[ \frac{\tmmathbf{\beta}_1
     -\tmmathbf{\beta}_1 w}{1 -\tmmathbf{\beta}_1 w} \right] + \gamma_1
     \mathbb{E} \left[ \frac{w - 1}{w -\tmmathbf{y}_1} \right] \right)^l
     \left( (1 - \gamma_2) \tmop{Cov} \left( \frac{\tmmathbf{\beta}_1 z}{1
     -\tmmathbf{\beta}_1 z}, \frac{\tmmathbf{\beta}_1 w}{1 -\tmmathbf{\beta}_1
     w} \right) \right. \]
  \[ \left. + \gamma_1 \tmop{Cov} \left( \frac{\tmmathbf{y}_1}{z
     -\tmmathbf{y}_1}, \frac{\tmmathbf{y}_1}{w -\tmmathbf{y}_1} \right) +
     (\gamma_1 - \gamma_2) \tmop{Cov} \left( \frac{\tmmathbf{y}_1}{z
     -\tmmathbf{y}_1}, \frac{\tmmathbf{\beta}_1 w}{1 -\tmmathbf{\beta}_1 w}
     \right) \right) d \text{} z \text{} d \text{} w. \]
  The covariance structure arises from a (deterministic pushforward of a)
  forward one-dimensional Brownian motion, that corresponds to the term with
  factor $\gamma_1$, and a backward (time-reversed) one-dimensional Brownian
  motion, that corresponds to the term with factor $1 - \gamma_2$. The term
  with factor $\gamma_1 - \gamma_2 < 0$, emerges from their interaction. As
  the above formula indicates, in the case that the two stochastic processes
  are independent (equivalently, $\tmmathbf{x}_1, \tmmathbf{\beta}_1$ are
  independent), this term vanishes.
\end{example}

\begin{example}
  \label{ex22}Notice that Theorem \ref{nzgamskl} can be applied when the
  random edge weights are given by a nice Markov chain, such that the classical Markov chain CLT holds in the sense of moments (see {\cite{B99}}). If
  $(\tmmathbf{x}_j, \tmmathbf{\beta}_j)$ is a $V$-uniformly ergodic Markov
  chain, such that the distribution of $(\tmmathbf{x}_1, \tmmathbf{\beta}_1)$
  is the stationary distribution, the limiting covariance is
  \[ \lim_{M \rightarrow \infty} \frac{1}{M^{k + l + 1}} \mathbb{E}
     [\dot{p}_l^{(N_1, M)} \dot{p}_k^{(N_2, M)}] = \frac{- 1}{4 \pi^2}
     \oint_{| w | = \varepsilon'} \oint_{| z | = \varepsilon} \frac{1}{z
     \text{} w} \left( (1 - \gamma_2) \mathbb{E} \left[
     \frac{\tmmathbf{\beta}_1 -\tmmathbf{\beta}_1 z}{1 -\tmmathbf{\beta}_1 z}
     \right] + \gamma_2 \mathbb{E} \left[ \frac{z - 1}{z -\tmmathbf{y}_1}
     \right] \right)^k \]
  \[ \times \left( (1 - \gamma_1) \mathbb{E} \left[ \frac{\tmmathbf{\beta}_1
     -\tmmathbf{\beta}_1 w}{1 -\tmmathbf{\beta}_1 w} \right] + \gamma_1
     \mathbb{E} \left[ \frac{w - 1}{w -\tmmathbf{y}_1} \right] \right)^l ((1 -
     \gamma_2) \tmmathbf{\Phi} (z, w) + \gamma_1 \tmmathbf{\Psi} (z, w) +
     (\gamma_1 - \gamma_2) \tmmathbf{\Xi} (z, w)) d \text{} z \text{} d
     \text{} w, \]
  where
  \[ \tmmathbf{\Phi} (z, w) = \sum_{m = 1}^{\infty} \left( \tmop{Cov} \left(
     \frac{\tmmathbf{\beta}_1 z}{1 -\tmmathbf{\beta}_1 z},
     \frac{\tmmathbf{\beta}_{m + 1} w}{1 -\tmmathbf{\beta}_{m + 1} w} \right)
     + \tmop{Cov} \left( \frac{\tmmathbf{\beta}_1 w}{1 -\tmmathbf{\beta}_1 w},
     \frac{\tmmathbf{\beta}_{m + 1} z}{1 -\tmmathbf{\beta}_{m + 1} z} \right)
     \right) \text{{\hspace{7em}}} \]
  \[ \text{{\hspace{27em}}} + \tmop{Cov} \left( \frac{\tmmathbf{\beta}_1 z}{1
     -\tmmathbf{\beta}_1 z}, \frac{\tmmathbf{\beta}_1 w}{1 -\tmmathbf{\beta}_1
     w} \right), \]
  \[ \tmmathbf{\Psi} (z, w) = \sum_{m = 1}^{\infty} \left( \tmop{Cov} \left(
     \frac{\tmmathbf{y}_1}{z -\tmmathbf{y}_1}, \frac{\tmmathbf{y}_{m + 1}}{w
     -\tmmathbf{y}_{m + 1}} \right) + \tmop{Cov} \left(
     \frac{\tmmathbf{y}_1}{w -\tmmathbf{y}_1}, \frac{\tmmathbf{y}_{m + 1}}{z
     -\tmmathbf{y}_{m + 1}} \right) \right) + \tmop{Cov} \left(
     \frac{\tmmathbf{y}_1}{z -\tmmathbf{y}_1}, \frac{\tmmathbf{y}_1}{w
     -\tmmathbf{y}_1} \right), \]
  \[ \tmmathbf{\Xi} (z, w) = \sum_{m = 1}^{\infty} \left( \tmop{Cov} \left(
     \frac{\tmmathbf{y}_1}{z -\tmmathbf{y}_1}, \frac{\tmmathbf{\beta}_{m + 1}
     w}{1 -\tmmathbf{\beta}_{m + 1} w} \right) + \tmop{Cov} \left(
     \frac{\tmmathbf{\beta}_1 w}{1 -\tmmathbf{\beta}_1 w},
     \frac{\tmmathbf{y}_{m + 1}}{z -\tmmathbf{y}_{m + 1}} \right) \right) +
     \tmop{Cov} \left( \frac{\tmmathbf{\beta}_1 w}{1 -\tmmathbf{\beta}_1 w},
     \frac{\tmmathbf{y}_1}{z -\tmmathbf{y}_1} \right) . \]
  Such a covariance structure is interpreted in the same way as in Example
  \ref{ex21}.
\end{example}

In a similar way, we obtain a multilevel version of Theorem \ref{gmstrs}.

\begin{theorem}
  \label{alloena}Assume that $\{ \tmmathbf{x}_i \}_{i = 1}^M$, $\{
  \tmmathbf{\beta}_j \}_{j = 1}^M$ are $1$-CLT appropriate sequences of random
  variables, and we denote $\tmmathbf{y}_i = 1 -\tmmathbf{x}_i$. Let $N_1 =
  \lfloor \gamma_1 M \rfloor$ and $N_2 = \lfloor \gamma_2 M \rfloor$, for $0 <
  \gamma_1 \leq \gamma_2 \leq 1$. We also consider the coefficients
  \[ \breve{\mathfrak{q}}_{i, j}^{(\gamma_1, \gamma_2)} = \lim_{M \rightarrow
     \infty} M^2 \mathbb{E} [\dot{\tmmathbf{X}}_{i, N_1, M}
     \dot{\tmmathbf{X}}_{j, N_2, M}] \text{, \quad} \breve{\mathfrak{p}}_{i,
     j}^{(\gamma_1, \gamma_2)} = \lim_{M \rightarrow \infty} M^2 \mathbb{E}
     [\dot{\tmmathbf{Y}}_{i, N_1} \dot{\tmmathbf{Y}}_{j, N_2}], \]
  \[ \breve{\mathfrak{s}}_{i, j}^{(\gamma_1, \gamma_2)} = \lim_{M \rightarrow
     \infty} M^2 \mathbb{E} [\dot{\tmmathbf{X}}_{i, N_1, M}
     \dot{\tmmathbf{Y}}_{j, N_2}], \]
  and the analytic functions
  \[ \breve{\mathbf{G}}_1^{(\gamma_2, \gamma_1)} (z, w) = \sum_{i, j \geq 1}
     \breve{\mathfrak{q}}_{i, j}^{(\gamma_1, \gamma_2)} z^j w^i \text{, \quad}
     \breve{\mathbf{G}}_2^{(\gamma_2, \gamma_1)} (z, w) = \sum_{i, j \geq 1}
     \breve{\mathfrak{p}}_{i, j}^{(\gamma_1, \gamma_2)} z^j w^i, \]
  \[ \breve{\mathbf{G}}_3^{(\gamma_2, \gamma_1)} (z, w) = \sum_{i, j \geq 1}
     \breve{\mathfrak{s}}_{i, j}^{(\gamma_1, \gamma_2)} z^j w^i . \]
  Then, we have the following:
  \begin{itemize}
    {\item The limiting covariance is given by
    \[ \lim_{M \rightarrow \infty} \frac{1}{M^{k + l}} \mathbb{E}
       [\dot{p}_l^{(N_1, M)} \dot{p}_k^{(N_2, M)}] = \frac{- 1}{4 \pi^2}
       \oint_{| w | = \varepsilon} \oint_{| z | = 2 \varepsilon} \frac{1}{z
       \text{} w} \left( (1 - \gamma_1) \left( \frac{1 - w}{w} \mathbf{F}_2
       (w) - \frac{1 - w}{w} \right) \right. \]
    \[ \left. + \gamma_1 \frac{w - 1}{w} \mathbf{F}_1 \left( \frac{1}{w}
       \right) \right)^l \left( (1 - \gamma_2) \left( \frac{1 - z}{z}
       \mathbf{F}_2 (z) - \frac{1 - z}{z} \right) + \gamma_2 \frac{z - 1}{z}
       \mathbf{F}_1 \left( \frac{1}{z} \right) \right)^k \]
    \[ \times \left( \gamma_1 \gamma_2  \breve{\mathbf{G}}_2^{(\gamma_2,
       \gamma_1)} \left( \frac{1}{z}, \frac{1}{w} \right) - \right. \gamma_2
       (1 - \gamma_1) \breve{\mathbf{G}}_3^{(\gamma_2, \gamma_1)} \left(
       \frac{1}{z}, w \right) - \gamma_1 (1 - \gamma_2)
       \breve{\mathbf{G}}_3^{(\gamma_1, \gamma_2)} \left( \frac{1}{w}, z
       \right) \]
    \[ \left. + (1 - \gamma_1) (1 - \gamma_2) \breve{\mathbf{G}}_1^{(\gamma_2,
       \gamma_1)} (z, w) + \frac{z \text{} w}{(z - w)^2} \right) d \text{} z
       \text{} d \text{} w, \]
    where the functions $\mathbf{F}_1, \mathbf{F}_2$ are the same as in
    Theorem \ref{LLN!}.}
    
    \item The collection of random variables
    \[ \{ M^{- k} (p_k^{(\lfloor \gamma M \rfloor, M)} -\mathbb{E}
       [p_k^{(\lfloor \gamma M \rfloor, M)}]) \}_{k \geq 1, 0 < \gamma \leq 1}
    \]
    converges as $M \rightarrow \infty$, in the sense of moments to a Gaussian
    vector.
  \end{itemize}
\end{theorem}

\begin{proof}
  The proof is the same as in the one-level case. Note that in the multi-level
  setting, the analogous of $\mathcal{G}_{l, k}$ of Lemma \ref{Lemma3} is
  given by
  \[ \mathcal{G}_{l, k}^{(N_1, N_2)} \assign \frac{- 1}{4 \pi^2} \sum_{m =
     0}^k \binom{k}{m} (- 1)^m m! \sum_{\mu = 0}^{l - 1} \binom{l}{\mu} (l -
     \mu) (- 1)^{\mu - 1} \mu ! \sum_{\underset{i_s \neq i_r \text{ for } s
     \neq r}{i_1, \ldots, i_{k - m} \in [N_2 ; M]}} \sum_{\{ l_1, \ldots, l_m
     \} \subseteq [N_2]} \]
  \[ \sum_{\underset{\iota_s \neq \iota_r \text{ for } s \neq r}{\iota_1,
     \ldots, \iota_{l - \mu - 1} \in [N_1 ; M]}} \sum_{\{ \lambda_0, \ldots,
     \lambda_{\mu} \} \subseteq [N_1]} \oint \oint \frac{(1 - z)^k (1 -
     w)^l}{(z - w)^2} \frac{\beta_{i_1} \ldots \beta_{i_{k - m}}}{(1 -
     \beta_{i_1} z) \ldots (1 - \beta_{i_{k - m}} z)} \]
  \begin{equation}
    \times \frac{1}{(z - y_{l_1}) \ldots (z - y_{l_m})} \frac{\beta_{\iota_1}
    \ldots \beta_{\iota_{l - \mu - 1}}}{(1 - \beta_{\iota_1} w) \ldots (1 -
    \beta_{\iota_{l - \mu - 1}} w)} \frac{1}{(w - y_{\lambda_0}) \ldots (w -
    y_{\lambda_{\mu}})} d \text{} z \text{} d \text{} w. \label{lktkos}
  \end{equation}
  This term appears in $f_{N_1}^{- 1} D_l (f_{N_1} \cdot f_{N_2}^{- 1} D_k
  (f_{N_2}))$ (where $D_k$ acts on $x_1, \ldots, x_{N_2}$, and $D_l$ acts on
  $x_1, \ldots, x_{N_1}$) in the same way that $\mathcal{G}_{l, k}$ appears in
  $f_N^{- 1} D_l D_k (f_N)$, and it will contribute to the limiting covariance
  for the same reasons as in Theorem \ref{gmstrs}.
\end{proof}

\begin{example}
  \label{ex24}The above theorem can be applied in the case that the random
  edge weights are described by the eigenvalues of the Gaussian Unitary
  Ensemble. For simplicity, we consider the one-level case. Let
  $\tmmathbf{x}_i = 1$ and $\tmmathbf{\beta}_{M - i + 1} = (\tmmathbf{l}_i^2 +
  1) / (\tmmathbf{l}_i^2 + 2)$, for all $i = 1, \ldots, M$, where
  $\tmmathbf{l}_1 \leq \tmmathbf{l}_2 \leq \cdots \leq \tmmathbf{l}_M$ are the
  ordered eigenvalues of GUE of size $M$. Equivalently, for the edge weights
  of Figure 1, we have $\tmmathbf{w}_{M - i + 1} =\tmmathbf{l}_i^2 + 1$, and
  the rest are equal to 1. In Theorem \ref{gmstrs}, we showed that the
  limiting covariance $\lim_{M \rightarrow \infty} N^{- k_1 - k_2} \mathbb{E}
  [\dot{p}_{k_1}^{(N, M)} \dot{p}_{k_2}^{(N, M)}]$, at level $N = N (M) < M$,
  where $N / M \rightarrow \gamma$ as $M \rightarrow \infty$, will have the
  form
  \[ \frac{- 1}{4 \pi^2} \oint_{| w | = \varepsilon} \oint_{| z | = 2
     \varepsilon} \left( \frac{1}{\gamma} \mathbf{F} (z) + \frac{z - 1}{z}
     \right)^{k_1} \left( \frac{1}{\gamma} \mathbf{F} (w) + \frac{w - 1}{w}
     \right)^{k_2} \left( \mathbf{G} (z, w) + \frac{1}{(z - w)^2} \right) d
     \text{} z \text{} d \text{} w. \]
  The functions $\mathbf{F}, \mathbf{G}$ can be determined by known results
  for the linear eigenvalue statistics of GUE {\cite{B97}}, {\cite{B98}}. Due
  to Theorem \ref{LLN!} we have
  \[ \mathbf{F} (z) = \lim_{M \rightarrow \infty} \frac{1}{M} \mathbb{E}
     \left[ \sum_{i = 1}^{M - N} \frac{\tmmathbf{l}_i^2 + 1}{2 - z - (z - 1)
     \tmmathbf{l}_i^2} \right] = \frac{1}{2 \pi} \int_{-
     2}^{\varepsilon_{\gamma}} \frac{(x^2 + 1) \sqrt{4 - x^2}}{2 - z - (z - 1)
     x^2} d \text{} x \]
  \[ = \frac{3 - 2 z}{2 \pi (z - 1)^2} \left( \arcsin \left(
     \frac{\varepsilon_{\gamma}}{2} \right) + \frac{\pi}{2} \right) -
     \frac{\varepsilon_{\gamma} \sqrt{4 - \varepsilon_{\gamma}^2}}{4 \pi (z -
     1)} \text{{\hspace{20em}}} \]
  \[ + \frac{5 z - 6}{2 \pi (z - 1)^2 \sqrt{(2 - z) (6 - 5 z)}} \left( \arctan
     \left( \sqrt{\frac{6 - 5 z}{2 - z}} \frac{\varepsilon_{\gamma}}{\sqrt{4 -
     \varepsilon_{\gamma}^2}} \right) + \frac{\pi}{2} \right), \]
  where $\varepsilon_{\gamma}$ is uniquely determined by
  \[ \frac{1}{2 \pi} \int_{- 2}^{\varepsilon_{\gamma}} \sqrt{4 - x^2} d
     \text{} x = 1 - \gamma . \]
  We have shown that $\mathbf{G} (z, w)$ is given by the limit of
  \[ \tmop{Cov} \left( \sum_{i = N + 1}^M \frac{\tmmathbf{\beta}_i}{1
     -\tmmathbf{\beta}_i z}, \sum_{i = N + 1}^M \frac{\tmmathbf{\beta}_i}{1
     -\tmmathbf{\beta}_i w} \right) \text{{\hspace{16em}}} \]
  \begin{equation}
    \text{{\hspace{5em}}=} \tmop{Cov} \left( \sum_{i = 1}^{M - N}
    \frac{\tmmathbf{l}_i^2 + 1}{2 - z - (z - 1) \tmmathbf{l}_i^2}, \sum_{i =
    1}^{M - N} \frac{\tmmathbf{l}_i^2 + 1}{2 - w - (w - 1) \tmmathbf{l}_i^2}
    \right) \label{WSX}
  \end{equation}
  It emerges from {\cite{B97}}, that this limit can be written as
  \[ \frac{1}{4 \pi^2} \int_{- 2}^{\varepsilon_{\gamma}} \int_{-
     2}^{\varepsilon_{\gamma}} \frac{(\varphi (z, x) - \varphi (z, y))
     (\varphi (w, x) - \varphi (w, y))}{(x - y)^2 \sqrt{4 - x^2} \sqrt{4 -
     y^2}} \left( 4 - x \text{} y \right) d \text{} x \text{} d \text{} y
     \text{{\hspace{12em}}} \]
  \begin{equation}
    \text{{\hspace{10em}}} - \frac{\sqrt{4 - \varepsilon_{\gamma}^2}}{2 \pi^2}
    \int_{- 2}^{\varepsilon_{\gamma}} \frac{(\varphi (z, x) - \varphi (z,
    \varepsilon_{\gamma})) (\varphi (w, x) - \varphi (w,
    \varepsilon_{\gamma}))}{(x - \varepsilon_{\gamma}) \sqrt{4 - x^2}} d
    \text{} x, \label{CVCVCB}
  \end{equation}
  where $\varphi (z, x) = (x^2 + 1) / (2 - z - (z - 1) x^2)$.
  
  The first double integral can be written as $\frac{1}{(z - w)^2}
  \mathbf{G}_1 (z, w)$ where $\mathbf{G}_1 (z, w)$ is a finite sum of products
  $\mathbf{F}_1 (z) \mathbf{G}_1 (w)$, and $\mathbf{F}_1, \mathbf{G}_1$ are
  analytic functions. The second integral can be written as $\frac{1}{z - w}
  \mathbf{G}_2 (z, w)$, where $\mathbf{G}_2 (z, w)$ is like $\mathbf{G}_1 (z,
  w)$. Although the computation of these integrals is elementary, the formulas
  of $\mathbf{G}_1 (z, w)$, $\mathbf{G}_2 (z, w)$ are too long and they are
  omitted. Since $\varepsilon_{\gamma} \neq 2$, the covariance consists of a
  new random object. Its \ singularity is weaker than that of the Gaussian
  Free Field.
  
  Notice that if the edge weights till level $N$ are encoded by all the
  eigenvalues of GUE, the limit of (\ref{WSX}) can be simplified since this
  corresponds to the case where $\varepsilon_{\gamma} = 2$, and the second
  integral in (\ref{CVCVCB}) will not contribute. For example, this would be
  the case if $\tmmathbf{w}_{M - i + 1} = \tilde{\tmmathbf{l}}_i^2 + 1$, for
  all $i = 1, \ldots, M - N$, where $\tilde{\tmmathbf{l}}_1 \leq \cdots \leq
  \tilde{\tmmathbf{l}}_{M - N}$ are eigenvalues of GUE of size $M - N$, and
  $\tmmathbf{w}_i = 1$ for $i \leq N$. Then, as $M \rightarrow \infty$, the
  covariance at level $N$ will be equal to
  \[ \lim_{M \rightarrow \infty} \frac{1}{N^{k_1 + k_2}} \mathbb{E}
     [\dot{p}_{k_1}^{(N, M)} \dot{p}_{k_2}^{(N, M)}] \text{{\hspace{20em}}} \]
  \[ = \frac{- 1}{4 \pi^2} \oint_{| w | = \varepsilon} \int_{| z | = 2
     \varepsilon} \left( \left( \frac{1}{\gamma} - 1 \right) \frac{5 z - 6 +
     (3 - 2 z) \sqrt{(2 - z) (6 - 5 z)}}{2 (z - 1)^2 \sqrt{(2 - z) (6 - 5 z)}}
     + \frac{z - 1}{z} \right)^{k_1} \]
  \[ \times \left( \left( \frac{1}{\gamma} - 1 \right) \frac{5 w - 6 + (3 - 2
     w) \sqrt{(2 - w) (6 - 5 w)}}{2 (w - 1)^2 \sqrt{(2 - w) (6 - 5 w)}} +
     \frac{w - 1}{w} \right)^{k_2} \]
  \[ \times \frac{1}{(z - w)^2} \left( \frac{24 - 16 z - 16 w + 10 z \text{}
     w}{(2 - z) (2 - w)} \sqrt{\frac{(2 - z) (2 - w)}{(6 - 5 z) (6 - 5 w)}} -
     1 \right) d \text{} z \text{} d \text{} w. \]
\end{example}
\
\

\subsection*{Gaussian Free Field and Poisson distribution}

In Theorems \ref{normal1!} and \ref{normal2!}, we established Gaussian fluctuations for $m_N
[\overline{\rho_{(\tmmathbf{\beta}, \tmmathbf{y})}}]$. These arise from the
fact that the random weights satisfy a central limit theorem, so that the
fluctuations of the random empirical measures associated to $\{ \tmmathbf{}
\tmmathbf{y}_i \}_{i = 1}^M$ and $\{ \tmmathbf{\beta}_i \}_{i = 1}^M$ are
asymptotically Gaussian. Now we consider the simplest case where the global
fluctuations of $\{ \tmmathbf{y}_i \}_{i = 1}^M$, $\{ \tmmathbf{\beta}  \}_{i
	= 1}^M$ are not Gaussian, and we can still characterize the fluctuations of
$m_N [\overline{\rho_{(\tmmathbf{\beta}, \tmmathbf{y})}}]$.

For simplicity we consider $\tmmathbf{y}_i = 0$, for every $i = 1, \ldots, M$. The random variables $\{ \tmmathbf{\beta}_i \}_{i = 1}^M$ will be
dependent and for each $\tmmathbf{\beta}_i$, its probability distribution is
supported on $\{ a, b \}$, where $a, b \in (0, 1)$. We also consider a random
variable $\Xi_M$ that takes values on $\{ 0, \ldots, M \}$. Then, we choose the
random edge weights $\tmmathbf{w}_i =\tmmathbf{\beta}_i / (1
-\tmmathbf{\beta}_i)$ of the one-periodic Aztec diamond of Figure 1 in the
following way. Starting from $\tmmathbf{w}_M$, the edge weights $\tmmathbf{w}_i$ are equal to $a
/ (1 - a)$ and at some random level $\Xi_M$ of the Aztec diamond we update the edge
weights and they become equal to $b / (1 - b)$ for the remaining levels. This means
that
\begin{equation}
	(\tmmathbf{w}_M, \ldots, \tmmathbf{w}_1) = \left( \left( \frac{a}{1 - a}
	\right)^{\Xi_M}, \left( \frac{b}{1 - b} \right)^{M - \Xi_M} \right)
	\text{\quad or, equivalently } (\tmmathbf{\beta}_M, \ldots,
	\tmmathbf{\beta}_1) = (a^{\Xi_M}, b^{M - \Xi_M}) . \label{Poisson}
\end{equation}
For the remaining of this section, the random variables $\{ \tmmathbf{\beta}_i
\}_{i = 1}^M$ will be given by (\ref{Poisson}).
%We also assume that $\Xi_M$
%has a binomial distribution with parameters $M, p_M$, such that $\lim_{M
%	\rightarrow \infty} M p_M =l$.
	\begin{theorem}
		\label{PoissonCLT}For the random variables $\{ \tmmathbf{x}_i \}_{i = 1}^M$,
		$\{ \tmmathbf{\beta}_i \}_{i = 1}^M$ assume that $\tmmathbf{x}_i \equiv 1$
		for every $i = 1, \ldots, M$, $\tmmathbf{\beta}_i$ are given by
		(\ref{Poisson}), and the random variable $\Xi_M$ has a binomial distribution
		with parameters $M, p_M$, such that $\lim_{M \rightarrow \infty} M p_M =
		l$. Let $N_i = \lfloor \gamma_i M \rfloor$, for $0 < \gamma_1 \leq
		\cdots \leq \gamma_{\nu} < 1$. Then for any positive integers $k_1, \ldots,
		k_{\nu}$ we have
		\begin{equation}
			\lim_{M \rightarrow \infty} \frac{1}{M^{k_1 + \cdots + k_{\nu}}}
			\mathbb{E} [\dot{p}_{k_1}^{(N_1, M)} \ldots \dot{p}_{k_{\nu}}^{(N_{\nu},
				M)}] =\mathbb{E} [(\tmmathbf{G}_1 + c_1 (a, b) (\tmmathbf{P}-
			l)) \ldots (\tmmathbf{G}_{\nu} + c_{\nu} (a, b) (\tmmathbf{P}-
			l))], \label{limitPoisson}
		\end{equation}
		where $\tmmathbf{P}$ is a Poisson random variable with parameter
		$l$, and $(\tmmathbf{G}_1, \ldots, \tmmathbf{G}_{\nu})$ is a
		Gaussian random vector with covariance
		\begin{equation}
			\tmop{Cov} (\tmmathbf{G}_i, \tmmathbf{G}_j) = \frac{- 1}{4 \pi^2} \oint
			\oint \frac{d \text{} z \text{} d \text{} w}{(z - w)^2} \left( (1 -
			\gamma_j) \frac{b - b z}{1 - b z} + \gamma_j \frac{z - 1}{z} \right)^{k_j}
			\left( (1 - \gamma_i) \frac{b - b w}{1 - b w} + \gamma_{i } \frac{w -
				1}{w} \right)^{k_{i }}, \label{covarPOisson}
		\end{equation}
		for $i < j$. Moreover, $\tmmathbf{P}$, $(\tmmathbf{G}_1, \ldots,
		\tmmathbf{G}_{\nu})$ are independent, and $c_j (a, b)$ are constants given
		by
		\[ c_j (a, b) = \frac{1}{2 \pi i} \oint \left( (1 - \gamma_{j }) \frac{b - b
			z}{1 - b z} + \gamma_{j } \frac{z - 1}{z} \right)^{k_j} \left( \frac{b}{1
			- b z} - \frac{a}{1 - a z} \right) d \text{} z, \]
		where the contour is a positively oriented circle with radius smaller than
		$1$.
	\end{theorem}
	
	\begin{proof}
		Similarly to CLT appropriate weights, to prove the claim we rely on
		asymptotic properties for the linear statistics of $\tmmathbf{\beta}_1,
		\ldots, \tmmathbf{\beta}_M$. By the definition (\ref{Poisson}), for $N = N
		(M) < M$, with $N \rightarrow \infty$, as $M \rightarrow \infty$, we have
		that
		\[ \tmmathbf{\beta}_{N + 1}^k + \cdots +\tmmathbf{\beta}_M^k = (M - N) a^k
		\tmmathbf{1}_{\{ \Xi_M \geq M - N \}} +\tmmathbf{1}_{\{ \Xi_M < M - N \}}
		(\Xi_M a^k + (M - N - \Xi_M) b^k) . \]
		Using the well-known formula
		\[ \mathbb{E} [\Xi_M^k] = \sum_{n = 0}^k S (k, n) M (M - 1) \ldots (M - n +
		1) p_M^n \]
		for the moments of the binomial distribution, we immediately get that
		\[ \lim_{M \rightarrow \infty} \frac{1}{M} \mathbb{E} [\Xi_M^k] = 0
		\text{\quad and\quad} \lim_{M \rightarrow \infty} (M - N)^k \mathbb{P}
		(\Xi_M \geq M - N) = 0. \]
		Therefore, we deduce that
		\begin{equation}
			\lim_{M \rightarrow \infty} \mathbb{E} \left[ \prod_{i = 1}^{\nu}
			\frac{\tmmathbf{\beta}_{N_i + 1}^{k_i} + \cdots
				+\tmmathbf{\beta}_M^{k_i}}{M - N_i} \right] = \prod_{i = 1}^{\nu} b^{k_i},
			\label{Poisson1}
		\end{equation}
		and
		\begin{equation}
			\lim_{M \rightarrow \infty} \mathbb{E} \left[ \prod_{i = 1}^{\nu}
			(\tmmathbf{\beta}_{N_i + 1}^{k_i} + \cdots +\tmmathbf{\beta}_M^{k_i}
			-\mathbb{E} [\tmmathbf{\beta}_{N_i + 1}^{k_i} + \cdots
			+\tmmathbf{\beta}_M^{k_i}]) \right] = \prod_{i = 1}^{\nu} (a^{k_i} -
			b^{k_i}) \mathbb{E} [(\tmmathbf{P}- l)^{\nu}], \label{Poisson2}
		\end{equation}
		where $\tmmathbf{P}$ is a Poisson random variable with parameter $l$.
		
		For the same reasons as in Theorem \ref{normal2!} and Theorem \ref{alloena}, the limit
		(\ref{limitPoisson}) will emerge by the higher order terms of
		\begin{equation}
			\sum_{\pi_1 \sqcup \pi_2 = [\nu]} \left( \prod_{(i, j) \in \pi_1}
			\mathcal{G}_{k_i, k_j}^{(N_i, N_j)} \right) \left( \prod_{i \in \pi_2}
			\mathcal{F}_{k_i}^{(N_i)} \right), \label{sumPoisson}
		\end{equation}
		where the sum is over all sets $\pi_1$ of pairs $(i, j)$, with $i < j$, of
		elements of $[\nu]$ and all subsets $\pi_2$ of $[\nu]$, such that $\pi_1,
		\pi_2$ are disjoint and their union is equal to $[\nu]$.
		
		In the case where $\pi_1 = \emptyset$, we have shown that the sum
		(\ref{sumPoisson}) gives a higher order term
		\[ \frac{M^{k_1 + \cdots + k_{\nu} + \nu}}{(2 \pi i)^{\nu}} \oint \ldots
		\oint \prod_{i = 1}^{\nu} (1 - z_i)^{k_i} \left( (1 - \gamma_i) \sum_{j =
			1}^{\infty} \bar{\tmmathbf{X}}_{j, N_i, M} z_i^{j - 1} -
		\frac{\gamma_i}{z_i} \right)^{k_i} \]
		\[ \times \mathbb{E}_{\tmmathbf{\beta}} \left[ \prod_{i = 1}^{\nu} \left(
		(\gamma_i - 1) \sum_{j = 1}^{\infty} \dot{\tmmathbf{X}}_{j, N_i, M}
		z_i^{j - 1} \right) \right] d \text{} z_1 \ldots d \text{} z_{\nu}, \]
		where we recall that $\tmmathbf{X}_{j, N, M} = \frac{1}{M - N} p_j
		(\tmmathbf{\beta}_{N + 1}, \ldots, \tmmathbf{\beta}_M)$. Due to
		(\ref{Poisson1}), (\ref{Poisson2}), this term will contribute to the limit
		(\ref{limitPoisson}) by
		\[ \frac{\mathbb{E} [(\tmmathbf{P}- l)^{\nu}]}{(2 \pi i)^{\nu}} 
		\prod_{i = 1}^{\nu} \oint \left( (1 - \gamma_i) \frac{b - b \text{} z}{1
			- b \text{} z} + \gamma_i \frac{z - 1}{z} \right)^{k_i} \left( \frac{b}{1
			- b \text{} z} - \frac{a}{1 - a \text{} z} \right) d \text{} z, \]
		where the contour encircles $0$, but it does not encircle $b^{- 1}, a^{-
			1}$.
		
		Now we consider the terms of (\ref{sumPoisson}) for $\pi_1 \neq \emptyset$.
		Let $\pi_1 = \{ (i_1, j_1), \ldots, (i_{\rho}, j_{\rho}) \}$. As we showed
		in the proof of Theorem \ref{normal2!}, the higher order term of the summand of
		(\ref{sumPoisson}) that corresponds to such $\pi_1$ is
		\[ \left( \frac{1}{2 \pi i} \right)^{\nu} \left( \prod_{n = 1}^{\rho}
		M^{k_{i_{n}} + k_{j_{n}}} \right)  \left( \prod_{\alpha_2
			\in [\nu] \backslash \{ i_n, j_n \}_{n = 1}^{\rho}} M^{k_{\alpha_2} + 1 }
		\right) \oint \ldots \oint \frac{(1 - z_1)^{k_1} \ldots (1 -
			z_{\nu})^{k_{\nu}}}{\prod_{n = 1}^{\rho} (z_{i_n} - z_{j_n})^2} \]
		\[ \times \mathbb{E}_{\tmmathbf{\beta}} \left[ \prod_{\alpha_2 \in [\nu]
			\backslash \{ i_n, j_n \}_{n = 1}^{\rho}} \left( (\gamma_{\alpha_2} - 1)
		\sum_{j = 1}^{\infty} \dot{\tmmathbf{X}}_{j, N_{\alpha_2}, M}
		z_{\alpha_2}^{j - 1} \right) \right. \text{{\hspace{6em}}} \]
		\[ \text{{\hspace{6em}}} \left. \times \prod_{\alpha_1 \in \{ i_n, j_n \}_{n
				= 1}^{\rho}} \left( (1 - \gamma_{\alpha_1}) \sum_{j = 1}^{\infty}
		\tmmathbf{X}_{j, N_{\alpha_1}, M} z_{\alpha_1}^{j - 1} -
		\frac{\gamma_{\alpha_1}}{z_{\alpha_1}} \right)^{k_{\alpha_1}} \right] \]
		\[ \times \prod_{\alpha_2 \in [\nu] \backslash \{ i_n, j_n \}_{n =
				1}^{\rho}} \left( (1 - \gamma_{\alpha_2}) \sum_{j = 1}^{\infty}
		\bar{\tmmathbf{X}}_{j, N_{\alpha_2}, M} z_{\alpha_2}^{j - 1} -
		\frac{\gamma_{\alpha_2}}{z_{\alpha_2}} \right)^{k_{\alpha_2}} d \text{}
		z_1 \ldots d \text{} z_{\nu} . \]
		Due to
		(\ref{Poisson1}), (\ref{Poisson2}), such terms will contribute to the limit (\ref{limitPoisson}) by

		\[ \frac{\mathbb{E} [(\tmmathbf{P}- l)^{\nu - 2 \rho}]}{(2 \pi
			i)^{\nu - 2 \rho}} \prod_{\alpha  \in [\nu] \backslash \{ i_n, j_n \}_{n
				= 1}^{\rho}} \oint \left( (1 - \gamma_{\alpha }) \frac{b - b z}{1 - b z}
		+ \gamma_{\alpha } \frac{z - 1}{z} \right)^{k_{\alpha}} \left( \frac{b}{1
			- b z} - \frac{a}{1 - a z} \right) d \text{} z \]
		\begin{equation}
			\times \left( \frac{1}{2 \pi i} \right)^{2 \rho} \prod_{n = 1}^{\rho}
			\oint \oint \frac{1}{(z - w)^2} \left( (1 - \gamma_{j_n}) \frac{b - b z}{1
				- b z} + \gamma_{j_n} \frac{z - 1}{z} \right)^{k_{j_n}} \left( (1 -
			\gamma_{i_n}) \frac{b - b w}{1 - b w} + \gamma_{i_n} \frac{w - 1}{w}
			\right)^{k_{i_n}} d \text{} z \text{} d \text{} w.
			\label{Poissoncontribution}
		\end{equation}
		Considering the sums of all the limits (\ref{Poissoncontribution}) that
		arise from (\ref{limitPoisson}), we deduce that the claim holds.
	\end{proof}
	
	\begin{remark}
		\label{PoissonRemark}There are two random objects that describe the global
		fluctuations
		in Theorem \ref{PoissonCLT}. It is well know that the Gaussian term
		corresponds to the Gaussian Free Field. The covariance of the form
		(\ref{covarPOisson}) was shown in \cite{B3} for the global fluctuations of the
		one-periodic Aztec diamond with deterministic weights $w_i = b / (1 - b)$.
		The additional Poisson term is new and it arises from outliers that appear
		in the limit shape of the random partition that corresponds to a one-dimensional slice of the Aztec diamond.
		
		To make it precise, consider the edge weights (\ref{Poisson}), and assume that $\Xi_M$ is
		deterministic and does not depend on $M$. This is a finite perturbation of
		the one-periodic measure with all edge weights $w_i$ being equal to $b / (1 - b)$,
		in the sense that now only finitely many of the deterministic edge weights
		$w_i$ are not equal to $b / (1 - b)$, but they are equal to $a / (1 - a)$. It
		was shown in \cite{BufetovZografos} (see Example 7) that this perturbation gives rise to
		outliers to the limit shape. This was done by studying its $1 / N$
		correction (see Theorem 14 of \cite{BufetovZografos}). In the framework of
		Theorem \ref{PoissonCLT}, considering $\Xi_M$ random gives rise to finitely
		many outliers, but their number is random. Therefore, the extra random term
		in Theorem \ref{PoissonCLT} describes the fluctuations of the random number
		of outliers, and it is Poisson distributed since we consider $\Xi_M$ to be binomial
		distributed, and its parameter $p_M$ decays like $1 / M$.
		
		The constant $c_j (a, b)$ is in general not zero, for $a \neq b$, and it can
		be described in terms of the $1 / N$ correction of the perturbed
		one-periodic measure mentioned above. More precisely, writing
		\[ \frac{1}{2 \pi i} \oint \left( (1 - \gamma_{}) \frac{b + b z}{1 + b z} +
		\gamma_{} \frac{z + 1}{z} \right)^{k } \frac{a}{1 + a z} d \text{} z \]
		\[ = \sum_{m = 0}^{k - 1} \binom{k}{m + 1} \frac{\gamma^k}{m!}  \left.
		\frac{d^m}{d x^m} \left( x^k \left( \left( \frac{1}{\gamma} - 1
		\right) \frac{b}{1 - b + b x} \right)^{k - m - 1} \frac{a}{1 - a + a x}
		\right) \right|_{x = 1}, \]
		by Theorem 14 and Example 7 of \cite{BufetovZografos} one can easily derive an expression for
		$c_j (a, b)$ as the $j$-th moment of a signed measure. This expression also
		determines the outliers.
	\end{remark}
	
	\begin{remark}
		As we have shown, the difference between the $2^{- 1}$-CLT appropriate
		random weights of Theorem \ref{normal1!} and the $1$-CLT appropriate random weights of
		Theorem \ref{normal2!} is that the former eliminate the Gaussian Free Field fluctuations.
		Let $\Xi_M$ be the binomial distribution with parameters $M$, $p_M$ as
		before. Following the framework of Theorem \ref{PoissonCLT}, and the
		discussion about outliers in Remark \ref{PoissonRemark}, considering the
		one-periodic Aztec diamond of size $M^2$ with random edge weights
		\[ (\tmmathbf{w}_{M^2}, \ldots, \tmmathbf{w}_1) = \left( \left( \frac{a}{1 -
			a} \right)^{M \Xi_M}, \left( \frac{b}{1 - b} \right)^{M^2 - M \Xi_M}
		\right), \]
		we get that the random number of outliers to the limit shape is of order
		$M$. Therefore, in this case the fluctuations of $m_N
		[\overline{\rho_{(\tmmathbf{\beta}, \tmmathbf{y})}}]$ (where $N<M^2$ and $N/M^2$ converges as $M \rightarrow \infty$) occur on scale
		$\sqrt{M^2}$ and they are described only by the random Poisson term of Theorem
		\ref{PoissonCLT}.
	\end{remark}

\section{The Quenched Central Limit Theorem for several levels}\label{S6}

In this section, we focus on the quenched CLT for the one-periodic Aztec
diamond with random edge weights. In this case we first fix the (random)
environment, instead of averaging over its randomness.

Let $k_1, \ldots, k_r$ be non-negative integers, $0 < \gamma_1 \leq \cdots
\leq \gamma_r \leq 1$, and $N_i = \lfloor \gamma_i M \rfloor$. For the Schur
process with random parameters $\mathbb{P}_{\tmmathbf{\beta}, \tmmathbf{y}}$,
we will analyze the asymptotic behavior of the quenched moments.
\begin{equation}
  X_{\gamma_1, \ldots, \gamma_r}^{k_1, \ldots, k_r} (\tmmathbf{\beta},
  \tmmathbf{y}) \assign \frac{1}{M^{k_1 + \cdots + k_r}}
  \mathbb{E}_{\tmmathbf{\lambda}} \left[ \prod_{i = 1}^r (p_{k_i}^{(N_i, M)}
  -\mathbb{E}_{\tmmathbf{\lambda}} [p_{k_i}^{(N_i, M)}]) \right] .
  \label{qmnm}
\end{equation}
We recall that $p_{k_i}^{(N_i, M)} = \sum_{j = 1}^{N_i} (\lambda_j^{(N_i)} +
N_i - j)^{k_i}$, and by $\mathbb{E}_{\tmmathbf{\lambda}}$ we denote the
expectation with respect to $\mathbb{P}_{\tmmathbf{\beta}, \tmmathbf{y}}$ for
a given realization of disorder.

\begin{theorem}
  \label{NMNm}Assume that $\{ \tmmathbf{x}_i \}_{i = 1}^M$, $\{
  \tmmathbf{\beta}_j \}_{j = 1}^M$ are $\varepsilon$-CLT appropriate sequences
  of random variables, where $\varepsilon = 1 / 2, 1$, and we denote
  $\tmmathbf{y}_i = 1 -\tmmathbf{x}_i$. Then, we have the following:
  \begin{itemize}
    {\item Almost surely with respect to the random environment $\{
    \tmmathbf{x}_i, \tmmathbf{\beta}_j \}$, the collection of random variables
    \[ \{ M^{- k} (p_k^{(\lfloor \gamma M \rfloor, M)}
       -\mathbb{E}_{\tmmathbf{\lambda}} [p_k^{(\lfloor \gamma M \rfloor, M)}])
       \}_{k \geq 1, 0 < \gamma \leq 1} \]
    converges in (finite-dimensional) distribution, with respect to
    $\mathbb{P}_{\tmmathbf{\beta}, \tmmathbf{y}}$, to a Gaussian random
    vector.}{\item Let $N_1 = \lfloor \gamma_1 M \rfloor$, $N_2 = \lfloor
    \gamma_2 M \rfloor$, for $0 < \gamma_1 \leq \gamma_2 \leq 1$. We have that
    almost surely,
    \[ \lim_{M \rightarrow \infty} \frac{1}{M^{k + l}}
       \mathbb{E}_{\tmmathbf{\lambda}} [(p_l^{(N_1, M)}
       -\mathbb{E}_{\tmmathbf{\lambda}} [p_l^{(N_1 .M)}]) (p_k^{(N_2, M)}
       -\mathbb{E}_{\tmmathbf{\lambda}} [p_k^{(N_2, M)}])] \]
    \[ = \frac{- 1}{4 \pi^2} \oint_{| w | = \varepsilon'} \oint_{| z | = 2
       \varepsilon'} \frac{1}{(z - w)^2} \left( (1 - \gamma_2) \left( \frac{1
       - z}{z} \mathbf{F}_2 (z) - \frac{1 - z}{z} \right) + \gamma_2 \frac{z -
       1}{z} \mathbf{F}_1 \left( \frac{1}{z} \right) \right)^k \]
    \begin{equation}
      \text{{\hspace{4em}}} \left( (1 - \gamma_1) \left( \frac{1 - w}{w}
      \mathbf{F}_2 (w) - \frac{1 - w}{w} \right) + \gamma_1 \frac{w - 1}{w}
      \mathbf{F}_1 \left( \frac{1}{w} \right) \right)^l d \text{} z \text{} d
      \text{} w, \label{poslft}
    \end{equation}
    where the functions $\mathbf{F}_1$, $\mathbf{F}_2$ are the same as in
    Theorem \ref{LLN!}.}
  \end{itemize}
\end{theorem}

The proof of the above theorem is based on the following proposition.

\begin{proposition}
  \label{pfgmsklr}Let $\{ \tmmathbf{x}_i \}_{i = 1}^M$, $\{ \tmmathbf{\beta}_j
  \}_{j = 1}^M$ be the same as in Theorem \ref{NMNm}, $k_1, \ldots, k_r$ be
  arbitrary non-negative integers, and $0 < \gamma_1 \leq \cdots \leq \gamma_r
  \leq 1$. Then, we have the following:
  \begin{itemize}
    {\item For $\varepsilon = 1 / 2$, the limit
    \[ \lim_{M \rightarrow \infty} M^2 \mathbb{E}_{\tmmathbf{\beta},
       \tmmathbf{y}} (X_{\gamma_1, \ldots, \gamma_r}^{k_1, \ldots, k_r}
       (\tmmathbf{\beta}, \tmmathbf{y}) -\mathbb{E}_{\tmmathbf{\beta},
       \tmmathbf{y}} [X_{\gamma_1, \ldots, \gamma_r}^{k_1, \ldots, k_r}
       (\tmmathbf{\beta}, \tmmathbf{y})])^4 \]
    exists.}
    
    \item For $\varepsilon = 1$, the limit
    \[ \lim_{M \rightarrow \infty} M^2 \mathbb{E}_{\tmmathbf{\beta},
       \tmmathbf{y}} (X_{\gamma_1, \ldots, \gamma_r}^{k_1, \ldots, k_r}
       (\tmmathbf{\beta}, \tmmathbf{y}) -\mathbb{E}_{\tmmathbf{\beta},
       \tmmathbf{y}} [X_{\gamma_1, \ldots, \gamma_r}^{k_1, \ldots, k_r}
       (\tmmathbf{\beta}, \tmmathbf{y})])^2 \]
    exists.
  \end{itemize}
\end{proposition}

\begin{proof}
  We start with the case $\varepsilon = 1 / 2$. An expression for
  $X_{\gamma_1, \ldots, \gamma_r}^{k_1, \ldots, k_r} (\tmmathbf{\beta},
  \tmmathbf{y})$ can be obtained by Lemma \ref{Lemapent}, and we have that it
  is a linear combination of integrals $\mathcal{I} (\{ \alpha_i, \theta_i
  \}_{i = 1}^r)$. To justify that the desired limit exists, we use similar
  arguments as in Theorem \ref{normal1!}, but for the special case where
  $\tmmathbf{\beta}_1, \ldots, \tmmathbf{\beta}_M, \tmmathbf{y}_1, \ldots,
  \tmmathbf{y}_M$ are deterministic, since the expectations in the formula of
  $X_{\gamma_1, \ldots, \gamma_r}^{k_1, \ldots, k_r} (\tmmathbf{\beta},
  \tmmathbf{y})$ are taken only with respect to the measure on dimer
  coverings. Therefore, first we have to determine the higher order term which
  would emerge in the case that $\tmmathbf{\beta}_1, \ldots,
  \tmmathbf{\beta}_M, \tmmathbf{y}_1, \ldots, \tmmathbf{y}_M$ were
  deterministic, and then to average with respect to the distribution of the
  random environment.
  
  The higher order term of $X_{\gamma_1, \ldots, \gamma_r}^{k_1, \ldots, k_r}
  (\tmmathbf{\beta}, \tmmathbf{y})$ will not be the same to the one of Theorem
  \ref{nzgamskl}, since for our current framework
  \begin{equation}
    \mathbb{E}_{\tmmathbf{\beta}, \tmmathbf{y}} \left[ \prod_{n =
    1}^{\tmxspace r} \left( \gamma_{r - n + 1} \sum_{j = 1}^{\infty}
    \frac{\dot{\tmmathbf{Y}}_{j, N_{r - n + 1}}}{z_n^{j + 1}} - (\gamma_{r - n
    + 1} - 1) \sum_{j = 1}^{\infty} \dot{\tmmathbf{X}}_{j, N_{r - n + 1}, M}
    z_n^{j - 1} \right) \right], \label{qaqaqaqz}
  \end{equation}
  is equal to zero. The desired term will be given by
  \begin{equation}
    \sum_{\pi \in P_2 (r)} \prod_{\underset{i < j}{(i, j) \in \pi}}
    \mathcal{G}_{k_i, k_j}^{(N_i, N_j)} . \label{-0-0}
  \end{equation}
  Notice that for the annealed framework of Theorem \ref{normal1!}, we showed
  that there are more terms of higher order than (\ref{-0-0}), except from the
  one that corresponds to (\ref{qaqaqaqz}), but they will also vanish in the
  current case because they depend on $\dot{\tmmathbf{Y}}_{j, N_i}$,
  $\dot{\tmmathbf{X}}_{j, N_i, M}$.
  
  The desired limit can be computed explicitly, following similar arguments as
  in Theorem \ref{normal1!}. More precisely, for the higher order term of
  $\mathcal{G}_{l, k}^{(N_i, N_j)}$, writing $\tmmathbf{X}_{n, N_j, M} =
  \dot{\tmmathbf{X}}_{n, N_j, M} + \bar{\tmmathbf{X}}_{n, N_j, M}$,
  $\tmmathbf{Y}_{n, N_j} = \dot{\tmmathbf{Y}}_{n, N_j} +
  \bar{\tmmathbf{Y}}_{n, N_j}$, and applying the binomial theorem, we
  introduce a collection of integers $(\eta^{(1)} \comma \eta^{(2)},
  \theta^{(1)}, \theta^{(2)})$ that correspond to the stochastic summands, in
  order to absorb $M^2$ using the $2^{- 1}$-CLT appropriateness. This implies
  that the higher order term is an explicit linear combination of double
  contour integrals
  \[ \oint \oint \frac{(1 - w)^l (1 - z)^k}{(z - w)^2} \left( \sum_{n =
     1}^{\infty} \dot{\tmmathbf{X}}_{n, N_j, M} z^{n - 1} \right)^{\eta^{(1)}}
     \left( \sum_{n = 1}^{\infty} \frac{\dot{\tmmathbf{Y}}_{n, N_j}}{z^{n +
     1}} \right)^{\theta^{(1)}} \left( \sum_{n = 1}^{\infty}
     \dot{\tmmathbf{X}}_{n, N_i, M} w^{n - 1} \right)^{\eta^{(2)}} \]
  \[ \left( \sum_{n = 1}^{\infty} \frac{\dot{\tmmathbf{Y}}_{n, N_i}}{w^{n +
     1}} \right)^{\theta^{(2)}} \left( \sum_{n = 1}^{\infty}
     \bar{\tmmathbf{X}}_{n, N_j, M} z^{n - 1} \right)^{k - m - \eta^{(1)}}
     \left( \sum_{n = 0}^{\infty} \frac{\bar{\tmmathbf{Y}}_{n, N_j}}{z^{n +
     1}} \right)^{m - \theta^{(1)}} \]
  \[ \left( \sum_{n = 1}^{\infty} \bar{\tmmathbf{X}}_{n, N_i, M} w^{n - 1}
     \right)^{l - \mu - \eta^{(2)}} \left( \sum_{n = 0}^{\infty}
     \frac{\bar{\tmmathbf{Y}}_{n, N_i}}{w^{n + 1}} \right)^{\mu -
     \theta^{(2)}} d \text{} z \text{} d \text{} w. \]
  For $\pi \in P_2 (r)$, there are $r / 2$ collections of integers
  $(\eta^{(1)}, \theta^{(1)}, \eta^{(2)}, \theta^{(2)})$ that correspond to
  $\prod_{(i, j) \in \pi} \mathcal{G}_{k_i, k_j}^{(N_i, N_j)}$. Justifying as
  in Theorem \ref{normal1!}, the desired limit emerges from
  \[ \mathbb{E}_{\tmmathbf{\beta}, \tmmathbf{y}} \left[ \prod_{\underset{i <
     j}{(i, j) \in \pi_1}} \mathcal{G}_{k_i, k_j}^{(N_i, N_j)} \cdot
     \prod_{\underset{i < j}{(i, j) \in \pi_2}} \mathcal{G}_{k_i, k_j}^{(N_i,
     N_j)} \cdot \prod_{\underset{i < j}{(i, j) \in \pi_3}} \mathcal{G}_{k_i,
     k_j}^{(N_i, N_j)} \cdot \prod_{\underset{i < j}{(i, j) \in \pi_4}}
     \mathcal{G}_{k_i, k_j}^{(N_i, N_j)} \right], \]
  if we consider all its summands, that correspond to the case where for each
  $\prod_{(i, j) \in \pi} \mathcal{G}_{k_i, k_j}^{(N_i, N_j)}$ in the
  expectation of the product, only for one among the $r / 2$ collections of
  integers $(\eta^{(1)}, \theta^{(1)}, \eta^{(2)}, \theta^{(2)})$ we have that
  $\eta^{(1)} + \theta^{(1)} + \eta^{(2)} + \theta^{(2)} = 1$. For the rest,
  we must have that $\eta^{(1)} + \theta^{(1)} + \eta^{(2)} + \theta^{(2)} =
  0$. This procedure will give something convergent, due to the $2^{- 1}$-CLT
  appropriateness. The limit will depend on double contour integrals similar
  to (\ref{poslft}), and other $4$-fold contour integrals that can be computed
  explicitly. All the other terms of $M^2 \mathbb{E}_{\tmmathbf{\beta},
  \tmmathbf{y}} (X_{\gamma_1, \ldots, \gamma_r}^{k_1, \ldots, k_r}
  (\tmmathbf{\beta}, \tmmathbf{y}) -\mathbb{E}_{\tmmathbf{\beta},
  \tmmathbf{y}} [X_{\gamma_1, \ldots, \gamma_r}^{k_1, \ldots, k_r}
  (\tmmathbf{\beta}, \tmmathbf{y})])^4$ either will cancel out or they will
  vanish as $M \rightarrow \infty$. This proves the claim.
  
  For $\varepsilon = 1$, the claim can be proved using exactly the same
  strategy as before. Notice that the limit of the second moment exists due to
  the $1$-CLT appropriateness, which is not the case for $\varepsilon = 1 /
  2$.
\end{proof}

Now we turn to the proof of Theorem \ref{NMNm}.

\begin{proof}[of Theorem \ref{NMNm}]
  We only consider the case $\varepsilon = 1 / 2$, since for $\varepsilon = 1$
  the proof is the same. Let $k_1, \ldots, k_r$ be non-negative integers and
  $N_i = \lfloor \gamma_i M \rfloor$, where $0 < \gamma_1 \leq \cdots \leq
  \gamma_r \leq 1$. Since the higher order term of
  \begin{equation}
    \mathbb{E}_{\tmmathbf{\beta}, \tmmathbf{y}}
    \mathbb{E}_{\tmmathbf{\lambda}} [(p_{k_1}^{(N_1, M)}
    -\mathbb{E}_{\tmmathbf{\lambda}} [p_{k_1}^{(N_1, M)}]) \ldots
    (p_{k_r}^{(N_r, M)} -\mathbb{E}_{\tmmathbf{\lambda}} [p_{k_r}^{(N_r,
    M)}])] \label{tbsprkl}
  \end{equation}
  is
  \[ \mathbb{E}_{\tmmathbf{\beta}, \tmmathbf{y}} \left[ \sum_{\pi \in P_2 (r)}
     \prod_{\underset{i < j}{(i, j) \in \pi}} \mathcal{G}_{k_i, k_j}^{(N_i,
     N_j)} \right], \]
  the CLT appropriateness implies that after normalization, (\ref{tbsprkl})
  converges to $\sum_{\pi \in P_2 (r)} \prod_{(i, j) \in \pi} c_{k_i,
  k_j}^{(\gamma_i, \gamma_j)}$, where $c_{l, k}^{(\gamma_1, \gamma_2)}$ is the
  covariance given in (\ref{poslft}). To conclude the proof, it suffices to
  show that almost surely,
  \[ \lim_{M \rightarrow \infty} \frac{1}{M^{k_1 + \cdots + k_r}}
     \mathbb{E}_{\tmmathbf{\lambda}} [(p_{k_1}^{(N_1, M)}
     -\mathbb{E}_{\tmmathbf{\lambda}} [p_{k_1}^{(N_1, M)}]) \ldots
     (p_{k_r}^{(N_r, M)} -\mathbb{E}_{\tmmathbf{\lambda}} [p_{k_r}^{(N_r,
     M)}])] \text{{\hspace{7em}}} \]
  \begin{equation}
    = \lim_{M \rightarrow \infty} \frac{1}{M^{k_1 + \cdots + k_r}}
    \mathbb{E}_{\tmmathbf{\beta}, \tmmathbf{y}}
    \mathbb{E}_{\tmmathbf{\lambda}} [(p_{k_1}^{(N_1, M)}
    -\mathbb{E}_{\tmmathbf{\lambda}} [p_{k_1}^{(N_1, M)}]) \ldots
    (p_{k_r}^{(N_r, M)} -\mathbb{E}_{\tmmathbf{\lambda}} [p_{k_r}^{(N_r,
    M)}])] \text{.} \label{toslft}
  \end{equation}
  To ensure that, it is sufficient to show that
  \begin{equation}
    \lim_{N \rightarrow \infty} \mathbb{P} \{ \sup_{M \geq N} | X_{\gamma_1,
    \ldots, \gamma_r}^{k_1, \ldots, k_r} (\tmmathbf{\beta}, \tmmathbf{y})
    -\mathbb{E}_{\tmmathbf{\beta}, \tmmathbf{y}} [X_{\gamma_1, \ldots,
    \gamma_r}^{k_1, \ldots, k_r} (\tmmathbf{\beta}, \tmmathbf{y})] | \geq
    \alpha \} = 0, \label{thskts}
  \end{equation}
  for all $\alpha > 0$. By Chebyshev's inequality, we get
  \[ \mathbb{P} \{ \sup_{M \geq N} | X_{\gamma_1, \ldots, \gamma_r}^{k_1,
     \ldots, k_r} (\tmmathbf{\beta}, \tmmathbf{y})
     -\mathbb{E}_{\tmmathbf{\beta}, \tmmathbf{y}} [X_{\gamma_1, \ldots,
     \gamma_r}^{k_1, \ldots, k_r} (\tmmathbf{\beta}, \tmmathbf{y})] | \geq
     \alpha \} \text{} \]
  \[ \leq \sum_{M = N}^{\infty} \mathbb{P} \left\{ | X_{\gamma_1, \ldots,
     \gamma_r}^{k_1, \ldots, k_r} (\tmmathbf{\beta}, \tmmathbf{y})
     -\mathbb{E}_{\tmmathbf{\beta}, \tmmathbf{y}} [X_{\gamma_1, \ldots,
     \gamma_r}^{k_1, \ldots, k_r} (\tmmathbf{\beta}, \tmmathbf{y})] | \geq
     \frac{\alpha}{2} \right\} \]
  \[ \text{} \leq \frac{16}{\alpha^4} \sum_{M = N}^{\infty}
     \mathbb{E}_{\tmmathbf{\beta}, \tmmathbf{y}} (X_{\gamma_1, \ldots,
     \gamma_r}^{k_1, \ldots, k_r} (\tmmathbf{\beta}, \tmmathbf{y})
     -\mathbb{E}_{\tmmathbf{\beta}, \tmmathbf{y}} [X_{\gamma_1, \ldots,
     \gamma_r}^{k_1, \ldots, k_r} (\tmmathbf{\beta}, \tmmathbf{y})])^4 \leq
     \tmop{const} \cdot \sum_{M = N}^{\infty} \frac{1}{M^2}, \]
  where the last inequality is due to Proposition \ref{pfgmsklr}. Since
  $\sum_{M = 1}^{\infty} \frac{1}{M^2} < \infty$, we get that (\ref{thskts})
  holds, and this proves the claim.
\end{proof}

\begin{remark}
  From the covariance structure (\ref{poslft}) we deduce that the quenched
  fluctuations are described by the Gaussian Free Field. Notice that the
  covariance depends only on the limit shape of $\{ \tmmathbf{x}_i,
  \tmmathbf{\beta}_i \}_{i = 1}^M$.
\end{remark}

\

(Panagiotis Zografos) Institute of Mathematics, Leipzig University,
Augustusplatz 10, 04109 Leipzig, Germany

{\tmem{Email address}}: pzografos04@gmail.com

\end{document}